\documentclass[11pt]{article}

\usepackage{latexsym}
\usepackage{epsfig}
\usepackage{amsfonts,amssymb,amsmath,amsthm}
\usepackage{graphicx,multicol,color}
\usepackage{stmaryrd}
\usepackage{float}
\usepackage[utf8]{inputenc}
\usepackage[T1]{fontenc}
\usepackage{caption}
\usepackage{natbib}
\usepackage{hyperref}
\usepackage{dsfont}
\usepackage{enumitem}

\usepackage{bold-extra}
\usepackage{color}                                            
\usepackage{array}                                            
\usepackage{longtable}                                        
\usepackage{calc}                                             
\usepackage{multirow}                                         
\usepackage{hhline}                                           
\usepackage{ifthen}    
\usepackage{xcolor}

\usepackage[margin=2.2cm,footskip=1cm]{geometry}

\usepackage{authblk}
\usepackage{lmodern}
\usepackage{xspace}

\let\oldabstract\abstract
\makeatletter
\renewcommand\abstract{%
  \providecommand\keywords{\par\medskip\noindent\textit{Keywords:}\xspace}
  \oldabstract\noindent\ignorespaces}
\makeatother

\newtheorem{theo}{Theorem}
\newtheorem{defi}{Definition}
\newtheorem{prop}[theo]{Proposition}

\newtheorem{lem}[theo]{Lemma}
\newtheorem{cor}[theo]{Corollary}

\newcommand\R{\mathbb{R}}
\newcommand\N{\mathbb{N}}

\newcommand\E{\mathbb{E}}
\renewcommand\P{\mathbb{P}}

\newcommand\F{\mathcal{F}}

\newcommand{\1}{{\mathbf 1}}

\newcommand{\one}{\ifmmode {\sf 1}\hspace{-.26em}{\sf
l}\hspace{-.35em}{\sf \_} \else ${\sf 1}\hspace{-.26em}{\sf
l}\hspace{-.35em}{\sf \_}$ \fi}

\newcommand{\FF}{\mathcal{F}}
\newcommand{\PP}{\mathcal{P}}
\renewcommand{\AA}{\mathcal{A}}
\newcommand{\BB}{\mathcal{B}}

\newcommand{\DD}{\mathcal{D}}
\newcommand{\EE}{\mathcal{E}}

\newcommand{\MM}{\mathcal{M}}

\renewcommand{\SS}{\mathcal{S}}
\newcommand{\ZZ}{\mathcal{Z}}

\renewcommand{\a}{\alpha}
\renewcommand{\b}{\beta}
\newcommand{\ck}{\check}
\renewcommand{\d}{\delta}
\newcommand{\dd}{\textnormal{d}}

\renewcommand{\t}{\tau}
\renewcommand{\ck}{\check}

\newcommand{\e}{\textnormal{e}}
\newcommand{\ep}{\varepsilon}
\newcommand{\Om}{\Omega}

\newcommand{\s}{\sigma}

\newcommand{\ind}{\1}

\renewcommand{\theequation}{\arabic{equation}}
\renewcommand{\thetheo}{\arabic{theo}}
\renewcommand{\Box}{\mbox{\rule{1ex}{1ex}}}

\newcounter{exno}
\newcommand{\ex}[1]{\refstepcounter{exno}\label{#1}}

\newcounter{rkno}

\newcounter{hypotno}

\graphicspath{{.}{Figures/}{Simus-Data/Figures/}{../Figures/}}

\usepackage{xr}
\title{Feller and ergodic properties of jump-move processes with applications to interacting particle systems}

\author[1]{Fr\'ed\'eric Lavancier}
\author[2]{Ronan Le Gu\'evel}
\author[3]{Emilien Manent}
\affil[1]{LMJL,  2 Rue de la Houssini\`ere, F-44322 Nantes Cedex 03, France, \texttt{frederic.lavancier@univ-nantes.fr}}
\affil[2]{Univ Rennes, CNRS, IRMAR - UMR 6625, F-35000 Rennes, France,
\texttt{ronan.leguevel@univ-rennes2.fr}}
\affil[3]{Univ Rennes, CNRS, IRMAR - UMR 6625, F-35000 Rennes, France,
\texttt{emilien.manent@univ-rennes2.fr}}

\begin{document}

\date{}

\maketitle

\begin{abstract} 

 We consider  Markov processes that alternate continuous motions and jumps in a general locally compact polish space. Starting from a mechanistic construction, a first contribution of this article is to provide conditions on the dynamics so that the associated transition kernel forms a Feller semigroup, and to deduce the  corresponding infinitesimal generator. In a second contribution, we investigate the ergodic properties in the special case where the jumps consist of births and deaths, a situation observed in several applications including epidemiology, ecology and microbiology. Based on a coupling argument, we obtain conditions for the convergence to a stationary measure with a geometric rate of convergence. Throughout the article, we illustrate our results by general examples of systems of  interacting particles in $\R^d$ with births and deaths. 
We show that in some cases the stationary measure can be made explicit and corresponds to a Gibbs measure on a compact subset of $\R^d$. Our examples include in particular Gibbs measure associated to repulsive Lennard-Jones potentials  and to Riesz potentials.


 \end{abstract}


\section{Introduction}

In the spirit of jump-diffusion models, we consider Markov stochastic processes that alternate continuous motions and jumps in  some locally compact Polish space  $E$. We call these general processes jump-move processes. In this paper, the state-space $E$ is typically not a finite-dimensional Euclidean space, contrary to  standard jump-diffusion models. 
Many examples of such dynamics have been considered in the literature, including piecewise deterministic processes \citep{davis1984}, branching particle systems \citep{skorokhod1964,athreya2012}, spatially structured population models  \citep{bansaye2015} and some variations  \citep{cinlar1991,locherbach2002likelihood}, to cite a few. 
A particular  case that  will be of special interest to us is when $E=\cup_{n\geq 0} E_n$ for some disjoint spaces $E_n$, $E_0$ consisting of a single element, and the jumps can only occur from $E_n$ to $E_{n+1}$ (like a birth) or from $E_n$ to $E_{n-1}$ (like a death). We call the latter specific dynamics a birth-death-move process \citep{preston, Lavancier_LeGuevel}.
We will provide several illustrations in the particular case of interacting particles in $\R^d$, with births and deaths.
 These processes are observed in a various range of applications, including microbiology \citep{Lavancier_LeGuevel}, epidemiology \citep{masuda2017} and ecology \citep{renshaw2001, pommerening2019}. 
The main motivation of this contribution is to provide some foundations for the  statistical inference of such processes, by especially studying their ergodic properties.

We start  in Section~\ref{sec2} from a mechanistic general definition of jump-move processes, in the sense that we explicitly  construct the process iteratively over time, which equivalently provides a simulation algorithm. This defines a Markov process $(X_t)_{t\geq 0}$,  the jump intensity function of which reads $\alpha(X_t)$, for some continuous function $\alpha$, and which follows  between its jumps a continuous Markov motion on $E$.
We then derive in Section~\ref{sec:feller} conditions ensuring that the transition kernel of $(X_t)_{t\geq 0}$ forms a Feller semigroup on $C_b(E)$ or on $C_0(E)$. We obtain  the natural result that if $\alpha$ is bounded, then a jump-move process is Feller whenever the transition kernel of the jumps and the transition kernel of the inter-jumps motion (i.e. the move part) are. Similarly, the infinitesimal generator is just the sum of the generator of the jumps and the generator of the move, the domain corresponding under mild conditions to the domain of the generator of the move. 

In Section~\ref{asymptotics}, we focus on birth-death-move processes and we obtain simple conditions on the birth and death intensity functions ensuring their ergodicity with a geometric rate of convergence, in line with standard results for simple birth-death processes on $\N$ \citep{karlin1957} and for spatial birth-death processes (the case without move) established by \cite{preston} and \cite{moller1989}. Following \cite{preston}, the main ingredient to establish such results is a coupling with a simple birth-death process on $\N$, that provides conditions implying that the single element of $E_0$ is an ergodic state for the process. However the inclusion of inter-jumps motions make this coupling more delicate to justify than for the pure spatial birth-death processes of \cite{preston}. We manage to realise this coupling under the assumption that the birth-death-move process is Feller on $C_0(E),$ making necessary the properties discussed before.

We emphasise that the above results are very general in the sense that neither $E$ is specified, nor the exact jump transition kernel, nor the form of the inter-jumps continuous Markov motion. Our only real working assumption is the boundedness of the intensity function $\alpha$. 
We however provide many illustrations 
in the case where $(X_t)_{t\geq 0}$ represents the dynamics of a system of particles in $\R^d$,  introduced in Section~\ref{sec2.3}. 
In this situation, we consider continuous inter-jumps motions driven by deterministic growth-interacting dynamics, as  already exploited in ecology  \citep{renshaw2001, habel2019}, or driven by interacting SDE systems, in particular overdamped Langevin dynamics, the Feller properties of which translate straightforwardly to the move part of $(X_t)_{t\geq 0}$. 
As to the jumps, they are continuous Feller in general, but not necessarily Feller on $C_0(E)$. The picture becomes  however more intelligible when they only consist of births and deaths.  
 We present in Section~\ref{spatial Feller} general birth transition kernels that imply the Feller properties under mild assumptions. On the other hand, a simple uniform death kernel cannot be Feller on $C_0(E)$ in this setting, unless the particles are restricted to a compact subset of $\R^d$, a situation where any death kernel is Feller. 
We finally  show in Section~\ref{invariant Gibbs} that for a system of interacting particles in $\R^d$ with births and deaths, we may obtain an explicit Gibbs distribution for the invariant probability measure. This happens when the inter-jumps motion is driven by a Langevin dynamics based on some potential function $V$, and the jumps characteristics depend in a proper way on the same potential $V$. Our assumption on $V$  include in particular  Riesz potentials, repulsive Lennard-Jones potentials, soft-core potentials and (regularised) Strauss potentials, that are standard models used in spatial statistics and statistical mechanics. 

We have gathered in Appendix~\ref{sec6.7} the details of the  coupling used in Section~\ref{asymptotics} to get the ergodic properties, while most proofs are postponed to a supplementary material along with additional results.

\section{Jump-move processes} \label{sec2}

\subsection{Iterative construction}\label{jump-move general}

Let $E$ be a Polish space equipped with the Borel $\sigma-$algebra $\EE$ and a distance $d.$ Let $(\Om,\FF)$ be a measurable space and $(\P_x)_{x \in E}$ a family of probability measures on $(\Om,\FF).$ 
In order to define a jump-move process  $(X_t)_{t\geq 0}$ on $E$, we need three ingredients:
\begin{enumerate}
\item An intensity function $\a:E \rightarrow \R_+$ that governs the inter-jumps waiting times. 
\item A transition kernel $K$ for the jumps, defined on $E \times \EE.$
\item A continuous homogeneous Markov process $((Y_t)_{t \geq 0}, (\P_x)_{x \in E})$ on $E$, the distribution of which will drive the inter-jumps motion of $(X_t)_{t\geq 0}$. 
 \end{enumerate}
We will work throughout this paper under the assumption that $\a:E \rightarrow \R_+ $ is continuous and bounded by  $\a^*>0$, i.e. for all $x\in E$
\begin{equation}\label{assumption alpha}
0 \leq \a(x) \leq \a^*.
\end{equation}

We denote by $(Q_t^Y)_{t \geq 0}$ the transition kernel of $(Y_t)_{t\geq 0}$, given by
$$Q_t^Y(x,A)=\P_x(Y_t \in A), \quad x \in E, A \in \EE.$$

The following iterative construction provides a clear intuition of the dynamics of the process $(X_t)_{t\geq 0}$. It follows closely the presentation in the supplementary material of \cite{Lavancier_LeGuevel}, where an algorithm of simulation on a finite time  interval is  also derived.

Let $(Y_t^{(j)})_{t \geq 0}$, $j\geq 0$, be  a sequence of processes on $E$ identically distributed as $(Y_t)_{t \geq 0}$. Set $T_0=0$ and let $x_0\in E$. Then $(X_t)_{t\geq 0}$ can be constructed as follows:  for $j\geq 0$, iteratively do
\begin{enumerate}[label=(\roman*)]
            \item Given $ X_{T_j}=x_j$, generate $(Y_t^{(j)})_{t \geq 0}$ conditional on $ Y_0^{(j)}=x_j$ according to the kernel $(Q_t^{Y}(x_j,.))_{t \geq 0}$.
            \item Given $ X_{T_j}=x_j$ and $(Y_t^{(j)})_{t \geq 0}$, generate $ \t_{j+1}$ according to the cumulative distribution function 
                  \begin{equation}\label{cdf tau}
                    F_{j+1}(t)=1-\exp \left ( - \int_0^t  \a \left ( Y_u^{(j)} \right ) \, \dd u  \right ).
                  \end{equation}
            \item Given $ X_{T_j}=x_j$, $( Y_t^{(j)})_{t \geq 0}$ and $ \t_{j+1}$, generate $x_{j+1}$ according to the transition kernel $ K (Y_{ \t_{j+1}}^{(j)}, .)$.
            \item Set $ T_{j+1}=T_j+\t_{j+1}$, $ X_t=( Y_{t-T_j}^{(j)})$ for $t \in [T_j,T_{j+1})$ and $ X_{T_{j+1}}= x_{j+1}$.
\end{enumerate}

We denote by $(\FF_t^Y)_{t \geq 0}$ the natural filtration of $(Y_t)_{t\geq 0}$, i.e.  $\FF_t^Y= \sigma ( Y_u, u \leq t)$,  and by $(\FF_t)_{t >0}$ the natural filtration of $(X_t)_{t \geq 0}$. We make these filtrations complete \citep[Section 20.1]{Bass} and abusively use the same notation. 
The jump-move process  $((X_t)_{t \geq 0}, (\P_x)_{x \in E})$ constructed above is a homogeneous Markov process with respect to $(\F_t)_{t>0}$.
The trajectories of $(X_t)_{t \geq 0}$ are continuous except at the jump times $(T_j)_{j \geq 1}$ where they are right-continuous. The specific form \eqref{cdf tau} implies that the law of the waiting time $\tau_j$ under $\P_x$ is absolutely continuous with respect to the Lebesgue measure with density 
$\E_x \left [  \a(Y_t^{(j-1)}) \textnormal{exp} \left ( - \int_0^t \a(Y_u^{(j-1)})\, \dd u \right )  \right ] \1_{t > 0}.$
It also implies that the intensity of jumps is $\alpha(X_{t})$. 
Denote by $N_t=\sum_{j\geq 0} \1_{T_j\leq t}$ the number of jumps before $t\geq 0$. Under the assumption~\eqref{assumption alpha}, we have for any $n \geq 0$ and $t \geq 0$,
\begin{equation} \label{dominationpoisson}
\P(N_t > n) \leq \P(N^*_t >n)
\end{equation}
where $N^*_t$ follows a Poisson distribution with rate $\a^*t$. This in particular implies that  $(N_t)_{t \geq 0}$ is a non-explosive counting process. All the aforementionned properties of $(X_t)_{t \geq 0}$ are either immediate or verified in \cite{Lavancier_LeGuevel}.

Note that the above construction only implies the weak Markov property of $(X_t)_{t\geq 0}$ in general, at least because the process $(Y_t)_{t\geq 0}$ is only assumed to be a (weak) Markov process. A more abstract construction obtained by ``piecing out'' strong Markov processes is introduced in \cite{ikeda1968}, leading to a strong Markov jump-move process. The strong Markov property can also be obtained  in our case by strengthening the assumptions, see Section~\ref{sec3.1}.

The transition kernel of $(X_t)_{t \geq 0}$ will be denoted, for any $t \geq 0$, $x \in E$ and $A \in \EE$, by
$$Q_t(x,A) = \P(X_t \in A | X_0=x)=\P_x(X_t \in A).$$
Also for $f \in M^b(E)$, where $M^b(E)$ is the set of real valued bounded and measurable functions on $E$, we will denote $Q_tf(x)=\E_x [ f(X_t)]=\int_E Q_t(x, \dd y) f(y)$. Similarly we will write  $Q_t^Yf(x)=\E_x^Y(f(Y_t))$.

\subsection{Special case of the birth-death-move process} \label{sec2.2}

A birth-death-move process is the particular case  of a jump-move process where $E$ takes the form 
$E=  \bigcup_{n=0}^{\infty} E_n$, with $(E_n)_{n \geq 0}$ a sequence of disjoint Polish spaces, and where the jumps are only births and deaths. 
We assume that each $E_n$ is equipped with the Borel $\sigma-$algebra $\EE_n$, so that $E$ is associated with the $ \displaystyle \sigma-$field $\EE=\sigma \left ( \bigcup_{n=0}^{\infty} \EE_n \right ) $. We further assume that $E_0$ consists of a single element denoted by $\textnormal{\O}$.
In this setting,  the Markov process $(Y_t)_{t \geq 0}$ driving the motions of $(X_t)_{t \geq 0}$ is supposed to satisfy 
\[ \P_x( (Y_t)_{t \geq 0} \subset E_n)= \ind_{E_n}(x), \quad \forall x \in E, \, \forall n \geq 0.  \]

We introduce a birth intensity function $\beta : E \rightarrow \R_+$ and a death intensity function $\d : E \rightarrow \R_+$, both assumed to be continuous on $E$ and satisfying $\a = \b + \d.$ We prevent a death in $E_0$ by assuming that $\d(\textnormal{\O})=0.$ 
The probability transition kernel $K$ for the jumps then reads, for any $x \in E$ and $A \in \EE$,
    \begin{equation}\label{K birth-death}
     K(x,A)=\frac{\b(x)}{\a(x)}K_{\b}(x,A)+ \frac{\d(x)}{\a(x)}K_{\d}(x,A),
     \end{equation}
    where $K_{\b} : E \times \EE \rightarrow [0,1]$ is a probability transition kernel for a birth and $K_{\d} : E \times \EE \rightarrow [0,1]$ is a probability transition kernel for a death. They satisfy, for $x \in E$ and $n \geq 0$,
    \[  K_{\b}(x,E_{n+1})=\ind_{E_n}(x) \quad \text{and} \quad   K_{\d}(x,E_{n})=\ind_{E_{n+1}}(x). \]
 
 Notice that a simple birth-death process is the particular case where $E=\N$, $E_n= \{ n \}$ and the intensity functions $\b$ and $\d$ are sequences.

For later purposes, when $E=  \bigcup_{n=0}^{\infty} E_n$ as in the present section, we define the function $n(.) : E \rightarrow \N$ by  $n(x)=k$ when $x \in E_k$, so that $x \in E_{n(x)}$ is always satisfied.

\subsection{Kolmogorov backward equation} \label{sec2.4}

The goal of this section is to present the Kolmogorov backward equation for the transition kernel of  the general jump-move process $(X_t)_{t \geq 0}$ of Section~\ref{jump-move general}, providing a more probabilistic viewpoint of its dynamics, and to show that the solution exists and is unique. To obtain these results we use similar methods as in \cite{feller1971} for pure jump processes, see also \cite{preston}. The key assumption is the boundedness \eqref{assumption alpha} of the intensity $\alpha$, which prevents the explosion of the process. The proofs are postponed to Section~\ref{proof KBE} in the supplementary material.

\begin{theo} \label{TheoKBEi}
For all $x\in E$ and all $A \in \EE$, the function $t \mapsto Q_t(x,A)$, for $t>0$, satisfies the  following Kolmogorov backward equation:
\begin{align} \label{kbei}
    Q_t(x,A) & = \E_x^Y \left [ \ind_{Y_t \in A} \,\e^{ - \int_0^t \a(Y_u) \, \dd u  }   \right ]  + \int_0^t  \int_E  Q_{t-s}(y,A)    \E_x^Y \left [ K \left  (Y_s,dy \right ) \,    \a(Y_s) \e^{- \int_0^s \a(Y_u) \, \dd u}    \right ] ds.
\end{align}
\end{theo}
In the case of the birth-death-move process of Section~\ref{sec2.2}, the above equation reads, for $x \in E_n$,
\begin{multline}\label{kbebdm}
    Q_t(x,A)  = \E_x^Y \left [ \ind_{Y_t \in A} \, \e^{ - \int_0^t \a(Y_u) \, \dd u }   \right ] + \int_0^t   \int_{E_{n+1}} Q_{t-s}(y,A) \,  \E_x^Y \left [ \b \left (Y_s \right ) K_\b \left ( Y_s  ,dy \right )  \, \e^{ - \int_0^s \a(Y_u) \, \dd u } \right ] \dd s \\
     + \int_0^t   \int_{E_{n-1}} Q_{t-s}(y,A) \,  \E_x^Y \left [ \d \left (Y_s \right ) K_\d \left ( Y_s  ,dy \right )  \, \e^{ - \int_0^s \a(Y_u) \, \dd u } \right ] \dd s.
\end{multline}

To show the existence of a unique solution to \eqref{kbei}, let $Q_{t,p}(x,A):=\P_x(X_t \in A , T_p >t)$ be the transition probability from state $x$ to $A$ in time $t$ with at most $p$ jumps.
Notice that we can define $Q_{t,\infty}= \lim \limits_{p \to \infty} Q_{t,p}$ because $Q_{t,p} \leq Q_{t,p+1} \leq 1.$ We prove in the following proposition that $Q_{t , \infty}$ is the unique solution to \eqref{kbei}  using a minimality argument as in \cite{feller1971}.
\begin{prop} \label{existuniciteKBE}
    $Q_{t,\infty}$ is the unique sub-stochastic solution of \eqref{kbei}, i.e. it is the unique solution satisfying $Q_{t}(x,E) \leq 1$ for all $x\in E$. Moreover  $Q_{t,\infty}$ is stochastic, \textit{i.e.} $Q_{t,\infty}(x,E)=1$ for all $x\in E$.
\end{prop}

To conclude this section, we present an interpretation of $Q_{t,\infty}$ for the birth-death-move process of Section~\ref{sec2.2}, which is much in the spirit of \cite{preston}.
We write $Q_{t,(p)}(x,A)$ for the transition probability from $x$ to $A$ in time $t$ without having entered $ \bigcup_{k=p+1}^{\infty} E_k,$ that is
\[ Q_{t,(p)}(x,A)=\P_x\left(X_t \in A,\,  \forall s\in[0,t] \ n(X_s)\leq p  \right). \]
We can also define  $Q_{t,(\infty)}(x,A) = \lim \limits_{p \rightarrow \infty} \ Q_{t,(p)}(x,A) \leq 1$ by monotonicity. 
\begin{prop}\label{prop13mini}
For all $x\in E$ and all $A \in \EE$, $Q_{t,(\infty)}(x,A)=Q_{t,\infty}(x,A).$
\end{prop}

\subsection{Systems of interacting particles in $\R^d$} \label{sec2.3}

In this section, we focus on the dynamics of a system of interacting particles in $\R^d$ and we provide general examples of birth kernels, death kernels and inter-jumps motions in this setting, that to our opinion constitute realistic models for applications and are actually already used in some domains. Some of them moreover lead to an explicit Gibbs  stationary measure of the dynamics, as we will show in Section~\ref{invariant Gibbs}.  
These running examples will serve in the rest of the paper to illustrate the theoretical results and make explicit our assumptions. 

\medskip

Let $W \subset \R^d$ be a closed set where the particles live, equipped with a $\s-$field $\BB$.  A collection of $n$ particles in $W$ is a point configuration for which the ordering does not matter. For this reason, for $n \geq 1,$ we will identify two elements $(x_1, \dots,x_n)$ and $(y_1,\dots,y_n)$ of $W^n$  if there exists a permutation $\s$ of $\{1,\dots,n\}$ such that $x_i=y_{\s(i)}$ for any $1 \leq i \leq n$. Following \cite{preston}, \cite{locherbach2002likelihood} and others, we thus define $E_n$ as the space obtained by this identification.  Specifically, denoting by $\pi_n : (x_1,\dots,x_n) \in W^n \mapsto \{ x_1,\dots,x_n \}$ the associated projection,  the space $E_n$ corresponds for $n\geq 1$ to $E_n = \pi_n( W^n) $ equipped with the $\s-$field  $\EE_n=\pi_n (\BB^{\otimes n})$, while $E_0=\{\text{\O}\}$ is just composed of the empty configuration. The general state space  of a system of particles is  then $E = \cup_{n \geq 0} E_n$ equipped with the $\sigma-$field $ \displaystyle \EE = \sigma \left ( \cup_{n \geq 0} \EE_n \right ).$
This formalism  allows us to go back and forth quite straightforwardly between the space $E_n$ and the space $W^n$, the latter being in particular more usual to define the inter-jumps motion of  $n$ particles, as detailed below. Note that an alternative formalism consists in viewing a configuration of particles as a  finite point measure in $W$, in which case $E$ becomes the set of finite point measures in $W$, see for instance \cite{kallenberg}. We choose in this paper to adopt the former point of view. 
We denote by $\|.\|$ the Euclidean norm on $\R^d$. If $x=\{x_1,\dots,x_n\} \in E_n$ and $\xi \in W$,   $x \cup \xi$ stands for $\{x_1,\dots,x_n,\xi \} \in E_{n+1}$ and if $1 \leq i \leq n$, we write $x \, \backslash \, x_i$ for $\{ x_1, \dots, x_{i-1},x_{i+1},\dots, x_n \} \in E_{n-1}.$

As long as we are concerned by continuous inter-jumps motions, we need to equip $E$ with a distance. 
Following  \cite{schuhmacher2008}, we consider the distance $d_1$ defined for $x=\{x_1,\dots,x_{n(x)} \}$ and $y=\{y_1,\dots,y_{n(y)} \}$ in $E$ such that $n(x) \leq n(y)$ by 
\begin{equation} \label{defd1} 
     d_1(x,y) = \frac{1}{n(y)} \left ( \min_{\s \in \SS_{n(y)}} \sum_{i=1}^{n(x)} (\|x_i-y_{\s(i)}\| \wedge 1) + (n(y)-n(x)) \right ),
    \end{equation}
with $d_1(x,\textnormal{\O})=1$ and where $\SS_n$ denotes the set of permutations of $\{1,\dots,n\}$. \cite{schuhmacher2008} and Section~\ref{sec6.2} in the supplementary material detail some topological properties of $(E,d_1)$. 
For the purposes of this section, let us quote in particular that $n(.) : (E,d_1) \rightarrow (\N,|.|)$ is continuous and that $\pi_n$ is continuous. 
Note that other distances than $d_1$ could have been chosen, provided these two last properties (at least) are preserved. Incidentally, the Hausdorff distance, which is a common choice of distance between random sets, does not satisfy these properties (see Section~\ref{sec6.2}) and is not appropriate in our setting.

\medskip

We now show how we can easily construct a continuous Markov process  $(Y_t)_{t \geq 0}$ on $E$  from continuous Markov processes on $W^n$ for any $n \geq 1$. We focus on the case where for any $x \in E$ and $n\geq 0$, $\P_x( (Y_t)_{t \geq 0} \subset E_n)= \ind_{E_n}(x)$, as we required it for birth-death-move processes in Section~\ref{sec2.2}. It is then enough to define a process $Y^{|n}$ on each $E_n$. To do so consider a continuous Markov process $(Z_t^{|n})_{t \geq 0}$ on $W^n$ 
whose distribution is permutation equivariant with respect to its initial value $Z^{ \, |n}_0$. This means that for any permutation $\s \in \SS_n$, the law of 
$Z^{ \, |n}_t = ( Z^{ \, |n}_{t,1}, \dots , Z^{ \, |n}_{t,n} )$ given $Z^{ \, |n}_0=(z_{\sigma(1)},\dots,z_{\sigma(n)})$ is the same as the law of  $(Z^{ \, |n}_{t,\s(1)}, \dots ,Z^{ \, |n}_{t,\s(n)} )$ given $Z^{ \, |n}_0=(z_1,\dots,z_n)$.
Let $x= \{x_1,\dots,x_n \} \in E_n$ and take the process $Z^{ \, |n}_{t}$ with initial state $Z^{ \, |n}_0=(x_1,\dots,x_n)$. Note that from the previous permutation equivariance property the choice of ordering for the coordinates of this initial state does not matter, as it will become clear below. We finally define the process $Y^{|n}_t$ on $E_n$ starting from $x$ as 
\begin{equation}\label{defmovesurEn}
   Y^{|n}_t = \pi_n \left ( Z^{ \, |n}_t \right ) = \left \{ Z^{ \, |n}_{t,1}, \dots, Z^{ \, |n}_{t,n} \right \}. 
\end{equation}
Note that the continuity of $t\to Y^{|n}_t$ (with respect to $d_1$) follows from the continuity of $t\to Z^{ \, |n}_t$ and the continuity of $\pi_n$. The continuity of $t\to Y_t$ is then implied by the continuity of $n(.)$.

With this construction, the transition kernel of $Y$ reads, for any $f \in M^b(E)$,
\[ Q_t^Yf(x) = \sum_{n \geq 0}\E \left [ f(Y^{|n}_t) \, |Y^{|n}_0=x\right]\ind_{x \in E_n}  =  \sum_{n \geq 0}\E \left ( f(\pi_n(Z^{ \, |n}_t))  \left | \right. Z^{ \, |n}_0=(x_{1},\dots,x_{n}) \right )\ind_{x \in E_n}, \]
so that denoting by  $Q_t^{Z^{ \, |n}}$ the transition kernel of $Z^{ \, |n}$ in $W^n$ we have
\[ Q_t^Y f(x) = \sum_{n \geq 0} Q_t^{Z^{ \, |n}}(f \circ \pi_n)((x_1,\dots, x_n))\ind_{x \in E_n}.\]
Note that if we had chosen another ordering for the initial state, that is $Z^{ \, |n}_0=(x_{\sigma(1)},\dots,x_{\sigma(n)})$ for some  $\s \in \SS_n$, then the transition kernel of $Y$ would have remained the same, since   by permutation equivariance 
\[ \E \left ( f(\pi_n(Z^{ \, |n}_t))  \left | \right. Z^{ \, |n}_0=(x_{\sigma(1)},\dots,x_{\sigma(n)}) \right ) = 
\E \left ( f(\pi_n(Z^{ \, |n}_{t,\s(1)}, \dots ,Z^{ \, |n}_{t,\s(n)} ))  \left | \right. Z^{ \, |n}_0=(x_{1},\dots,x_{n}) \right )
\]
which is $\E \left ( f(\pi_n(Z^{ \, |n}_t))  \left | \right. Z^{ \, |n}_0=(x_{1},\dots,x_{n}) \right )$.

\medskip

We are now in position to present general examples of jump transition kernels and inter-jumps motions for a system of particles in $W$. The first example introduces a death transition kernel where an existing particle dies with a probability that may depend on the distance to the other particles. The next two examples focus on birth transition kernels, either driven by a mixture of densities around each particle, or by a Gibbs potential. The last two examples apply the above construction of  $(Y_t)_{t \geq 0}$ on $E$ to introduce inter-jumps Langevin  diffusions and growth interaction processes.

\bigskip

\ex{noymort} \noindent \textbf{Example \ref{noymort}}\textit{(death kernel)}:   Let $g : \R_+ \rightarrow \R_+^*$ be a continuous function and for $x=\{x_1,\dots,x_n\}\in E_n$, set $w(x_1,x)=1$ if $n=1$ and if $n\geq 2$, for any $i\in\{1,\dots,n\}$,
\[ w(x_i,x)=\frac{1}{z(x)} \displaystyle \sum_{k \neq i} g \left ( \|x_k-x_i\| \right ),\] 
with $z(x)=\sum_{i=1}^n \sum_{k \neq i} g ( \|x_k-x_i\| )$. A general example of death transition kernel is
\[\displaystyle K_\d(x,A)=  \sum_{i=1}^{n(x)} w(x_i,x) \ind_A(x \, \backslash \, x_i),\quad x\in E,\ A \in \EE.\]
The probability  $w(x_i,x)$ that  $x_i$ disappears then depends on the distance between $x_i$ and the other particles in $x$ through $g$. Uniform deaths correspond to the particular case $w(x_i,x)=1/n(x)$.

\bigskip

\ex{noynaissgauss} \noindent \textbf{Example \ref{noynaissgauss}}\textit{(birth kernel as a mixture)}: Let $\varphi$ be a density function on $W$, and $\phi_1: W\to \R$ and $\phi_2:\R_+\to\R$ be two continuous  functions. 
We set for 
$x=\{x_1,\dots,x_{n(x)}\}\in  E\setminus E_0$, $ v(x_i,x) = \exp\left(\phi_1(x_i)+\sum_{k\neq i} \phi_2 \left ( \|x_k-x_i\| \right )\right)$ and we consider the birth kernel defined for $\Lambda \subset W$  and $x \in E\setminus E_0$ by  $K_\b(\text{\O},\Lambda)=\int_\Lambda \varphi(\xi)  \dd \xi$ and
\[ K_\b (x,\Lambda \cup x) = \dfrac{1}{n(x)}  \sum_{i=1}^{n(x)} \dfrac{1}{z(x_i,x)} \int_\Lambda \varphi \left ( \frac{\xi-x_i}{v(x_i,x)} \right ) \, \dd \xi,\]
where $\Lambda \cup x =\{\{u\} \cup x, u\in \Lambda\}$ and  $z(x_i,x)=\int_W \varphi \left ( (\xi-x_i)/v(x_i,x)\right )\dd \xi$. Note that $z(x_i,x)=v(x_i,x)^d$ if $W=\R^d$. 
It is easily checked that $K_\b (x,E_{n+1})=K_\b ( x, W \cup x)= 1$ for $x \in E_n$, and in particular this kernel is a genuine birth kernel in the sense that the transition from $E_n$ to $E_{n+1}$ is only due to the addition of a new particle, the existing ones remaining unchanged. Moreover, the new particle is distributed as a  mixture of distributions driven by $\varphi$, each of them being centred at the existing particles. The term $v(x_i,x)$ quantifies the dispersion of births around the particle $x_i$ and it depends on the distance between $x_i$ and the other particles through $\phi_2$. A natural example is a mixture of isotropic Gaussian distributions on $\R^d$ (restricted to $W$), respectively centred at $x_i$ with standard deviation $v(x_i,x)$.

\bigskip

\ex{noynaissgibbs} \noindent \textbf{Example \ref{noynaissgibbs}}\textit{(birth kernel based on a Gibbs potential)}:  We introduce a measurable function $V: E \rightarrow \R$, so called a potential,  satisfying $ z(x):=\int_W \exp(-(V(x \cup \xi)-V(x))) \, \dd \xi < \infty$ for all $x\in E$, and we
consider the birth kernel defined for $\Lambda \subset W$ and  $x \in E$  by
\[
K_\b(x,\Lambda \cup x) = \frac{1}{z(x)} \int_\Lambda \e^{-(V(x \cup \xi)-V(x))} \, \dd \xi.
\]
Note that $K_\b ( x, W \cup x)= 1$ for $x \in E$. 
With this kernel, given a configuration $x$, a new particle is more likely to appear in the vicinity of points $\xi\in W$ that make $V(x \cup \xi)-V(x)$ minimal.
This kind of kernels $K_\b$ has been introduced in \cite{preston} for spatial birth-death processes, the case of a birth-death-move process with no move. Their importance is due to the fact that the invariant measure of a spatial birth-death process associated to $K_\b$, uniform deaths and specific birth and death intensities has been explicitly obtained in \cite{preston} and corresponds to the Gibbs measure with potential $V$. This result is at the basis of perfect simulation of  spatial Gibbs point process models, see \cite{moeller:waagepetersen:03}. 
We will similarly show in Section~\ref{invariant Gibbs} that the same Gibbs measure is also invariant for a birth-death-move process associated to the same characteristics for the jumps and a well chosen inter-jumps move process $(Y_t)_{t \geq 0}$ constructed as in the next example.

\bigskip

\ex{moveito} \noindent \textbf{Example \ref{moveito}}\textit{(Langevin diffusions as inter-jumps motions)}:  Let $g : \R^d \rightarrow \R^d$ be a globally Lipschitz continuous function, $\b >0$ and $\{ B_{t,i} \} _{1 \leq i \leq n}$, $n\geq 1$, a collection of $n$ independent Brownian motions on $\R^d$. We start from the following system of SDEs usually called overdamped Langevin equations
\[ \dd Z_{t,i}^{\, |n} = - \sum_{j \neq i} g (  Z_{t,i}^{\, |n}- Z_{t,j}^{\, |n}) \, \dd t + \sqrt{2 \beta^{-1}} \, \dd B_{t,i}, 
\quad 1 \leq i \leq n.\]
For $z=(z_1,\dots,z_n) \in (\R^d)^n$, denoting by $\Phi_n: (\R^d)^n \rightarrow (\R^d)^n$ the function defined by $\Phi_n(z) = (\Phi_{n,1}(z),\dots,\Phi_{n,n}(z))$ with $\Phi_{n,i}(z)= \sum_{j \neq i} g (  z_{i}- z_{j})$,
this system of SDEs reads
\begin{equation} \label{sdeito}
    \dd Z_t^{\, |n} = - \Phi_n(Z_t^{\, |n}) \, \dd t + \sqrt{2 \beta^{-1}} \, \dd B_t^{\, |n},
\end{equation}
where $B_t^{\, |n}=(B_{t,1}, \dots, B_{t,n})$. 
Since $\Phi_n$ is a permutation equivariant function, that is for any  $\s \in \SS_n$
\[\Phi_n(z_{\sigma(1)},\dots,z_{\sigma(n)})=(\Phi_{n,\sigma(1)}(z),\dots,\Phi_{n,\sigma(n)}(z)),\] 
and  since $B_t^{\, |n}$ is exchangeable,  
we can verify by writing \eqref{sdeito} in integral form that the law of $Z_t^{|n}$ is permutation equivariant with respect to its initial state. 
So when $W=\R^d$, we can define each inter-jumps process $Y^{\, |n}$ in $E_n$ from $Z^{\, |n}$ as in \eqref{defmovesurEn}, yielding  $(Y_t)_{t \geq 0}$ on $E$. The same construction can be generalised if $W\subsetneq \R^d$ by considering a Langevin equation with reflecting boundary conditions \citep{fattler2007}. This inter-jumps dynamics, associated with the birth kernel of Example~\ref{noynaissgibbs}
and a drift function $g$ related to the potential $V$, converges to a Gibbs measure on $W$ with potential $V$ (see Section~\ref{invariant Gibbs}).

\bigskip

\ex{growthprocess} \noindent \textbf{Example \ref{growthprocess}}\textit{(Growth interaction processes)}:  
This example is motivated by models used in ecology \citep{renshaw2001, renshaw2009, comas2009, habel2019}. Each particle consists of a plant located in $S\subset\R^d$ and  associated with a positive mark, that typically represents the size of the plant, so that $W=S\times \R^+$ here. Births and deaths of plants occur according to a spatial birth and death process, while a deterministic growth applies to their mark. Specifically, when a plant appears, its mark is set to zero or generated according to a uniform distribution on $[0,\ep]$ for some $\ep >0$ (\cite{renshaw2001}). Then the mark increases over time in interaction with the other marks. 
In order to formally define this inter-jumps dynamics, let us denote by $(U_i(t),m_i(t))_{t \geq 0}$, for $i = 1, \dots , n$, each component of the process $(Z_t^{|n})_{t \geq 0}$ where $U_i(t) \in S$  and  $m_i(t)>0$, so that  $Z_t^{|n} \in W^n$. We introduce the system
\begin{align}\label{Z growth}
    & \dfrac{\dd Z^{\, |n}_t}{\dd t} = \left  ((0,F_{1,n}(Z^{\, |n}_t)),\dots,(0,F_{n,n}( Z^{\, |n}_t)) \right ), 
\end{align}
where for all $1\leq i\leq n$,  $F_{i,n}$ is a function from $W^n$ into  $\R_+$.
We thus have 
  $U_i(t) = U_i(0)$ for all $i$ and the evolution of the marks $(m_1(t), \dots ,m_n(t))$ is driven by a deterministic differential equation depending on $(U_1(0),\dots,U_n(0))$ as expected. To define $Y^{\, |n}$ by \eqref{defmovesurEn}, we finally assume permutation equivariance, namely that $F_{\sigma(i),n}(z_1,\dots,z_n)=F_{i,n}(z_{\sigma(1)},\dots,z_{\sigma(n)})$ for all $i$ and all  $\s \in \SS_n$, which is satisfied for all examples in the aforementioned references.

\section{Feller properties and infinitesimal generator}\label{sec:feller}

\subsection{Feller properties} \label{sec3.1}

We assume henceforth that $E$ is a locally compact Polish space. Let $C_b(E)$ be the set of continuous and bounded functions on $E$ and  $C_0(E)$ be the set of continuous functions that vanish at infinity in the sense that for all $ \epsilon>0$, there exists a compact set $B\in E$ such that $x\notin B \Rightarrow |f(x)|<\epsilon$. 

Following \cite{dynkin1965markov} and \cite{oksendal}, we say that the jump-move process $(X_t)_{t \geq 0}$ on $E$ with transition kernel $Q_t$ is  Feller continuous if $Q_t C_b(E) \subset C_b(E)$, and we say that it is Feller if both $  \lim_{t \rightarrow 0} \| Q_tf - f \|_{\infty}= 0$ for any $f \in C_0(E)$ (strong continuity) and $Q_t C_0(E) \subset  C_0(E)$.

 The following proposition, proved in Section~\ref{proofs Feller} of the supplementary material, provides information on  the continuity property of $Q_t$ when $t$ goes to $0$.
\begin{prop}\label{prop24}
We have
    \begin{enumerate} 
        \item  For any $f \in C_b(E)$ and any $x \in E$,  $\lim_{t \rightarrow 0}\limits Q_tf(x) = f(x)$.
        \item Let $ f \in M^b(E)$. Then $ \lim_{t \rightarrow 0}\limits \|Q_tf-f\|_{\infty} = 0 $ iff $ \, \lim_{t \rightarrow 0}\limits \|Q_t^Yf-f\|_{\infty} = 0 .  $
    \end{enumerate}
\end{prop}
By the second point above, the strong continuity of $Q_t$ is implied by the strong continuity of $Q_t^Y$, which in turn holds automatically true if $Q_t^Y \, C_0(E) \subset C_0(E)$  by continuity of $Y_t$. We thus obtain the following natural conditions on the jump-move process on $E$ to be Feller continuous or Feller. The proof is given in Section~\ref{proofs Feller} of the supplementary material.
\begin{theo} \label{conditionsfeller}
Let $(X_t)_{t \geq 0}$ be a general jump-move process on $E$.
\begin{enumerate}
    \item If $Q_t^Y \, C_b(E) \subset C_b(E)$ and $K \, C_b(E) \subset C_b(E)$ then $(X_t)_{t \geq 0}$ is a Feller continuous process.
    \item If $Q_t^Y \, C_0(E) \subset C_0(E)$ and $K \, C_0(E) \subset C_0(E)$ then $(X_t)_{t \geq 0}$ is a Feller process.
\end{enumerate}
\end{theo}

We deduce in particular from this theorem that 
if $Q_t^Y \, C_b(E) \subset C_b(E)$ and $K \, C_b(E) \subset C_b(E)$ (or alternatively with $C_0(E)$ instead of  $C_b(E)$), then $(X_t)_{t \geq 0}$  is a strong Markov process for the filtration $(\FF_t)_{t \geq 0}$, a property implied by the Feller continuous and Feller properties (\cite{Bass}). 
The Feller property will also be useful to us in Section~\ref{asymptotics} to construct a coupling between a birth-death-move process and a simple birth-death process on $\N$, in a view to establish ergodic properties. 

 We investigate in Section~\ref{spatial Feller} the conditions of Theorem~\ref{conditionsfeller} for the examples of dynamics of systems of  interacting particles in $\R^d$ introduced  in Section~\ref{sec2.3}. They turn out to be generally satisfied under mild conditions for these examples.

\subsection{Infinitesimal generator} \label{sec3.3}

In this section we compute the infinitesimal generator associated to the jump-move process $(X_t)_{t \geq 0}$. We first introduce some notations and recall below the definition of the generator, see for instance  \cite{dynkin1965markov}.
In connection, remember that the family $(Q_t)_{t \geq 0}$ of transition operators is a semigroup on $(M_b(E),\|.\|_{\infty}).$ If moreover the process $(X_t)_{t \geq 0}$ is Feller continuous (\textit{resp.} Feller), then $(Q_t)_{t \geq 0}$ is a semigroup on $(C_b(E),\|.\|_{\infty})$ (\textit{resp.} $(C_0(E),\|.\|_{\infty})$). 
\begin{defi} \label{defgenerator}
    Let $L \subset M^b(E)$ and $(U_t)_{t \geq 0}$ be a semigroup on $(L,\|.\|_{\infty})$. We set 
\begin{equation} \label{defL0}
L_0 = \{f \in L : \lim_{t \rightarrow 0}\limits \|U_t f-f\|_{\infty}=0\}
\quad\text{and}\quad
\DD_{\AA} = \{f \in L : \lim_{t \rightarrow 0}\limits \frac{U_t f-f}{t} \; \textnormal{exists in } (L,\|.\|_{\infty})\}.
\end{equation}
For $f \in \DD_\AA$, define $\AA f = \lim_{t \searrow 0} (U_t f -f)/t.$ 
The operator $\AA : \DD_{\AA}   \rightarrow L  $ is called the infinitesimal generator associated to the semigroup $(U_t)_{t \geq 0}$ and $\DD_{\AA}$ is called the domain of the generator $\AA.$
\end{defi}

In the following we denote by $L_0$ (resp. $L_0^Y$) and $\AA$ (resp. $\AA^Y$) the set as in \eqref{defL0} and the infinitesimal operator associated to $(Q_t)_{t \geq 0}$ (resp. $(Q_t^Y)_{t \geq 0}$). Note that $L_0=L_0^Y$ by Proposition \ref{prop24}.

\begin{theo} \label{theogenerator}
    Let $(X_t)_{t \geq 0}$ be a general jump-move process on a Polish space $E$.
    Suppose that if $ f \in L_0^Y$, then $\a \times f \in L_0^Y$ and $Kf \in L_0^Y$.
    Then $\DD_\AA=\DD_{\AA^Y}$ and for any $f \in \DD_{\AA^Y}$,
    \[ 
    \AA f = \AA^Y f + \a \times Kf  - \a \times f .
    \]
\end{theo}

This result, proved in Section~\ref{proof generator} of the supplementary material, shows that the generator $\AA$ of the jump-move process $(X_t)_{t \geq 0}$ is just the sum of the generator of the move $\AA^Y$ and the generator of the jump, specifically of a pure jump Markov process with intensity $\a$ and transition kernel $K$, that is $\a \times (K- Id)$ (see \cite{feller1971}). 

Note that for a pure jump process, $Q_t^Y=Id$ for any $t \geq 0$, $L_0^Y=\DD_{\AA^Y}=M^b(E)$ and $\AA^Y \equiv 0$, so that all assumptions of Theorem~\ref{theogenerator} are trivially true in this setting. More generally, consider a jump-move process with a Feller inter-jumps process, i.e. $Q_t^Y \, C_0(E) \subset C_0(E)$, and a Feller jump transition, i.e. $K \, C_0(E) \subset C_0(E)$, so that $(X_t)_{t \geq 0}$ is Feller by Theorem~\ref{conditionsfeller}, then we can take $L_0^Y=C_0(E)$ and again the assumptions of Theorem~\ref{theogenerator} are satisfied since $\alpha$ is bounded.

%

\subsection{Application to systems of interacting particles in $\R^d$}\label{spatial Feller}

We go back to the setting of Section~\ref{sec2.3} concerned with systems of interacting particles in $W\subset \R^d$, in order to inspect whether the examples of dynamics presented therein are (continuous) Feller or not. To do so and be able to check the conditions of Theorem~\ref{conditionsfeller}, we need to first clarify what are the sets $ C_b(E)$ and $ C_0(E) $ in this framework. 
Remember that in this setting $E = \cup_{n \geq 0} E_n$ where $E_n = \pi_n( W^n)$ corresponds to the set of unordered $n$-tuples of $W$, and we have equipped $E$ with the distance $d_1$ defined by \eqref{defd1}.
As a first result, it can be verified that $(E,d_1)$ is a locally compact Polish space, see \cite[Proposition 2.2]{schuhmacher2008} and Section~\ref{sec6.2} in the supplementary material. To characterise the elements of $ C_b(E)$, we shall use the following proposition, proved in Section~\ref{sec6.2}. 
\begin{prop} \label{suitecvgentedeE}
 Let $x \in E$ and $(x^{(p)})_{p \geq 1}$ a sequence converging to $x$, i.e. $d_1(x^{(p)},x)\to 0$ as $p\to\infty$.  Then there exist $p_0 \geq 1$ such that for all $p \geq p_0$ $n(x^{(p)})=n(x)$  
                and, when $n(x)\geq 1$, a sequence $(\sigma_{p})_{p \geq p_0}$ of $\SS_{n(x)}$
        such that for any $i \in \{ 1 , \dots, n(x) \},$
        \begin{equation}
            \|x_{\sigma_{p}(i)}^{(p)}-x_{i}\| \underset{p \rightarrow \infty}{\longrightarrow} 0.
        \end{equation}
\end{prop}

On the other side,  to deal with $C_0(E)$, we provide a characterization of the compact sets  of each $E_n$, for $n \geq 1$, and an important property about the compact sets of $E$.

\begin{prop} \label{compactsetofE}
Suppose that $W$ is a closed set of $\R^d$. 
\begin{enumerate}
    \item Let $n \geq 1$ and $A$ be a closed subset of $(E_n,d_1).$ Then $A$ is compact if and only if the following property holds:
    \[ \forall \, w \in W, \; \exists \, R \geq 0, \; s.t. \; \forall \, x = \{ x_1,...,x_n\} \in A, \; \underset{1 \leq k \leq n}{\max} \{ \|x_k-w\|  \} \leq R. \]
    \item Let $A$ be a compact set of $E$. Then there exists $n_0 \geq 0$ such that $\displaystyle A \subset \bigcup_{n=0}^{n_0} E_n. $
\end{enumerate}
\end{prop}

The two previous propositions are the main tools to investigate the continuous Feller and Feller properties of the jump kernel $K$ of a jump-move process. Concerning the inter-jumps move process $(Y_t)_{t\geq 0}$, remember that we can easily define it on each $E_n$ from a continuous process $(Z^{|n}_t)_{t\geq 0}$ on $W^n$ by the projection \eqref{defmovesurEn}. Similarly as for the continuity property discussed in Section~\ref{sec2.3}, the Feller properties of  $(Y_t)_{t\geq 0}$ on $(E,d_1)$ inherit from that of  $(Z^{|n}_t)_{t\geq 0}$  on $W^n$. 

\begin{prop} \label{prop33}
    Let $(Y_t)_{t \geq 0}$ be defined on $E$ by \eqref{defmovesurEn}, then
 if $( Z^{\, |n}_t)_{t \geq 0}$ is a Feller continuous (resp. Feller) process on $W^n$ for every $n \geq 1$ then $(Y_t)_{t \geq 0}$ is a Feller continuous (resp. Feller) process on $E$. 
\end{prop}

By this result, standard inter-jumps motions are Feller continuous and Feller, as this is the case under mild assumptions for our examples \ref{moveito} and \ref{growthprocess} detailed below. 
Concerning the jump kernels, the global picture is as follows. They are generally Feller continuous, but not necessarily Feller even if  the underlying space $W$ is compact, as showed  in the following example. 
However if we restrict ourselves to birth kernels, then they are generally Feller (see Examples~\ref{noynaissgauss}  and \ref{noynaissgibbs} below). On the other hand, if we restrict ourselves to death kernels, then they are Feller if $W$ is compact, but not otherwise, see Example~\ref{noymort} below. Remark that a birth-and-death  jump kernel as in \eqref{K birth-death} is (continuous) Feller when the birth kernel $K_\beta$ and the death kernel $K_\delta$ are. So it is generally continuous Feller, and if $W$ is compact, it is generally Feller.

Let us make these informal claims more specific through some examples. The first one presents an example of jump kernel on a set $W$, possibly compact, that is continuous Feller but not Feller. The other ones correspond to the examples introduced in Section~\ref{sec2.3}.

\bigskip

\noindent \textbf{Example.} 
Consider the jump kernel $K$ defined for $f \in M^b(E)$ by $Kf(x)= \sum_{i=1}^{n(x)} f(\{x_i\})/n(x)$ for $x=\{x_1,\dots,x_{n(x)}\}\in E$, so that  $K(x,E_1)=1$ for any $x\in E$. Let $x^{(p)}$ be a sequence converging to $x$, from which we define $p_0$ and $(\sigma_p)_{p\geq p_0}$ as in Proposition~\ref{suitecvgentedeE}. Let $f \in C_b(E)$. Then  $K f(x^{(p)})  = \sum_{i=1}^{n(x)}   f( x^{(p)}_{\s_p(i)})/n(x)$ tends to $\sum_{i=1}^{n(x)}  f( x_{i})/n(x) = Kf(x)$ as $p\to\infty$, which shows the continuous Feller property of $K$, i.e. $K \, C_b(E) \subset C_b(E)$. Let us now show that $K$ is not Feller. Assume without loss of generality that $0\in W$. Consider the function $f(x)=\max(1-\|x\|,0)\1_{n(x)=1}$, where we abusively write $\|x\|:=\|x_1\|$ when $x=\{x_1\}$, $x_1\in W$. 
Note that $f\in C_0(E)$. Let $B$ be a compact subset of $E$. From  Theorem \ref{compactsetofE} there exists $n_0 \geq 0$ such that $ B \subset \cup_{n=0}^{n_0} E_n$. Choose $y=\{0,\dots,0\}\in E_{n_0+1}$. Then $y\notin B$ but $Kf(y)=1$ proving that $Kf\notin C_0(E)$.

\bigskip

\noindent \textbf{Example \ref{noymort} (continued)}\textit{(death kernel)}: For the death kernel $K_\delta$ of this example, we have
\begin{itemize}
\item[(i)] $K_\d C_b (E) \subset C_b(E)$
\item [(ii)] $K_\d C_0(E)\subset  C_0(E)$ if $W$ is compact, but not necessarily otherwise.
\end{itemize}
To prove the first property, take $x\in E$, a sequence $(x^{(p)})_{p \geq 0}$ converging to $x$, and $p_0$ and $(\sigma_p)_{p\geq p_0}$ from Proposition~\ref{suitecvgentedeE}. Then it is not difficult to verify that  $\lim_{p \rightarrow \infty} w(x_{\sigma_p(i)}^{(p)},x^{(p)}) = w(x_i,x)$ by continuity of $g$. Moreover, $ d_1 \left ( x^{(p)} \, \backslash \, x_{ \sigma_p(i)}^{(p)}, x \, \backslash \, x_{i} \right ) \leq  \sum_{j \neq i} \|x_{\sigma_p(j)}^{(p)} - x_{j} \|/(n-1)$ which shows that $x^{(p)} \, \backslash \, x_{\sigma_p(i)}^{(p)} \underset{p \rightarrow \infty}{\longrightarrow} x \, \backslash \, x_{i}$. 
Therefore, for any $f\in C_b(E)$,
\begin{align*}
    \lim \limits_{p \rightarrow \infty} K_\d f(x^{(p)}) & = \lim \limits_{p \rightarrow \infty}  \sum_{i=1}^{n(x)}     w(x_i^{(p)},x^{(p)}) f(x^{(p)} \, \backslash \, x^{(p)}_i) \\
    & = \lim \limits_{p \rightarrow \infty}  \sum_{i=1}^{n(x)}     w(x_{\sigma_p(i)}^{(p)},x^{(p)}) f(x^{(p)} \, \backslash \, x^{(p)}_{\sigma_p(i)})  \\
    & = \sum_{i=1}^{n(x)} w(x_{i},x) f(x \backslash \, x_{i})   \\
    &= K_\d f(x).
\end{align*}
Let us now consider the second claim $(ii)$. Take $f \in C_0(E)$ and $\ep >0$. We fix $A$ a compact set of $(E,d_{1})$ such that $|f(x)|< \ep$ for $x \notin A$. By Proposition~\ref{compactsetofE},  $ A \subset \bigcup_{n=0}^{n_0} E_n$ for some $n_0$. As a straightforward consequence of Proposition~\ref{compactsetofE} (see Corollary \ref{Cor9} in Section~\ref{sec6.2} of the supplementary material) the set $ B :=  \bigcup_{n=0}^{n_0+1} E_n $ is a compact set when $W$ is compact and it satisfies $K_\d(x,A)=0$ for $x \notin B$. 
This implies that for $x\notin B$,  
\begin{equation}\label{C0 example}
    |K_\d f(x)|  \leq \left | \int_A f(y) K_\d (x,\dd y)  \right | + \left | \int_{A^c} f(y) K_\d (x, \dd y)  \right | \leq ||f||_{\infty} K_\d (x,A) + \ep \, K_\d (x,A^c) \leq \ep,
\end{equation}
and so $K_\delta f\in C_0(E)$. Let us finally show that this result is not valid any more if $W$ is not compact. Assume without loss of generality that $0\in W$ and consider as in the previous example the function $f\in C_0(E)$ defined by  $f(x)=\max(1-\|x\|,0)\1_{n(x)=1}$. Let $B$ be any compact subset of $E$. Then $B_2=B\cap E_2$ is compact because $E_2$ is closed and by Proposition~\ref{compactsetofE},  for any $x=\{x_1,x_2\} \in B_2$, there exists  $R > 0$ such that $\max\{\|x_1\|,\|x_2\|\} \leq R$. Take $y=\{0,y_2\}$ in $E_2$ such that $\|y_2\|>R+1$, which is possible since $W$ is not compact. Then $y\notin B$ but  $K_\d f(y)=w(y_2,y)$ proving that $Kf\notin C_0(E)$.

\bigskip

\noindent \textbf{Example \ref{noynaissgauss} (continued)}\textit{(birth kernel as a mixture)}: For this example,  we shall prove that if $\mathring W \neq \varnothing$ and if  the dispersion function $v$ is continuous, then $K_\b C_b (E) \subset C_b(E)$
and  $K_\b C_0(E)\subset  C_0(E)$. 
Take $f\in C_b (E)$,  $x\in E$ and a sequence $(x^{(p)})_{p \geq 0}$ converging to $x$, from which we define $p_0$ and $(\sigma_p)_{p\geq p_0}$ from Proposition~\ref{suitecvgentedeE}. We have, for $p\geq p_0$,
\begin{align*} 
    K_\b f(x^{(p)})
    &=  \dfrac{1}{n(x)}  \sum_{i=1}^{n(x)} \dfrac{1}{z(x_{\sigma_p(i)}^{(p)},x^{(p)})}  \int_{W} f(x^{(p)} \cup \{ \xi \} )\varphi \left ( \frac{\xi-x_{\sigma_p(i)}^{(p)}}{v(x_{\sigma_p(i)}^{(p)},x^{(p)})} \right )  d\xi\\
    &=\dfrac{1}{n(x)}  \sum_{i=1}^{n(x)}  \frac{\int_{\R^d} \1_W(x_{\sigma_p(i)}^{(p)}+ v(x_{\sigma_p(i)}^{(p)} ,x^{(p)})\xi ) f(x^{(p)} \cup \{ x_{\sigma_p(i)}^{(p)}+ v(x_{\sigma_p(i)}^{(p)} ,x^{(p)})\xi \} ) \varphi(\xi)  d\xi}{\int_{\R^d}  \1_W(x_{\sigma_p(i)}^{(p)}+ v(x_{\sigma_p(i)}^{(p)} ,x^{(p)})\xi ) \varphi(\xi) d\xi}.
    \end{align*}
  By continuity of $v$, the involved indicator functions tend to $\1_W(x_i+ v(x_i ,x)\xi )$ for any $x_i+ v(x_i ,x)\xi\in \mathring W$. On the other hand, for any $i\in \{1,...,n(x)\}$ and any $\xi$, 
\begin{align*}
    & d_1 \left ( x^{(p)} \cup \{ x_{\sigma_p(i)}^{(p)}+ v(x_{\sigma_p(i)}^{(p)} ,x^{(p)})\xi \}, x \cup \{ x_{i} + v(x_{i},x) \, \xi \} \right ) \\
    & \leq \frac{1}{n(x)+1} \left ( \sum_{j =1}^{n(x)} \|x_{\sigma_p(j)}^{(p)} - x_{j} \| + \|x_{\sigma_p(i)}^{(p)} + v(x_{\sigma_p(i)}^{(p)},x^{(p)}) \, \xi  - x_{i} - v(x_{i},x) \, \xi  \| \right ) \\
    & \leq \frac{1}{n(x)+1} \left ( \sum_{j =1}^{n(x)} \|x_{\sigma_p(j)}^{(p)} - x_{j} \| + \|x_{\sigma_p(i)}^{(p)} - x_{i} \| + \|\xi\| \, |v(x_{\sigma_p(i)}^{(p)},x^{(p)})-v(x_{i},x)| \right )
    \end{align*}
which tends to 0 as $p\to\infty$. So by continuity of $f$, $ f(x^{(p)} \cup \{ x_{\sigma_p(i)}^{(p)}+ v(x_{\sigma_p(i)}^{(p)} ,x^{(p)})\xi \} )$ tends to $f(x \cup \{x_i + v(x_i,x)\xi)$. We conclude by the dominated convergence theorem, since $f$ is bounded and $\varphi$ is a density, that  $K_\b f(x^{(p)})$ converges to $ K_\b f(x)$ as $p\to\infty$, which proves that $K_\b C_b (E) \subset C_b(E)$.

Let us now prove that $K_\b C_0(E) \subset C_0(E).$
Let $f \in C_0(E)$ and $\ep >0$. We fix $A \subset E$ a compact set such that $x \notin A \Rightarrow | f(x) | < \ep$. By Proposition~\ref{compactsetofE}, $A \subset \bigcup_{n=0}^{n_0} E_n$ for some $n_0$. Letting $A_n=A \cap E_n$, for $n=0,\dots,n_0,$ we remark that $A_n$ is a compact set because $E_n$ is closed. 
By Proposition~\ref{compactsetofE}, there exists $R_n \geq 0$ such that for every $a=\{a_1,\dots,a_n\} \in A_n,$ 
$\max_{1 \leq k \leq n} \|a_k\| \leq R_n$. Let now $B_{n}=  \{ x \in E_{n}, \sum_{k=1}^{n} \|x_k\|/n \leq R_n \}$ and $B = \bigcup_{n=0}^{n_0-1} B_{n}$. We can verify (see the proof of Proposition~\ref{compactsetofE}) that $B_n$ is compact and so is $B$. We claim that if $x\notin B$, then $K_\beta(x,A)=0$. Indeed, if $K_\b (x,A)>0$ then $K_\b (x,A_n) >0$ for some $n\in\{0,\dots,n_0\}$, but since $  K_\b(x,A_0) \leq K_\b(x, \{ \textnormal{\O} \})=0$ it cannot be $n=0$. Now, for $n=1,\dots,n_0$,  $K_\b (x,A_n) >0$ implies that $x \in E_{n-1}$ and $A_n \subset \{ z \cup x, \, z \in W \}$ since $K_\b ( x, W \cup x)= 1$. So  $\max_{1 \leq k \leq n-1} \|x_k\| \leq R_n$, whereby $x\in B_{n-1}$. This shows that if $K_\b (x,A)>0$ then $x\in B$, as we claimed it. We deduce that for any $x\notin B$,  $|K_\b f(x)|\leq \epsilon$ as in \eqref{C0 example}.

\bigskip

\noindent \textbf{Example \ref{noynaissgibbs} (continued)}\textit{(birth kernel based on a Gibbs potential)}:  
This birth kernel $K_\b$ is both Feller continuous and Feller, whenever the potential $V$ is continuous and  locally stable. By the latter, we mean that there exists $\psi\in L^1(W)$ such that for any $x\in E$, 
$\exp(-(V(x \cup \xi)-V(x)))\leq \psi(\xi)$ \citep{moeller:waagepetersen:03}. Under these conditions, we can prove similarly as in Example~\ref{noynaissgauss} that $K_\b C_b(E) \subset C_b(E)$ by use of the dominated convergence theorem and that $K_\b C_0(E) \subset C_0(E)$. Note that the examples of potentials considered in 
Section~\ref{invariant Gibbs}, leading to an invariant Gibbs measure, are continuous and locally stable.

\bigskip

\noindent \textbf{Example \ref{moveito} (continued)}\textit{(Langevin diffusions as inter-jumps motions)}:
The inter-jumps process $(Y_t)_{t \geq 0}$, defined through the stochastic differential equation \eqref{sdeito}, is a Feller continuous and a Feller process on $E$. This is due to the fact that $g$ being globally Lipschitz, the function $\Phi_n$ in \eqref{sdeito} is also globally Lipschitz for any $n\geq 1$, and so the solution $(Z^{\, |n}_t)_{t \geq 0}$ of \eqref{sdeito} is Feller continuous and Feller \citep{schilling2012}. The conclusion then follows from Proposition \ref{prop33}.

\bigskip

\noindent \textbf{Example \ref{growthprocess} (continued)}\textit{(Growth interaction processes)}: 
In this example, the inter-jumps motion  is driven by \eqref{Z growth}. If the functions $F_{1,n}, \dots, F_{n,n}$ are Lipschitz continuous, then $(Y_t)_{t \geq 0}$ is Feller continuous and Feller. Indeed, the solution of  \eqref{Z growth} is continuous in the initial condition $Z_0^{|n}$ under this assumption \citep{markley}, implying the Feller continuity of $(Y_t)_{t \geq 0}$ by Proposition \ref{prop33}. Moreover, since the marks $m_i(t)$ in  $(Z_t^{|n})_{t \geq 0}$ are all increasing functions, we have that $\|Z_t^{|n}\|\geq \|Z_0^{|n}\|$. Let $f\in C_0(W^n)$, $\epsilon>0$ and $R>0$ such that $\|x\|>R\Rightarrow |f(x)|<\epsilon$. Therefore if $\|Z_0^{|n}\|>R$, we have $\|Z_t^{|n}\|\geq R$ and so $f(Z_t^{|n})<\epsilon$, proving that $Z_t^{|n}$ is Feller and so is $(Y_t)_{t \geq 0}$ by Proposition \ref{prop33}.

\section{Ergodic properties of birth-death-move processes}\label{asymptotics}

In this section we focus on birth-death-move processes as described in Section~\ref{sec2.2}. Accordingly, the state space is $E=\bigcup_{n =0}^{\infty} E_n$ where $(E_n)_{n \geq 1}$ is a sequence of disjoint locally compact Polish spaces with $E_0=\{ \textnormal{\O} \}$, and the jump-kernel $K$ reads as in \eqref{K birth-death}.  Remember that in this setting the jump intensity function is $\alpha=\beta+\delta$ where $\beta$ and $\delta$ are the birth and death intensity functions. We introduce the following notation
\begin{equation}\label{bndn}
    \b_n = \underset{x \in E_n}{\sup} \b(x), \quad \d_n = \underset{x \in E_n}{\inf} \d(x)\quad \textnormal{and}\quad \a_n = \b_n + \d_n.
\end{equation}
We present in Section~\ref{sec4.1} conditions on the sequences $(\beta_n)$ and $(\delta_n)$ ensuring the convergence of the birth-death-move process towards a unique invariant probability measure. Section~\ref{sec4.2} 
 provides a geometric rate of convergence and we characterize some invariant measures in Section~\ref{sec:invariant}.

\subsection{Convergence to an invariant measure} \label{sec4.1}

Let $(X_t)_{t \geq 0}$ be a birth-death-move process as in Section~\ref{sec2.2}.  Inspired by \cite{preston}, the first step to establish the ergodic properties of $(X_t)_{t \geq 0}$ is to construct a coupling between  $(X_t)_{t \geq 0}$ and a  simple birth-death process $(\eta_t)_{t \geq 0}$  on $\N$ with birth rates $\b_n $ and death rates $\d_n$. This coupling is detailed in Appendix~\ref{sec6.7}. In a nutshell, we built a jump-move process $\ck C_t=(X'_t, \eta'_t)$ so that the  properties stated in the following theorem are satisfied, in particular  we have the equality in distribution $(X'_t)_{t\geq 0}=(X_t)_{t\geq 0}$ and $( \eta'_t)_{t\geq 0}=(\eta_t)_{t\geq 0}$.   We denote by $\ck Q_t$ the transition kernel of $\ck C_t$ and by $\PP(\N)$  the power set of $\N$. 

\begin{theo} \label{Marginals}
Suppose that $(Y_t)_{t \geq 0}$ is a Feller process and that $K C_0(E) \subset C_0(E)$. Then for any $t \geq 0$, $(x,n) \in E \times \N$, $A \in \EE$ and $S \in \PP(\N)$ we have :
\begin{enumerate}
    \item $\ck Q_t((x,n);E \times S)=q_t(n,S)$,
    \item $\ck Q_t((x,n);A \times \N)=Q_t(x,A)$,
    \item if $x \in E$ with $n(x) \leq n$, then $\ck Q_t((x,n);\Gamma)=0$ where $\Gamma= \{ (y, m) \in E \times \N \, ; \,  n(y) >  m\}$.
\end{enumerate}
\end{theo}

This theorem is proved in Appendix~\ref{sec6.7}. 
When the move $(Y_t)_{t \geq 0}$ is constant, which is the setting in  \cite{preston}, then the proof is easy under \eqref{assumption alpha} by use of the derivative form of the Kolmogorov backward equation. In the general case of a birth-death-move process, this strategy does not work anymore and the proof becomes more challenging. We managed to do it by exploiting the generator of $(X_t)_{t \geq 0}$, see Theorem~\ref{theogenerator}, which explains the Feller conditions in Theorem~\ref{Marginals}.

   We deduce from the third point of Theorem~\ref{Marginals} that for any $x \in E_m$ with $m \leq n$, then 
   $$\P_{(x,n)}\left (  ( \ck C_s)_{s \geq 0 } \subset \Gamma^c  \right ) =1.$$
     This means  that the simple process $(\eta_t)_{t \geq 0}$ converges more slowly to the state $0$ than  $(X_t)_{t \geq 0}$ converges to the state \O. We can thus build upon the renewal theory \citep{feller1971} to prove that \O \  is an ergodic state for $(X_t)_{t \geq 0}$ whenever $0$ is an ergodic state for $(\eta_t)_{t \geq 0}$. Conditions ensuring the latter are either  \eqref{eq30} or \eqref{eq31} below \citep{karlin1957}, so that we obtain the following, proved in Section~\ref{proof 4.1}. 

 \begin{theo} \label{existunicitémesinvariante}
    Suppose that $(Y_t)_{t \geq 0}$ is a Feller process and that $K C_0(E) \subset C_0(E)$. Suppose that $\d_n>0$ for all $n \geq 1 $ and one of the following condition holds :
\begin{align} 
(i)&  \text{ there exists }  n_0 \geq 1 \text{ such that } \b_{n}=0 \text{ for any } n \geq n_0,  \label{eq30}\\
(ii)& \ \b_n>0 \text{ for all } n \geq 1, \ \sum_{n=2}^{\infty} \dfrac{\b_1 \dots \b_{n-1}}{\d_1 \dots \d_n} < \infty \text{ and }  \displaystyle \sum_{n=1}^{\infty} \dfrac{\delta_1  \dots \delta_n}{\beta_1  \dots \beta_n} = \infty.  \label{eq31}
\end{align}
Then, $\mu(A):=\lim_{t \rightarrow \infty} Q_t(x,A)$ exists for all $x \in E$ and $A \in \EE$, and is independent of $x$.  
Moreover $\mu$ is a probability measure on $(E,\EE)$ and it is the unique invariant probability measure for the process, i.e. such that $\mu(A) = \int_E Q_t(x,A) \, \mu(\dd x)$ for any $A \in \EE$ and $t \geq 0$.
\end{theo}

\subsection{Rate of convergence} \label{sec4.2}

Based on the coupling given by Theorem~\ref{Marginals}, and under the assumptions of Theorem~\ref{existunicitémesinvariante}, the rate of convergence of $Q_t$  towards the invariant measure $\mu$ follows from the rate of convergence of the simple birth-death process $\eta$ towards its invariant distribution. This is proven and exploited in 
\cite{moller1989} in the case of  spatial birth-death processes (without move), based upon the coupling of \cite{preston}.
Since Theorem~\ref{Marginals} extends this coupling, we deduce in the following theorem the same rates of convergence as in \cite{moller1989}. The proof is readily the same and we omit the details. 

\begin{theo}\label{theo:rate}
Suppose that $(Y_t)_{t \geq 0}$ is a Feller process and that $K C_0(E) \subset C_0(E)$. 
 Let $\gamma_1$ and  $\gamma_2$ be two probability measures on $(E,\EE)$, such that  one of the two following conditions holds: 
 \begin{align} 
(i)& \text{  \eqref{eq30} holds true and for }k=1,2,\  \gamma_k \left ( \bigcup_{n=0}^{n_0} E_n \right ) = 1,  \label{eq32}\\
(ii)& \text{  \eqref{eq31} holds true and for }k=1,2,\  \sum_{n=2}^{\infty } \gamma_k ( E_n) \sqrt{\frac{\d_1 \dots \d_n}{\b_1 \dots \b_{n-1}}}< \infty.   \label{eq33}
\end{align}
Then there exist real constants $c >0$ and $0<r<1$ such that for any $t \geq 0$ 
$$ \underset{A \in \EE}{\sup} \left | \int_E Q_t(x,A) \gamma_1 (dx) - \int_E Q_t(y,A) \gamma_2(dy) \right | \leq c r^t . $$
Moreover when the condition \eqref{eq32} holds, the constants $c$ and $r$ can be chosen independently of $\gamma$ and $\kappa.$
\end{theo}
This result is declined into several particular cases in \cite[Corollary 3.1]{moller1989}, that are also valid in our setting.
In particular, when $\gamma_1$ corresponds to the invariant measure $\mu$ obtained in Theorem~\ref{existunicitémesinvariante}, and $\gamma_2$ is a point measure, the assumptions \eqref{eq32} and \eqref{eq33} simplify and we get the following corrollary. 
\begin{cor} 
Suppose that $(Y_t)_{t \geq 0}$ is a Feller process and that $K C_0(E) \subset C_0(E)$. 
Assume  either \eqref{eq30} or \eqref{eq31} along with
 \begin{align} 
 \sum_{n=2}^{\infty} \sqrt{\frac{\b_1 \dots \b_{n-1}}{\d_1 \dots \d_n}} < \infty \quad\text{ and }  \quad
 \exists N \geq 0, \; s.t. \; \forall \, n \geq N, \; \b_n \leq \d_{n+1}. \end{align}
Denote by $\mu$ the invariant measure given by Theorem~\ref{existunicitémesinvariante}. Then for any $y \in E$,  there exists $c>0$ and $0<r<1$ such that 
\[ \underset{A \in \EE}{\sup} \left |  \mu(A)-  Q_t(y,A)  \right | \leq c r^t . \]
\end{cor}

\subsection{Characterisation of some invariant measures}\label{sec:invariant}

In general the invariant measure $\mu$ of a birth-death-move process $(X_t)_{t \geq 0}$, provided it exists, can be a very complicated distribution that mixes the repartition in $E$ due to births and deaths of points, including the probability to be in $E_n$ for each $n$, with the average distribution on each $E_n$ due to the move process $Y$. 
In particular, note that according to Theorem~\ref{existunicitémesinvariante}, $Y$ does not need to be a stationary process for $(X_t)_{t \geq 0}$ to converge to an invariant measure. Heuristically, this is because the move process is always eventually ``killed'' by a return to \O \  of $(X_t)_{t \geq 0}$ under the hypotheses of Theorem~\ref{existunicitémesinvariante}. 

The situation becomes more intelligible when $Y$ admits an invariant measure that  is compatible with the jumps of  $(X_t)_{t \geq 0}$, as formalised in the next proposition.

\begin{prop} \label{mesinv}
Suppose that $(Y_t)_{t \geq 0}$ is a Feller process and that $K C_0(E) \subset C_0(E)$. 
Assume moreover that there exists a finite measure $\mu$ on $E$ such that for any $f \in \DD_{\AA^Y}$, 
    \begin{align}
        & \int_{E_n} \AA^Y f(x) \, \dd \mu_{|E_n}(x) = 0,\quad \forall n\geq 0, \label{invariant Y} \\
       \text{and }\quad & \int_E \left ( \a (x)K f(x)-\a(x)f(x) \right ) \, \dd \mu(x)=0. \label{invariant jump}
    \end{align}
    Then for any $f \in \DD_{\AA^Y}$, $\displaystyle \int_E \AA f(x) \, \dd \mu(x) = 0$.
\end{prop}
\begin{proof}
    By Theorem~\ref{theogenerator}, for any $f \in \DD_{\AA^Y}$, 
    \begin{align*}
        \int_E \AA f(x) \, \dd \mu(x) & = \int_E \left ( \a (x)K f(x)-\a(x)f(x) \right ) \, \dd \mu(x) + \int_E \AA^Y f(x) \, \dd \mu(x) \\
       &  = \sum_{n \geq 0} \int_{E_n} \AA^Y f(x) \, \dd \mu_{|E_n}(x) =0.
    \end{align*}
\end{proof}

This proposition will be useful to characterize the invariant measure of the birth-death-move processes considered in Section~\ref{invariant Gibbs}. Indeed,  suppose that the hypotheses of Theorem~\ref{existunicitémesinvariante} are satisfied. Then $(X_t)_{t \geq 0}$ converges to a unique invariant measure $\nu$. 
Suppose moreover that the pure jump Markov process with intensity $\a$ and transition kernel $K$ admits some invariant measure $\mu$, and that for any $n\geq 0$, $\mu_{|E_n}$ is also invariant for the move process $Y^{|n}$ on $E_n$. Then by Proposition~\ref{mesinv} and the unicity of $\nu$, we have that $\nu=\mu$.

\section{Application to pairwise interaction processes on $\R^d$} \label{invariant Gibbs}

We present in this section examples of birth-death-move processes, defined through a pairwise potential function $V$ on a compact set $W\subset \R^d$, that converge to the Gibbs probability measure associated to $V$. The specificity is that we make compatible the jumps dynamics with the inter-jumps diffusion, so that Proposition~\ref{mesinv} applies and allows us to characterize this Gibbs measure as the invariant measure. 

When there is no inter-jumps motion, this type of convergence is proved in \cite{preston} and is a prerequisite for perfect simulation of spatial Gibbs point process models (see \cite[Chapter 11]{moeller:waagepetersen:03}).
However the weakness of this approach is that for rigid interactions (as  for instance induced by a Lennard-Jones or a Riesz potential, see the examples below), the dynamics based on spatial births and deaths may mix poorly,  so that the convergence to the associated Gibbs measure may  be very slow. Adding inter-jumps motions  that do not affect the stationary measure, as carried out in this section, may alleviate this issue.

Let $W:= I_1 \times  \dots \times I_d$ where, for $i \in \{1,\dots,d\}$, $I_i$ is a compact interval of $\R$.
Define $\tilde W_n =  \{ (x_1,\dots,x_n) \in  (\mathring{W})^n, \; i \neq j \Rightarrow x_i \neq x_j \}$. As in  Section~\ref{sec2.3}, we let $E_0 = \{ \textnormal{\O} \}$, $E_n=\pi_n( \tilde W_n)$ for $n\geq 1$, and $E=\bigcup_{n=0}^\infty E_n$.

We consider a pairwise potential function $V: E \to \R \cup \{\infty\}$, in the sense that there exist $a>0$ and  $\phi: \R^d\to\R \cup\{\infty\}$ satisfying  $\phi(\xi)=\phi(-\xi)$ for all $\xi\in \R^d$ such that for any $x=\{x_1,\dots,x_n\}\in E_n$, 
$$V(x)=a \,n(x) + \sum_{1\leq i\neq j\leq n} \phi(x_i-x_j),$$ 
when $n\geq 2$, while $V(\{\textnormal{\O} \})=0$ and $V(\{\xi\})=a$ for $\xi\in W$. 
 Let $\phi_0 : (0,\infty)\to \R_+$ be a decreasing function with  $\phi_0(r)\to\infty$ as $r \to 0$. We will assume the following on $\phi$.
\begin{enumerate}[label=(\Alph*)]
\item The potential is locallly stable, i.e there exists $\psi : W \rightarrow \R_+$ integrable such that : 
\begin{equation*}
        \forall \, n \geq 1, \; \forall \, x  \in E_n, \; \forall \, \xi \in W, \; \; \exp \left ( - \sum_{i=1}^n \phi(x_i-\xi) \right ) \leq \psi(\xi).
    \end{equation*}  \label{conditionLS}
 \item $\phi$ is bounded, or otherwise there exists $r_1>0$ such that $\phi(\xi)\geq \phi_0(\|\xi\|)$ for all $\|\xi\|<r_1$. \label{condition repulsive}
   \item $\phi$ is weakly differentiable on $\R^d \backslash \{ 0 \}$,  $\exp(-\phi)$ is weakly differentiable on $\R^d$ and for any $p > d$ we have $\e^{-\phi} \nabla \phi \in L^p_{loc}.$ \label{conditionC}
\end{enumerate}
Let us present some examples of pairwise potentials $\phi$ that fulfil these assumptions. These are standard instances used in spatial statistics and statistical mechanics.

\medskip

\noindent \textbf{Example.} \textit{(repulsive Lennard-Jones potential)}: 
For $\xi\in\R^d$,  $\phi(\xi)=c\| \xi \|^{-12}$ with $c>0$. This potential verifies  condition \ref{conditionLS} with $\psi \equiv 1$ and  condition \ref{condition repulsive}. It is moreover  differentiable on $\R^d \backslash \{0 \}$ and for any $\xi \in \R^d \backslash \{0 \}$, $\nabla \phi(\xi)= - 12c  \xi/\| \xi  \|^{14}.$ We deduce that the function $ \e^{-\phi}\nabla \phi$ can be extended to a continuous function on $\R^d$ by setting $( \e^{-\phi}\nabla \phi)(0)=0$. As a consequence the condition \ref{conditionC} is satisfied.

\medskip

\noindent \textbf{Example.} \textit{(Riesz potential)}:  It is defined on $\R^d \backslash \{ 0 \}$ by $\phi(\xi)=c\| \xi \|^{\a-d}$ for $c>0$ and $0< \a < d$. As in the previous example, we obtain that $\phi$ satisfies the conditions \ref{conditionLS}, \ref{condition repulsive} and \ref{conditionC}.

\medskip

\noindent \textbf{Example.} \textit{(soft-core potential)}:  $\phi(\xi)=- \ln \left ( 1-\exp(-c \|\xi \|^2) \right )$ for $c>0$. Again this potential verifies condition \ref{conditionLS} with $\psi \equiv 1$ and condition \ref{condition repulsive}.  Moreover, we compute for $\xi \in \R^d \backslash \{ 0 \}$, $ \nabla \phi(\xi)= - (2 c \e^{-c \|\xi \|^2})(1-\e^{-c \|\xi \|^2}) \, \xi.$ As $\| \nabla \phi(\xi) \| \sim 1/(c \| \xi \|)$ as  $\| \xi \|\to 0$, we also obtain that the function $ \e^{-\phi}\nabla \phi$ can be extended to a continuous function on $\R^d$ and  condition \ref{conditionC} follows.

\medskip

\noindent \textbf{Example.} \textit{(regularised Strauss potential)}: For $R>0$ and $\gamma\geq 0$, the standard Strauss potential corresponds to $\phi(\xi)=\gamma \1_{\|\xi\|<R}$. We consider a regularised version by introducing a parameter $0<\epsilon<R$, so that  $\phi(\xi)= \gamma$ if $ \| \xi \| \leq R - \ep$, $\phi(\xi)=0$  if $ \| \xi \| \geq R + \ep$, and $\phi$ is interpolated between $R - \ep$ and $R + \ep$ in such a way that it is differentiable. With this regularised version, $\phi$  satisfies the conditions \ref{conditionLS} with $\psi \equiv 1$, \ref{condition repulsive} and \ref{conditionC}.

\medskip

Based on a potential $V$ as above, we construct a birth-death-move process $(X_t)_{t\geq 0}$ with the following characteristics. The birth transition kernel is given as in Example~\ref{noynaissgibbs} by 
\[
K_\b(x,\Lambda \cup x) = \frac{1}{z(x)} \int_\Lambda \e^{-(V(x \cup \xi)-V(x))} \, \dd \xi,
\]
for any $x\in E$ and $\Lambda\subset W$, where $ z(x)=\int_W \exp(-(V(x \cup \xi)-V(x))) \, \dd \xi$. Note that by the local stability assumption \ref{conditionLS}, $z(x)<\infty$ for any $x\in E$.  The death transition kernel is just the uniform kernel, a particular case of Example~\ref{noymort}, i.e. 
$$K_\d f(x) = \dfrac{1}{n(x)} \sum_{i=1}^{n(x)} f(x \, \backslash \, x_i)$$
 for any  $f \in \MM_b(E)$ and $x=\{x_1,\dots,x_{n(x)}\}\in E$. For the birth  and death intensity functions, we take 
 $$\beta(x)=\frac{z(x)}{n(x)\vee 1} \quad \text{and}\quad \delta(x)=\1_{n(x)\geq 1},$$
 for any $x\in E$. Finally, for the move process, we  start with the following Langevin diffusion on $\tilde W_n$:
 \[
\dd Z_{t,i}^{|n} = - \sum_{j \neq i} \nabla \phi(Z_{t,i}^{|n} - Z_{t,j}^{|n}) \, \dd t + \sqrt{2} \, \dd B_{t,i},\quad 1 \leq i \leq n,
\]
with reflecting boundary conditions \citep{fattler2007}, and we deduce the move process $Y$ on $E$ as in Example~\ref{moveito}.

\begin{prop}
The birth-death-move process $(X_t)_{t\geq 0}$ defined above is a Feller process  and converges towards the invariant Gibbs probability measure on $W$ with potential $V$, i.e. the mesure having a density proportional to $\exp(-V(x))$ with respect to the unit rate Poisson point process on $W$.
\end{prop}

\begin{proof}
First note that by the local stability assumption \ref{conditionLS}, $\beta(x)\leq e^{-c} \|\psi\|_1/(n(x)\vee 1)$, where $\|\psi\|_1=\int_W \psi(\xi)d\xi$, so that $\alpha(x)=\beta(x)+\delta(x)$ is uniformly bounded as required by \eqref{assumption alpha}. 

Under the assumptions \ref{condition repulsive} and \ref{conditionC}, \cite{fattler2007} proved that the process $(Z_{t}^{|n})_{t\geq 0}$ is a well defined Markov process on $\tilde W_n$ and it is a Feller process. By Proposition~\ref{prop33}, $Y$ is then a Feller process on $E$. On the other hand the jump transition kernel $K$ given by \eqref{K birth-death} satisfies $KC_0(E)\subset C_0(E)$, as verified in Examples~\ref{noymort} and \ref{noynaissgibbs}  in Section~\ref{spatial Feller} since $W$ is compact. We thus obtain by Theorem~\ref{conditionsfeller} that $(X_t)_{t\geq 0}$ is a Feller process.
Moreover, by \ref{conditionLS} we have that for all $n\geq 1$, $\beta_n\leq e^{-c} \|\psi\|_1/n$, so that \eqref{eq31} is verified. 
All assumptions of Theorem~\ref{existunicitémesinvariante} are satisfied, which implies that  
$(X_t)_{t\geq 0}$ converges to a unique invariant probability measure as $t\to\infty$.

It remains to characterize this invariant measure. The choice of $\beta$, $\delta$, $K_\beta$ and $K_\delta$ satisfy the conditions of \cite[Theorem~8.1]{preston}, see also  \cite[Chapter 11]{moeller:waagepetersen:03}, which implies that the invariant measure $\mu$ for the birth-death process (without move) having the previous characteristics is the one claimed in the proposition. We deduce that \eqref{invariant jump} holds true. On the other hand,  \cite{fattler2007} proved under \ref{condition repulsive} and \ref{conditionC} that $(Z_{t}^{|n})_{t\geq 0}$ converges towards the invariant measure on $\tilde W_n$ with a density (with respect to the Lebesgue measure) proportional  to $\exp(-\sum_{1\leq i\neq j\leq n} \phi(x_i-x_j))$. After projection on $E_n$, this means that \eqref{invariant Y} follows, with the same measure $\mu$ as before. Proposition~\ref{mesinv} then applies and $\mu$ is the invariant measure of $(X_t)_{t\geq 0}$. 
\end{proof}

\appendix

\section{Appendix: coupling of birth-death-move processes} \label{sec6.7}

\subsection{Construction of the coupling}\label{construction coupling}

We start from a birth-death-move process $(X_t)_{t \geq 0}$ as defined in Section~\ref{sec2.2}. 
We consider a simple birth-death process  $(\eta_t)_{t \geq 0}$ on $\N$ with birth rate $\b_n$ and death rate $\d_n$ given by \eqref{bndn}.  Note that  $(\eta_t)_{t \geq 0}$ can be viewed as a birth-death-move process on $\N$ having a constant move process $y_t =y_0$, for all $t \geq 0$.  We denote by $(t_j)_{j \geq 1}$ the jump times of $(\eta_t)_{t \geq 0}$ and by $n_t := \sum_{j \geq 1} \ind_{t_j \leq t}$ the number of jumps before $t \geq 0$. We also denote by $q_t$ the transition kernel of $(\eta_t)_{t \geq 0}$, i.e. $q_t(n,S)=\P(\eta_t\in S|\eta_0=n)$ for any $n\in \N$ and $S\in \PP(\N)$.

We define the coupled process $\ck C=(X',\eta')$ as a jump-move process on the state space $\ck E = E \times \N$ equipped with the $\s-$algebra $\ck \EE = \EE \otimes \PP (\N)$. Denoting by $d$ the distance on $E$, we also equip $\ck E$ with the distance $\ck d((x,k);(y,n)):=d(x,y)+  |n-k|/(n\wedge k)\ind_{nk \neq 0}$. To fully characterize $\ck C$, we now specify its jump intensity function $ \ck \a$, its jump kernel $\ck K$ and its inter-jump move process $\ck Y$.

The intensity function $\ck \a : E \times \N \rightarrow \R_+$ is given by 
\begin{equation*}
    \ck \a (x,n)= \left \{    \begin{array}{lcl}
    \beta(x) + \delta(x) + \beta_n + \delta_n & & \textnormal{if } x \in E_m, \, m \neq n  , \\
    \beta_n + \delta(x) & & \textnormal{if } x \in E_n.  
\end{array} \right.
\end{equation*}
One can easily check that $\ck \a $ is a continuous function on $\ck E$, bounded by $2\alpha^*$. 

The transition kernel $\ck K : \ck E \times \ck \EE \rightarrow [0,1]$ takes the same specific form as in \cite{preston}:
\begin{enumerate}
    \item If $x \in E_m, \, m \neq n$ : 
    \[ 
        \ck K((x,n);A \times\{n\})  =  \dfrac{\a (x)}{ \ck \a (x,n)} K(x,A); \]
       \[ \ck K((x,n);\{x\}\times\{n+1\})  =  \dfrac{\beta_n}{ \ck \a (x,n)} ;\]
       \[ \ck K((x,n);\{x\}\times\{n-1\})  =  \dfrac{\delta_n}{ \ck \a (x,n)}.
        \]
    \item If $x \in E_n$ : 
    \[ 
        \ck K((x,n);A_{n+1}\times\{n+1\})  =  \dfrac{\beta (x)}{\ck \a (x,n)} K_\beta(x,A_{n+1}) ;\]
      \[  \ck K((x,n);\{x\}\times\{n+1\})  =  \dfrac{\beta_n-\beta(x)}{\ck \a (x,n)}; \]
       \[ \ck K((x,n);A_{n-1}\times\{n-1\})  =  \dfrac{\delta_n}{\ck \a (x,n)}K_\delta(x,A_{n-1}); \]
      \[  \ck K((x,n);A_{n-1}\times\{n\})  =  \dfrac{\delta(x)-\delta_n}{\ck \a (x,n)}K_\delta(x,A_{n-1}) .
     \]
\end{enumerate}

The inter-jump move process $\ck Y$ is finally obtained by an independent coupling of $(Y_t)_{ t \geq 0}$ and $(y_t)_{t \geq 0}$, specifically its transition kernel $Q_t^{\ck Y}$ is given for any $(x,p)\in\ck E$ and $A\times S\in \ck \EE$ by 
\begin{equation}\label{QtYcheck}
      Q_t^{\ck Y}((x,p);A \times S)= \P( \ck Y_t \in A \times S | \ck Y_0=(x,p))=\P(Y_t \in A | Y_0=x) \mathds{1}_S(p) = Q_t^{Y}(x,A) \mathds{1}_S(p).
\end{equation}
This means that $\ck Y_t = (Y'_t,y'_t)=(Y'_t,y'_0)$ for any $t\geq 0$, where $(Y'_t)_{t\geq 0}$ and $(y'_t)_{t\geq 0}$ are independent and follow the same distribution as $(Y_t)_{t\geq 0}$ and $(y_t)_{t\geq 0}$, respectively. 
Since $Y$ is a continuous Markov process for the distance $d$, we can choose a version of $Y'$ such that  $\ck Y$ is also continuous for $\ck d$. Remark moreover that  $(\ck Y_t)_{t \geq 0}$ satisfies 
$$ \P( ( \ck Y_t)_{t \geq 0} \subset E_n \times \{ k \} \, | \, \ck Y_0=(x,p))= \ind_{E_n}(x) \, \ind_{k=p}, \; \; \; \forall x \in E, \, \forall n \geq 0.  $$

Given  $ \ck \a$, $\ck K$ and $\ck Y$ as above, the jump-move process   $\ck C$ is well defined and can be constructed as in Section~\ref{jump-move general}. We denote by $\ck Q_t$ its transition kernel, by $(\ck T_j)_{j \geq 1}$ its jump times and by $\ck N_t := \sum_{j \geq 1} \ind_{ \ck T_j \leq t }$  the number of jumps before $t \geq 0$.  We also set $ \ck \t_{j}=\ck T_j - \ck T_{j-1}$. 
  The fact that $\ck C$ defines a relevant coupling of $X$ with $\eta$ is the object of Theorem~\ref{Marginals}.

\subsection{Proof of Theorem~\ref{Marginals}}

To prove the first point of the theorem, we use the following results, proved in Section~\ref{proofs coupling} of the supplementary material. 
Fix $(x,n) \in E \times \N$ and $q \geq 0$, and let  $$\psi_q: t \in \R_+ \mapsto \ck Q_t((x,n),E \times \{ q \}).$$

\begin{lem} \label{QqisC0}
For any $(x,n) \in E \times \N$ and $q \geq 0$, $\psi_q$  is a continuous function. 
\end{lem}

\begin{lem} \label{lem41}
For any $(x,n) \in E \times \N$ and $q \geq 0$, $\psi_q$  is right differentiable and satisfies :
     \begin{align*}
        \dfrac{\partial_+}{\partial t} \psi_q(t) = - \a_q \, \psi_q(t) + \b_{q-1} \, \psi_{q-1}(t) + \d_{q+1} \,  \psi_{q+1}(t).
    \end{align*}
    \end{lem}

\begin{cor} \label{Marg1diff}
For any $(x,n) \in E \times \N$ and $q \geq 0$, 
 \begin{align*}
           \psi_q(t) = \ind_{q}(n) + \int_0^t \left ( - \a_q \, \psi_q(s) + \b_{q-1} \, \psi_{q-1}(s) + \d_{q+1} \,  \psi_{q+1}(s) \right ) \, \dd s
        \end{align*}
 and in particular $\psi_q$  is differentiable. 
\end{cor}

Let now $ w_s(x,n) = \ck Q_{t-s} ( \ind_E \times q_s \ind_{\{q\}})(x,n)$ for $s \in [0,t].$ Then using Corollary~\ref{Marg1diff}, we have:
\begin{lem}\label{wsdiff}
For any $(x,n) \in E \times \N$ and $q \geq 0$, $s\mapsto w_s$ is differentiable on $[0,t]$ and $\partial w_s/\partial s \equiv0$.
\end{lem}

Since $ w_0(x,n) = \ck Q_t((x,n); E \times \{ q \})$ and $w_t(x,n) = q_t(n, \{ q \})$,  Lemma~\ref{wsdiff}  entails that these two quantities are equal.
 The first part of Theorem~\ref{Marginals}  then  follows from the decomposition:
    $$ q_t(n,S) = \sum_{q \in S} q_t(n, \{ q \} ) = \sum_{q \in S} \ck Q_t((x,n); E \times \{ q \}) = \ck Q_t((x,n); E \times S).  $$

We turn to the proof of the second point of Theorem~\ref{Marginals}. Similarly as for the first part, it is based on the following results proved in Section~\ref{proofs coupling}. For  $(x,n) \in E \times \N$ and $f \in C_0(E)$, we set  $$\psi_f: t \in \R_+ \mapsto \ck Q_t(f \times \ind_\N)(x,n).$$

\begin{lem} \label{marg2C0}
    Suppose that $(Y_t)_{t \geq 0}$ is a Feller process. Then for any  $(x,n) \in E \times \N$ and any $f \in C_0(E)$, $\psi_f$  is a continuous function.
\end{lem}

\begin{lem} \label{suptendsvers0}
    Suppose that $(Y_t)_{t \geq 0}$ is a Feller process and that $K C_0(E) \subset C_0(E)$. 
    Then for any $(x,n) \in E \times \N$ and any $f \in \DD_{\AA^Y}$ the function $\psi_f$ is right-differentiable and satisfies :
    \begin{equation}\label{diffpsi}
    \dfrac{\partial_+}{\partial t} \psi_f(t) =   \psi_{\AA f}(t),
        \end{equation}
     where $\AA$ is the infinitesimal generator of  $X$ given by Theorem~\ref{theogenerator}.
 \end{lem}

\begin{cor} \label{coro46}
    Suppose that $(Y_t)_{t \geq 0}$ is a Feller process and that $K C_0(E) \subset C_0(E)$.  Then for any  $(x,n) \in E \times \N$ and any $f \in \DD_{\AA^Y}$, 
      \begin{align} \label{formuledeckQt}
       \psi_f(t)   = f(x) + \int_0^t    \psi_{\AA f}(s)  \, \dd s,
    \end{align}
  and in particular the function $\psi_f$ is differentiable with derivative corresponding to \eqref{diffpsi}.
  \end{cor}

By the Dynkin theorem, the second point of Theorem~\ref{Marginals} is implied by the equality $\ck Q_t((x,n);U \times \N) = Q_t(x,U)$ for any open set $U\subset E$, or equivalently 
\begin{equation}\label{marg2}
\ck Q_t(g \times \ind_\N)(x,n) = Q_t (g)(x)
\end{equation}
for $g=\1_U$. We first prove \eqref{marg2} for $g\in \DD_{\AA^Y}$, second for $g \in C_0(E)$, before getting the result for  $g=\1_U$.

 Let $g\in \DD_{\AA^Y}$ and for $s\in[0,t]$, define $v_s(x,n)=\psi_{Q_s g}(t-s) =  \ck Q_{t-s} \left ( Q_s g \times \ind_\N \right )(x,n)$. We shall prove that $s\mapsto v_s$ is differentiable with $v'_s=0$. For any $h\in \R$, write $(v_{s+h}(x,n)-v_s(x,n))/h=A_1+A_2+A_3$
 with 
 \begin{align*}
A_1 &= \frac1 h\left(\psi_{Q_{s+h} g}(t-s-h) - \psi_{Q_s g}(t-s-h)\right)  - \psi_{\AA Q_s g}(t-s-h),  \\
A_2 &= \psi_{\AA Q_s g}(t-s-h)  - \psi_{\AA Q_s g}(t-s), \\ 
A_3 &=  \frac 1 h \left(\psi_{Q_s g}(t-s-h) -  \psi_{Q_s g}(t-s)\right)+ \psi_{\AA Q_s g}(t-s).
\end{align*}
We know by Theorem~\ref{theogenerator}, with $L_0^Y=C_0(E)$,  that $\DD_{\AA} =  \DD_{\AA^Y}$ and since $Q_s \DD_{\AA}\subset \DD_{\AA}$  \citep[Chapter 1, \S 2]{dynkin1965markov}, we deduce from Corollary~\ref{coro46} that $A_3$ tends to $-\partial  \psi_{Q_s g}(t-s)/\partial t + \psi_{\AA Q_s g}(t-s)=0$ as $h\to 0$. Regarding $A_2$, note that $Q_s g\in   \DD_{\AA}$ implies that $\AA Q_s g\in C_0(E)$ \citep[Chapter 1, \S 2]{dynkin1965markov}, so that Lemma~\ref{marg2C0} applies and $A_2\to 0$ as $h\to 0$. Concerning $A_1$, using the linearity of $\psi_f(t)$ in $f$, we can write 
$$|A_1|=|\psi_{(Q_{s+h} - Q_s) g/h - \AA Q_s g}(t-s-h)|\leq \| (Q_{s+h} - Q_s) g/h - \AA Q_s g\|_\infty$$
which also tends to 0 as $h\to 0$ \citep[Chapter 1, \S 2]{dynkin1965markov}. We therefore obtain that  $v'_s=0$ and so     $v_t(x,n) =  (Q_t g \times \ind_\N)  (x,n) = \ck Q_t(g \times \ind_\N)(x,n) = v_0(x,n)$, proving \eqref{marg2} when $g\in \DD_{\AA^Y}$.

   Now let $g \in C_0(E).$ By our assumptions and Theorem~\ref{conditionsfeller},  $ (X_t)_{t \geq 0}$ is Feller, which implies that $C_0(E) = \overline{\DD_\AA}$ \citep{dynkin1965markov}. So there exists a sequence of functions $(g_p)_{p \geq 0}$ in $\DD_{\AA^Y}$ so that $\|g_p-g\|_{\infty} \to 0$ as $p\to\infty$. The two linear operators  $f \in \MM_b(E) \mapsto \ck Q_t(f \times \ind_\N)$ and $f \in \MM_b(E) \mapsto Q_t(f)$ being bounded, we can take the limit  in  \eqref{marg2} when applied to $g_p$ to get the same relation for $g \in C_0(E).$
   
    Finally take $U \subset E$ an open subset. Then define for any $p \geq 0$ the function
    $$ \phi_p : x \in E \mapsto \dfrac{d(x,E \backslash U)}{d(x,E \backslash U) + d(x,U_p)},  $$
    where $U_p = \{ y \in E, \, d(y,E \backslash U) \geq 1/p \}$. Then $\phi_p \in C_0(E)$ for any $p \geq 0$, and for any $x \in E$  we have $ \phi_p(x) \to \ind_U(x)$ as $p\to\infty$.  Taking the limit we obtain by the dominated convergence theorem the relation \eqref{marg2} for $g=\1_U$, concluding the proof of the second statement of Theorem~\ref{Marginals}.

\medskip

We finish by the proof of the third point of Theorem~\ref{Marginals}. Let $x\in E$ and $n\in\N$ such that $n(x)\leq n$. 
We show by induction on $k \geq 0$ that $\P_{(x,n)}(\ck C_{\ck T_k} \in \Gamma )=0.$ 
If $k=0$, then $\P_{(x,n)}(\ck C_{\ck T_0} \in \Gamma )=\ind_{(x,n) \in \Gamma } = 0.$ 
Suppose next that there exists $k \geq 0$ such that $\P_{(x,n)}(\ck C_{\ck T_k} \in \Gamma )=0$, then
\begin{align*}
     \P_{(x,n)}(\ck C_{\ck T_{k+1}} \in \Gamma )  &= \E_{(x,n)} \left [ \E_{(x,n)} \left ( \ind_{ \ck C_{\ck T_{k+1}} \in \Gamma} \left | \ck C_{\ck T_k}, \left ( \ck Y_t^{(k)} \right )_{t \geq 0} , \ck \t_{k+1}  \right. \right )   \right ] \\
    & = \E_{(x,n)} \left [ \ck K \left ( \ck Y^{(k)}_{\ck \t_{k+1}}, \Gamma \right )   \right ] \\ 
    & = \E_{(x,n)} \left [ \ck K \left ( \ck Y^{(k)}_{\ck \t_{k+1}}, \Gamma \right )  \ind_{ \ck Y^{(k)}_{\ck \t_{k+1}} \in \Gamma } \right ] + \E_{(x,n)} \left [ \ck K \left ( \ck Y^{(k)}_{\ck \t_{k+1}}, \Gamma \right )  \ind_{ \ck Y^{(k)}_{\ck \t_{k+1}} \in \Gamma^c } \right ] \\
    & = \E_{(x,n)} \left [ \ck K \left ( \ck Y^{(k)}_{\ck \t_{k+1}}, \Gamma \right )  \ind_{ \ck Y^{(k)}_{\ck \t_{k+1}} \in \Gamma } \right ] \; \; \; \textnormal{(by definition of } \ck K ) \\
    & = \E_{(x,n)} \left [ \ck K \left ( \ck Y^{(k)}_{\ck \t_{k+1}}, \Gamma \right )  \ind_{ \ck C_{\ck T_k} \in \Gamma } \right ] \; \; \; \textnormal{(for any } t \geq 0, \;  \ck C_{\ck T_k} \in \Gamma \Leftrightarrow \ck Y_t^{(k)} \in \Gamma )\\
    & \leq  \E_{(x,n)} \left [   \ind_{ \ck C_{\ck T_k} \in \Gamma } \right ] 
    =  \P_{(x,n)} \left [    \ck C_{\ck T_k} \in \Gamma \right ] =0,
\end{align*}
which proves the induction step. To conclude, 
recall that  $\P_{(x,n)}(  \ck N_t < \infty )=1$ and notice that due to the form of $\Gamma$ one has 
    $ \{ \ck C_t \in \Gamma \} = \{ \ck C_{\ck T_{\ck N_t} } \in \Gamma \} $ for any $t \geq 0$.
Then 
\begin{align*}
    \ck Q_t((x,n); \Gamma)  & =  \P_{(x,n)} (\ck C_t \in \Gamma)  = \P_{(x,n)} (\ck C_{\ck T_{\ck N_t} } \in \Gamma) 
          =  \sum_{k=0}^{\infty}  \P_{(x,n)} (\ck C_{\ck T_{k} } \in \Gamma, \ck N_t = k)=0.  
\end{align*}


%
%


\newpage 
 
 \appendix
 
 \renewcommand{\thefigure}{S\arabic{figure}}
\renewcommand{\theequation}{S\arabic{equation}}
\renewcommand{\thetheo}{S\arabic{theo}}
\renewcommand{\Box}{\qedsymbol}
\renewcommand{\thesection}{S-\arabic{section}}

\section*{\centering Supplementary material: proofs and additional results}

This supplementary material  contains the proofs of some results of the article. It also describes some topological properties of the space $E$ endowed with the distance $d_1$ in the case of interacting particles in $\R^d$, as introduced in Section~\ref{sec2.3}. All numbering and references in this supplementary material begin with the letter S, the other references referring to the main article.

\section{Proofs of Section \ref{sec2.4} about the Kolmogorov backward equation}\label{proof KBE}

\subsection{Proof of Theorem \ref{TheoKBEi}}

On the one hand, for any  $x\in E$, $t>0$ and $A \in \EE$,
\begin{align}
    \P_x(X_t \in A, \t_1 > t) & = \E_x(\ind_{X_t \in A} \, \ind_{\t_1 > t}) \nonumber\\
    &= \E_x \left [ \E_x( \ind_{X_t \in A} \, \ind_{\t_1 > t} | (Y_u^{(0)})_{u \geq 0}) \right ] \nonumber\\
    & = \E_x \left [ \E_x( \ind_{Y_t^{(0)} \in A} \, \ind_{\t_1 > t} | Y^{(0)}) \right ]  \nonumber\\
    &= \E_x \left [  \ind_{Y_t^{(0)} \in A} \, \E_x( \ind_{\t_1 > t}  | Y^{(0)})  \right ] \nonumber\\
    & = \E_x \left [ \ind_{Y_t^{(0)} \in A} \,\e^{ - \int_0^t \a(Y_u^{(0)}) \, \dd u }   \right ] \nonumber \\
    & = \E_x^Y \left [ \ind_{Y_t \in A} \, \e^{ - \int_0^t \a(Y_u)  \, \dd u }    \right ]. \label{Kolmopart1}
\end{align}
On the other hand, by construction of the process
\[\E_x[ \1_{X_t \in A}|\F_{\tau_1}]\1_{\tau_1 \leq t} = Q_{t-\tau_1}(X_{\tau_1},A)\1_{\tau_1 \leq t} \]
where $\F_{\tau_1} = \left \{ F \in \FF :   F \cap  \{ \tau_1 \leq t \} \in \FF_t, \; \forall \, t \geq 0   \right \}.$ Then
\begin{align}
    \P_x(X_t \in A, \t_1 \leq t)&=\E_x[\E_x[ \1_{X_t \in A}|\F_{\tau_1}]\1_{\tau_1 \leq t} ] \nonumber\\
    &=\E_x[Q_{t-\tau_1}(X_{\tau_1},A)\1_{\tau_1 \leq t}] \nonumber\\
    &=\E_x[\E_x[ Q_{t-\tau_1}(X_{\tau_1},A)\1_{\tau_1 \leq t}|\tau_1,Y^{(0)}]] \nonumber\\
    &= \E_x \left [ \int_{y\in E} K( Y^{(0)}_{\tau_1},dy)Q_{t-\tau_1}(y,A)\1_{\tau_1 \leq t}  \right ] \nonumber \\
    &= \E_x \left [ \E_x[\int_{y\in E} K( Y^{(0)}_{\tau_1},dy)Q_{t-\tau_1}(y,A)\1_{\tau_1 \leq t} | Y^{(0)}]  \right ] \nonumber\\
    &= \E_x \left [ \int_0^t \int_{y\in E} K( Y^{(0)}_{s},dy)Q_{t-s}(y,A) \a(Y_s^{(0)}) \e^{ - \int_0^s \a(Y_u^{(0)})  \, \dd u }  \, \dd s \right ] \nonumber\\
    &=\int_0^t  \int_E  Q_{t-s}(y,A)    \E_x \left [ K \left  (Y_s^{(0)},dy \right ) \,    \a(Y_s^{(0)}) \e^{ - \int_0^s \a(Y_u^{(0)})  \, \dd u }     \right ] \, \dd s  \nonumber \\
    & = \int_0^t  \int_E  Q_{t-s}(y,A)    \E_x^Y \left [ K \left  (Y_s,dy \right ) \,    \a(Y_s) \e^{- \int_0^s \a(Y_u) \, \dd u}    \right ] \, \dd s.  \label{Kolmopart2}
\end{align}
The result then follows gathering \eqref{Kolmopart1} and \eqref{Kolmopart2}.

\subsection{Proof of Proposition \ref{existuniciteKBE}}

The proof  is made up from Lemmas \ref{propSol}, \ref{lemminimal} and \ref{propStoch}, the approach being similar to \cite{feller1971}. In Lemma~\ref{propSol} we built a solution $ Q_{t,\infty}(x,A)$ of \eqref{kbei} for any $x \in E$ and $A \in \EE$, while  
 Lemmas~\ref{lemminimal} and \ref{propStoch} will imply the unicity of the solution. 

\begin{lem}\label{propSol}
For all $x \in E$ and $A \in \EE$, the function $t \in \R_+ \mapsto Q_{t,\infty}(x,A)$ is a solution of \eqref{kbei}.
\end{lem}
\begin{proof}
We will proceed as in the proof of Theorem \ref{TheoKBEi}.
First
\[ \P_x(X_t \in A , T_{p+1} >t, \t_1>t)= \P_x(X_t \in A , \t_{1} >t)=\E_x^Y \left [ \ind_{Y_t \in A} \, \e^{- \int_0^t \a(Y_u) \, \dd u}   \right ].\]
Secondly, if the process jumps once before $t$ (at time $\tau_1$) and is in $A$ at time $t$ with at most $p+1$ jumps, the process has at most $p$ jumps after the time $\tau_1.$ By construction of the process, the law of $X_t$ given $X_{T_1}$ is the same as the one of $X_{t - \tau_1}$ given $X_0=X_{\tau_1}.$ We then obtain
\[\E_x[\1_{X_t \in A} \1_{\tau_1 \leq t <T_{p+1}} | \F_{\tau_1}] = Q_{t-\tau_1,p}(X_{\tau_1},A)\1_{\tau_1 \leq t}.\]
 This leads to
\begin{align*}
    \P_x(X_t \in A, T_{p+1}>t, \t_1 \leq t)& =\E_x[Q_{t-\tau_1,p}(X_{\tau_1},A)\1_{\tau_1 \leq t}]\\
    & = \E_x[ \E_x[Q_{t-\tau_1,p}(X_{\tau_1},A)\1_{\tau_1 \leq t}|Y^{(0)},\tau_1]]\\
    &= \E_x[ \int_{y \in E} Q_{t-\tau_1,p}(y,A)\1_{\tau_1 \leq t} K( Y^{(0)}_{\tau_1},dy ) ]\\
  &=  \E_x[ \E_x[\int_{y \in E} Q_{t-\tau_1,p}(y,A)\1_{\tau_1 \leq t} K( Y^{(0)}_{\tau_1},dy ) |Y^{(0)} ]]\\
  &= \E_x[ \int_0^t \int_{y \in E} Q_{t-s,p}(y,A)\1_{\tau_1 \leq t} K( Y^{(0)}_{s},dy ) \a(Y_s^{(0)}) \e^{ - \int_0^s \a(Y_u^{(0)})  \, \dd u }  ]]\\
  &= \int_0^t  \int_E  Q_{t-s,p}(y,A)     \E_x^Y \left [ K \left (Y_s,dy \right )  \,    \a(Y_s) \e^{ - \int_0^s \a(Y_u)  \, \dd u }  \right ] \,  \dd s.
\end{align*}

We then obtain the induction formula
\begin{align*}
     Q_{t,p+1}(x,A) & = \E_x^Y \left [ \ind_{Y_t \in A} \, \e^{- \int_0^t \a(Y_u) \, \dd u}   \right ]  + \int_0^t  \int_E  Q_{t-s,p}(y,A)     \E_x^Y \left [ K \left (Y_s,dy \right )  \,    \a(Y_s) \e^{- \int_0^s \a(Y_u) \, \dd u}    \right ] \,  \dd s.
\end{align*}
This leads  by monotone convergence to
\begin{align*}
    Q_{t,\infty}(x,A) & = \E_x^Y \left [ \ind_{Y_t \in A} \, \e^{- \int_0^t \a(Y_u) \, \dd u}   \right ]  + \int_0^t  \int_E  Q_{t-s,\infty}(y,A)     \E_x^Y \left [  K \left (Y_s,dy \right ) \,    \a(Y_s) \e^{- \int_0^s \a(Y_u) \, \dd u}    \right ] \,  \dd s.
\end{align*}
\end{proof}

\begin{lem}\label{lemminimal}
$Q_{t,\infty}$ is called the minimal solution of \eqref{kbei} in the sense that for any non-negative solution $Q_t$ of \eqref{kbei}, we have $Q_t \geq Q_{t,\infty}.$ 
\end{lem}
\begin{proof}
Let $Q_t$ be a non-negative solution of \eqref{kbei}. Then for any $x \in E$ and $A \in \EE$
\[Q_t(x,A) \geq Q_{t,1}(x,A) = \E_x^Y \left [ \ind_{Y_t \in A} \, \e^{- \int_0^t \a(Y_u) \, \dd u}   \right ] . \]
We then proceed by induction. If $Q_t \geq Q_{t,p}$  then
\begin{align*}
    Q_t(x,A) & = \E_x^Y \left [ \ind_{Y_t \in A} \,\e^{ - \int_0^t \a(Y_u) \, \dd u  }   \right ] 
     + \int_0^t  \int_E  Q_{t-s}(y,A)    \E_x^Y \left [ K \left  (Y_s,dy \right ) \,    \a(Y_s) \e^{- \int_0^s \a(Y_u) \, \dd u}    \right ] \, \dd s \\
    & \geq  \E_x^Y \left [ \ind_{Y_t \in A} \,\e^{ - \int_0^t \a(Y_u) \, \dd u  }   \right ] 
    + \int_0^t  \int_E  Q_{t-s,p}(y,A)    \E_x^Y \left [ K \left  (Y_s,dy \right ) \,    \a(Y_s) \e^{- \int_0^s \a(Y_u) \, \dd u}    \right ] \, \dd s \\
     & = Q_{t,p+1}(x,A).
\end{align*}
Finally $Q_t(x,A) \geq Q_{t,p}(x,A)$ for every $p\geq 1$ and the result follows by letting $p$ go to infinity.
\end{proof}

\begin{lem}\label{propStoch}
The minimal solution $Q_{t,\infty}$ is stochastic, i.e. $Q_{t,\infty}(x,E)=1.$
\end{lem}
\begin{proof}
 Recall that $\a$ is bounded by $\a^* > 0$. 
It is then enough to show by induction that $Q_{t,p}(x,E) \geq 1-(1-\textnormal{e}^{- \a^*t})^p$ for any $p \geq 1.$ 
First
\begin{align*}
    Q_{t,1}(x,E) & = \E_x^Y \left [ \ind_{Y_t \in E} \,\e^{ - \int_0^t \a(Y_u) \, \dd u  }   \right ]  = \E_x^Y \left [  \,\e^{ - \int_0^t \a(Y_u) \, \dd u  }   \right ]  \geq  \E_x \left ( \e^{-\a^*t}   \right )
    = \e^{-\a^*t}.
\end{align*}
Then notice that
\begin{align*}
    \P_x( \tau_1 \leq t)  & =  \E_x^Y \left [  1- \e^{ - \int_0^t \a(Y_u) \, \dd u  }   \right ]
    \leq  1- \e^{- \a^* t}.
\end{align*}
We then obtain by induction
\begin{align*}
     Q_{t,p+1}(x,E) & = \E_x^Y \left [ \ind_{Y_t \in E} \, \e^{- \int_0^t \a(Y_u) \, \dd u}   \right ]  + \int_0^t  \int_E  Q_{t-s,p}(y,A)     \E_x^Y \left [ K \left (Y_s,dy \right )  \,    \a(Y_s) \e^{- \int_0^s \a(Y_u) \, \dd u}    \right ] \, \dd s \\
    & \geq   \E_x^Y \left [ \, \e^{- \int_0^t \a(Y_u) \, \dd u}   \right ] + \int_0^t \left ( 1-(1-\textnormal{e}^{- \a^*(t-s)})^p \right )    \E_x^Y \left [ \int_E K \left (Y_s,dy \right )  \,    \a(Y_s) \e^{- \int_0^s \a(Y_u) \, \dd u}    \right ] \dd s  \\ 
    &  \geq    \E_x^Y \left [ \, \e^{- \int_0^t \a(Y_u) \, \dd u}   \right ] + \int_0^t \left ( 1-(1-\textnormal{e}^{- \a^*t})^p \right )    \E_x^Y \left [  \a(Y_s) \e^{- \int_0^s \a(Y_u) \, \dd u}    \right ] \dd s  \\ 
    & =  \P_x( \tau_1 > t)    + \left ( 1-(1-\textnormal{e}^{- \a^*t})^p \right ) \, \P_x( \tau_1 \leq t) \\
    & = 1 - (1-\textnormal{e}^{- \a^*t})^p \, \P_x( \tau_1 \leq t) \geq 1-(1-\textnormal{e}^{- \a^*t})^{p +1}.
\end{align*} 
\end{proof}

By bringing together the last three lemmas it does not take long to prove Proposition \ref{existuniciteKBE}.
   By Lemma \ref{propSol}, $Q_{t,\infty}$ is a solution of \eqref{kbei}. We now prove the unicity.  Let $Q_t$ be a non-negative sub-stochastic solution of \eqref{kbei}. Lemma \ref{lemminimal} entails $Q_t(x,A) \geq Q_{t,\infty}(x,A)$ and $Q_t(x,E \backslash A) \geq Q_{t,\infty}(x,E \backslash A)$ for every $A \in \mathcal{E}.$ We get then
\[    1 \geq Q_t(x,E) = Q_t(x,A)+Q_t(x,E \backslash A)   
      \geq   Q_{t,\infty}(x,A)+Q_{t,\infty}(x,E \backslash A) =  Q_{t,\infty}(x,E) =1\]
by Lemma \ref{propStoch} so $Q_t(x,A) = Q_{t,\infty}(x,A)$ for every $A \in \mathcal{E}$.

\subsection{Proof of Proposition \ref{prop13mini}}

We first show that $Q_{t,(\infty)}$ is a solution of \eqref{kbei}. Let $n \geq 0$, $x \in E_n$ and $p \geq n$.
If there is no jump before $t,$ then
\[ \P_x(X_t \in A, \t_1 > t,\,  \forall s\in[0,t] \ n(X_s)\leq p )=\P_x(X_t \in A, \t_1 > t).  \]

By construction of the process, if the first jump before $t$ is a death,
\[ \P_x(X_t \in A ,\,  \forall s\in[0,t] \ n(X_s)\leq p   \, | \, \F_{\t_1}, \textrm{ a death occurs at } \t_1)  =  Q_{t-\t_1,(p)}(X_{\t_1},A),\]
and if the first jump before $t$ is a birth,
\[ \P_x(X_t \in A ,\,  \forall s\in[0,t] \ n(X_s)\leq p  \, | \, \F_{\t_1}, \textrm{ a birth occurs at } \t_1)  =  Q_{t-\t_1,(p)}(X_{\t_1},A) \1_{p >n}.\]

Following the same computations as in the proof of Theorem \ref{TheoKBEi}, we obtain
\begin{multline*}
     Q_{t,(p)}(x,A) = \E_x^Y \left [ \ind_{Y_t \in A} \, \e^{ - \int_0^t \a(Y_u) \, \dd u }   \right ]  \\ + \int_0^t  \int_{E_{n+1}} Q_{t-s,(p)}(y,A) \,    \E_x^Y \left [ \b \left (Y_s \right ) K_\b \left ( Y_s  ,dy \right )  \, \e^{ - \int_0^s \a(Y_u) \, \dd u } \right ] \dd s  \, \1_{p > n} \\
     + \int_0^t  \int_{E_{n-1}} Q_{t-s,(p)}(y,A) \,  \E_x^Y \left [ \d \left (Y_s \right ) K_\d \left ( Y_s  ,dy \right )  \, \e^{ - \int_0^s \a(Y_u) \, \dd u } \right ] \dd s ,
\end{multline*}
and $Q_{t,(\infty)}(x,A)$ satisfies \eqref{kbei} by continuity of the probability.
The proof is then complete thanks to the unicity of the solution to \eqref{kbei}.

\section{Proofs of  Section \ref{sec3.1} about Feller properties}\label{proofs Feller}

\subsection{Proof of Proposition \ref{prop24}}

Both results of the proposition are based on the following calculation, for any $f \in M^b(E)$:

\begin{align*}
    Q_tf(x)-f(x) & = \E_x^Y \left [ f(Y_t) \e^{- \int_0^t \a(Y_u) \, \dd u}  \right ] - f(x) + \E_x \left [ f(X_t) \ind_{N_t \geq 1} \right ] \\
    & = Q_t^Yf(x)- f(x) + \E_x^Y \left [ f(Y_t) \left ( \e^{- \int_0^t \a(Y_u) \, \dd u} -1 \right ) \right ] + \E_x \left [ f(X_t) \ind_{N_t \geq 1} \right ].
\end{align*}
The last two terms goes uniformly to $0$ when $t\to 0$. Indeed,

\begin{align*}
    \left | \E_x^Y \left [ f(Y_t) \left ( \e^{- \int_0^t \a(Y_u) \, \dd u} -1 \right ) \right ] + \E_x \left [ f(X_t) \ind_{N_t \geq 1} \right ] \right | & \leq ||f||_{\infty} \a^* t + ||f||_{\infty} \P_x(N_t \geq 1) \\
    & = ||f||_{\infty} \a^* t + ||f||_{\infty}  \E_x^Y \left [ \left (1- \e^{- \int_0^t \a(Y_u) \, \dd u}\right )\right ] \\
    & \leq 2 \a^* t ||f||_{\infty}.
\end{align*}

So we obtain directly the second point of the proposition. For the first point remark that when  $f \in C_b(E)$, by continuity of $f \circ Y$  and the dominated convergence theorem, $\lim_{t \rightarrow 0} Q_t^Yf(x) = f(x)$.

\subsection{Proof of Theorem \ref{conditionsfeller} (part 1)}

The proof of the  Feller continuous property of $(X_t)_{t\geq 0}$ is based on the following Lemma~\ref{lem23} that exploits the Feller continuous property of $Q_t^Y$, and on Lemma~\ref{lem24} which  in addition makes use of the Feller continuous property of the jump kernel $K$. 

\begin{lem} \label{lem23}  Assume that for any $t\geq 0$, $Q_t^Y \, C_b(E) \subset C_b(E)$. Then for any $p \geq 1$, $f_1, \dots f_p \in C_b(E)$ and $0\leq t_1  < \dots < t_p$ the function 
    $x \mapsto \E_x^Y \left [ f_1 (Y_{t_1})  \dots f_p (Y_{t_p}) \right ]$ is continuous. Furthermore, for any $f \in C_b(E)$ the function
    $x \mapsto \E_x^Y [ f (Y_{t})  \e^{ - \int_0^{t} \alpha( Y_u) \, \dd u } ]$ 
    is continuous. 
\end{lem}

\begin{proof}
To prove the first statement, we proceed first  by induction on $p \geq 1$.  
    Since $ x \mapsto \E_x^Y \left [ f_1 (Y_{t_1}) \right ]= Q_{t_1}^Y f_1(x),$ the property is satisfied for $p=1$ because $Q_t^Y \, C_b(E) \subset C_b(E)$ for any $t\geq 0$ by assumption. Suppose now that the property is true for some $p \geq 1$. Let $f_1,  \dots, f_{p+1} \in C_b(E)$ and $0\leq t_1 <   \dots < t_{p+1}$. Then
    \begin{align*}
        \E_x^Y \left [  f_1 (Y_{t_1})  \dots f_{p+1} (Y_{t_{p+1}}) \right ] & = \E_x^Y \left [ \E_x^Y \left ( f_1 (Y_{t_1})  \dots f_{p+1} (Y_{t_{p+1}}) \left | Y_{t_1}, \dots, Y_{t_p}  \right. \right ) \right ] \\
        & = \E_x^Y \left [  f_1 (Y_{t_1})  \dots  f_p (Y_{t_p}) \E_x^Y \left ( f_{p+1} (Y_{t_{p+1}}) | Y_{t_p}  \right ) \right ] \\
        & =  \E_x^Y \left [  f_1 (Y_{t_1}) \dots  f_p (Y_{t_p}) Q^Y_{t_{p+1}-t_p}f_{p+1}(Y_{t_p}) \right ] .
    \end{align*}
    The function $f_p \times Q^Y_{t_{p+1}-t_p}f_{p+1}$ is continuous by assumption so we can apply the induction hypothesis. 

Regarding the second statement of the lemma, let us take $f \in C_b(E)$ and $t \geq 0$.
  We have
    \begin{align*}
        \hspace{- 1cm} \E_x^Y \left [ f(Y_t) \e^{ - \int_0^t \a(Y_u) \, \dd u } \right ] & = \E_x^Y \left [ f(Y_t) \sum_{k \geq 0} \dfrac{(-1)^k}{k!} \left (  \int_0^t \a(Y_u) \, \dd u \right )^k \right ] \\
        & = \sum_{k \geq 0} \dfrac{(-1)^k}{k!} \E_x^Y \left [  f(Y_t) \left (  \int_0^t \a(Y_u) \, \dd u \right )^k  \right ]  \\
        & = \sum_{k \geq 0} \dfrac{(-1)^k}{k!} \E_x^Y \left [  f(Y_t) \prod_{j=1}^k  \int_0^t \a(Y_{u_j}) \, \dd u_j   \right ]  \\
        & = \sum_{k \geq 0} \dfrac{(-1)^k}{k!} \int_{u_1=0}^t \dots \int_{u_k=0}^t  \E_x^Y \left [  f(Y_t)  \a(Y_{u_1})  \dots \a(Y_{u_k}) \right ]  \, \dd u_1 \dots \dd u_k 
    \end{align*}
   which is valid because $f \times \a^k$ is bounded.
    For any $u_1 \geq 0, \dots, u_k \geq 0$, the function $ x \in E \mapsto \E_x^Y \left [  f(Y_t)  \a(Y_{u_1})  \dots \a(Y_{u_k}) \right ]  $ is continuous by the first part of the proof and this expression is bounded uniformly in $x$ by $\|f\|_{\infty} \times (\a^*)^k \in L^1([0,t]^k). $  Again, by normal convergence, we obtain the expected result.
\end{proof}

\begin{lem}\label{lem24}
 Assume that $Q_t^Y \, C_b(E) \subset C_b(E)$ for any $t\geq 0$ and that $K \, C_b(E) \subset C_b(E)$. Let $t >0.$ Then for any  $k \geq 1,$ for any bounded measurable function $\varphi$ on $E \times \R_+$ such that $\varphi(.,u)$ is continuous for any $ u \leq t$,  
the function $x \mapsto \E_x[ \varphi(X_{T_k},T_k)\1_{T_k \leq t}]$ 
is continuous.
 \end{lem}
 
 \begin{proof}
  We shall proceed by induction. For $k=1,$
  \begin{align*}
   \E_x[\varphi(X_{T_1},T_1)\1_{T_1 \leq t}] &= \E_x[\1_{T_1 \leq t} \, \E_x [\varphi(X_{T_1},T_1)|Y^{(0)},T_1]] \\
   & = \E_x[ \int_E K(Y^{(0)}_{T_1} ,dz) \varphi( z,T_1) \1_{T_1 \leq t} ]\\ 
   & = \E_x[  \int_E \E_x [ K(Y^{(0)}_{T_1} ,dz) \varphi( z,T_1) \1_{T_1 \leq t} | Y^{(0)} ] ]  \\
   &= \E_x[\int_E \int_0^t  K(Y^{(0)}_{t_1} ,dz) \varphi( z,t_1) \alpha(Y^{(0)}_{t_1}) e^{-\int_0^{t_1} \alpha(Y_u^{(0)})\, \dd u} \, \dd t_1]\\
   &= \int_0^t \E_x^Y[ H(Y_{t_1},t_1)\alpha(Y_{t_1}) e^{-\int_0^{t_1} \alpha(Y_u)\, \dd u}] \, \dd t_1
  \end{align*}
where $H(x,u)= \int_E K(x,dz )\varphi(z,u).$ Since $z \mapsto \varphi(z,t_1)$ belongs to $C_b(E)$ for every $t_1\leq t$, the Feller continuous property of $K$ entails the continuity of $x \mapsto H(x,t_1)$ for every $t_1 \leq t.$ Consequently the function  
\[x \mapsto \E_x^Y[ H(Y_{t_1},t_1)\alpha(Y_{t_1}) e^{-\int_0^{t_1} \alpha(Y_u)\, \dd u}]\]
is continuous for every $t_1$ by Lemma \ref{lem23}. The functions $H$ and $\alpha$ being bounded, the dominated convergence theorem yields the continuity of $x \mapsto  \E_x[\varphi(X_{T_1},T_1)\1_{T_1 \leq t}]$, proving the statement for $k=1$. Assume now that the property holds for $k \geq 1.$ We compute similarly
\begin{align*}
 \E_x[ \varphi(X_{T_{k+1}},T_{k+1})\1_{T_{k+1}\leq t}] & = \E_x[\E_x[\int_E K(Y^{(k)}_{T_{k+1}-T_k},dz) \varphi( z,T_{k+1})\1_{T_{k+1}\leq t}|\mathcal{F}_{T_k},Y^{(k)}]]\\
 &= \E_x[\E^Y_{X_{T_k}}[\int_0^{t-T_k} \int_E K(Y_{\tau},dz)\varphi(z,\tau+T_k) \alpha( Y_{\tau})e^{-\int_0^{\tau}\alpha(Y_u)\, \dd u}] \1_{T_k \leq t}]\\
 &= \E_x[\tilde{\varphi}( X_{T_k},T_k)\1_{T_k\leq t}],
\end{align*}
where 
\begin{align*}
 \tilde{\varphi}(x,u) &= \E_x^Y[ \int_0^{t-u} \int_E K(Y_{\tau},dz) \varphi(z,\tau+u) \alpha(Y_{\tau}) e^{-\int_0^{\tau} \alpha(Y_u) \, \dd u} \, \dd \tau]\\
 &= \int_0^{t-u} \E_x^Y[ H(Y_{\tau},\tau+u) \alpha(Y_{\tau})e^{-\int_0^{\tau} \alpha(Y_u) \, \dd u} ] \, \dd \tau.
\end{align*}
By Lemma \ref{lem23}, $x \mapsto \E_x^Y[ H(Y_{\tau},\tau+u) \alpha(Y_{\tau})e^{-\int_0^{\tau} \alpha(Y_u) \, \dd u} ]$ is continuous for each $u,\tau,$ so $\tilde{\varphi}(.,u)$ is continuous for every $u \leq  t.$ We then obtain the result applying the induction hypothesis.
 \end{proof}

 We are now in position to prove the first part of Theorem~\ref{conditionsfeller} about the Feller continuous property of  $(X_t)_{t \geq 0}$.
    We compute for $t >0,$ $x \in E$ and $f \in C_b(E)$
    \begin{align*}
        Q_tf(x) &= \sum_{k=0}^{\infty}\limits \E_x[f(X_t)\1_{N_t=k}]\\
        &= \E_x^Y[f(Y_t)e^{-\int_0^t \alpha(Y_u)\, \dd u}]+ \sum_{k\geq 1} \E_x[f(X_t)\1_{T_k \leq t < T_{k+1}}]\\
        &= \psi(x,t)+ \sum_{k\geq 1}\E_x[f(X_t)\1_{T_{k+1}-T_k >t-T_k}\1_{T_k \leq t}]
    \end{align*}
    where $\psi(x,t) = \E_x^Y[f(Y_t)e^{-\int_0^t \alpha(Y_u)\, \dd u}].$ We get from Lemma \ref{lem23} that $\psi(.,t)$ belongs to $C_b(E)$ for every $t>0.$ Then
    \begin{align}
        \E_x[f(X_t)\1_{T_{k+1}-T_k >t-T_k}\1_{T_k \leq t}] &= \E_x[ \1_{T_k \leq t} f(Y^{(k)}_{t-T_k}) \, \E_x [\1_{T_{k+1}-T_k >t-T_k}| \FF_{T_k}, Y^{(k)} ]] \nonumber \\
        & = \E_x[f(Y^{(k)}_{t-T_k}) \e^{-\int_0^{t-T_k}\alpha( Y^{(k)}_u)\, \dd u}  \1_{T_k \leq t}] \\
        & = \E_x[\E_{X_{T_k}}^Y[f(Y_{t-T_k}) e^{-\int_0^{t-T_k}\alpha( Y_u)\, \dd u}]\1_{T_k\leq t}]\label{Qtfegalpsi}\\
        &= \E_x[\psi(X_{T_k},t-T_k)\1_{T_k\leq t}],\nonumber
    \end{align}
    so Lemma \ref{lem24} entails that $x\mapsto \E_x[\psi(X_{T_k},t-T_k)\1_{T_k\leq t}]$ is continuous for every $k \geq 1.$ The domination 
    \begin{align*}
        \left|\E_x[\psi(X_{T_k},t-T_k)\1_{T_k\leq t}]\right|& \leq \|f\|_{\infty} \P_x( T_k \leq t)\\
        &\leq \|f\|_{\infty} \P(N^*_t \geq k )
    \end{align*}
    where $N^*t \sim \mathcal{P}(\alpha^*t )$ (by \eqref{dominationpoisson}) allows us to conclude that $x \mapsto Q_tf(x)$ is continuous.

\subsection{Proof of Theorem \ref{conditionsfeller} (part 2)}

Our aim is to prove the Feller property of $(X_t)_{t\geq 0}$ assuming that for every $t >0$, $ Q_t^Y C_0(E) \subset C_0(E)$ and that $K \, C_0(E) \subset C_0(E)$. We follow the same steps as for the proof of Theorem \ref{conditionsfeller} (part 1), by first inspecting the consequences of $ Q_t^Y C_0(E) \subset C_0(E)$ in Lemma~\ref{csqdeH3} and second the additional effect of $K \, C_0(E) \subset C_0(E)$  in Lemma~\ref{lemC0}.

\begin{lem} \label{csqdeH3}
    Suppose that for every $t >0$, $Q_t^Y C_0(E) \subset C_0(E)$. Then 
    \begin{enumerate}
        \item  for any $f \in C_0(E)$, $\lim_{t \rightarrow 0} \|Q_t^Yf - f\|_{\infty} = 0$,
        \item for any $t>0$, $\sup_{s \in [0,t]} Q_s^Y C_0(E) \subset C_0(E)$,
        \item for any $f \in C_0(E)$ the function
    $x \mapsto \E_x^Y [ f (Y_{t})  \e^{ - \int_0^{t} \alpha( Y_u) \, \dd u } ]$ 
    is continuous. 
    \end{enumerate}
\end{lem}

\begin{proof} 
     By continuity of $(Y_t)_{t \geq 0}$, $\lim_{t \rightarrow 0}\limits Q_t^Yf(x) = f(x)$ for every $f \in C_0(E)$ and every $x \in E$. As proved  in \cite{RevuzYor1991}, this is equivalent when $Q_t^Y C_0(E) \subset C_0(E)$  to $\lim_{t \rightarrow 0} \|Q_t^Yf - f\|_{\infty} = 0$, which proves the first statement of the lemma.
      
 Concerning the second property, let $\varepsilon >0$ and $f \in C_0(E).$ Fix $\eta(f) >0$ such that for every $s < \eta(f),$ $\|Q_s^Y f - f\|_{\infty} \leq \varepsilon$ and $s(x) \in [0,t]$ satisfying $\sup_{s \in [0,t]} Q_s^Yf(x) = Q^Y_{s(x)}f(x).$ Then we have

\[ Q^Y_{\frac{\lfloor 2^n s(x) /t\rfloor t}{2^n}}f(x) \leq \max_{k=0,...,2^n}\limits Q^Y_{\frac{kt}{2^n}}f(x) \leq \sup_{s \in [0,t]}\limits Q_s^Yf(x).\] 

So 
\begin{align*}
 \left| \sup_{s \in [0,t]}\limits Q_s^Yf(x) -  \max_{k=0,...,2^n}\limits Q^Y_{\frac{kt}{2^n}}f(x)\right|& \leq \left| \sup_{s \in [0,t]}\limits Q_s^Yf(x) - Q^Y_{\frac{\lfloor 2^n s(x) /t\rfloor t}{2^n}}f(x) \right|\\
 & =  \left| Q^Y_{s(x)}f(x) - Q^Y_{\frac{\lfloor 2^n s(x) /t\rfloor t}{2^n}}f(x) \right|\\
 & =  \left| Q^Y_{\frac{\lfloor 2^n s(x) /t\rfloor t}{2^n}} (Q^Y_{s(x) - \frac{\lfloor 2^n s(x) /t\rfloor t}{2^n}}f(x) - f(x)  ) \right|\\
  &\leq \|Q^Y_{s(x) - \frac{\lfloor 2^n s(x) /t\rfloor t}{2^n}}f -f\|_{\infty}\\
  &\leq \varepsilon
\end{align*}
whenever $t 2^{-n} \leq \eta(f).$ This leads to $\lim_{n \rightarrow \infty} \|  \sup_{s \in [0,t]} Q_s^Yf -  \max_{k=0,...,2^n} Q^Y_{kt 2^{-n}}f \|_{\infty}=0.$ 
Since $ \max_{k=0,...,2^n} Q^Y_{kt 2^{-n}}f \in C_0(E)$ for $f \in C_0(E)$ by assumption and $C_0(E)$ is a  closed subset of $\mathcal{M}_b(E)$ for $\|.\|_{\infty}$, we deduce that $\sup_{s \in [0,t]} Q_s^Yf \in C_0(E).$

We finally prove the third point of the lemma in a similar way as in the proof of Lemma \ref{lem23}. First we show by induction on $p \geq 1$ that for any $t \geq 0$ and $0 \leq u_1 \leq \dots \leq u_p \leq t$ and $f \in C_0(E)$ the function $x \mapsto \E_x^Y \left [ f (Y_{t}) \a (Y_{u_1}) \dots \a (Y_{u_p}) \right ]$ is in $C_0(E)$. Indeed for $p = 1$
    \begin{equation*}
        \E_x \left [ f (Y_{t}) \a (Y_{u_1})  \right ] = \E_x \left [ \a (Y_{u_1}) \E_x \left [  f (Y_{t}) | \FF_{u_1} \right ]  \right ] = \E_x \left [ \a (Y_{u_1}) Q_{t-u_1}^Y f(Y_{u_1})\right ] = Q_{u_1}^Y \left ( \a \times Q_{t-u_1}^Y f \right )(x)
    \end{equation*} 
   and $Q_{u_1}^Y \left ( \a \times Q_{t-u_1}^Y f \right ) \in C_0(E)$ by  assumption. For the induction step we just write 
    \begin{align*}
        \E_x^Y \left [ f (Y_{t}) \a (Y_{u_1}) \dots \a (Y_{u_p}) \a (Y_{u_{p+1}})\right ] & =  \E_x^Y \left [  \a (Y_{u_1}) \dots \a (Y_{u_p}) ( Q_{t-u_{p+1}}^Y f \times \a) (Y_{u_{p+1}})\right ]
    \end{align*} 
    that is in $C_0(E)$ by assumption and the induction hypothesis. 
    We then obtain the continuity of the function
    \[x \mapsto \E_x^Y [ f (Y_{t})  \e^{ - \int_0^{t} \alpha( Y_u) \, \dd u } ]\]
   similarly as in the proof of Lemma \ref{lem23}.
\end{proof}

\begin{lem}\label{lemC0}
 Assume that for every $t >0$, $ Q_t^Y C_0(E) \subset C_0(E)$ and that $K \, C_0(E) \subset C_0(E)$.  Let $t>0$. Then for every $k \geq 1$ and all $g \in C_0(E)$, $x \mapsto \E_x[g(X_{T_k})\1_{T_k \leq t}]$ 
 vanishes at infinity.
\end{lem}

\begin{proof}

 Let us prove the result by induction. For $k=1,$
 \begin{align*}
  \left| \E_x [ g(X_{T_1})\1_{T_1\leq t}]\right|& = \left| \int_0^t \E_x[ Kg(Y_s)\alpha(Y_s)e^{-\int_0^s \alpha(Y_u)\, \dd u}] \, \dd s\right|\\
  &\leq \alpha^* \int_0^t \E_x[K|g|(Y_s)]\, \dd s\\
  &\leq \alpha^* t\sup_{s \in [0,t]}Q_s^Y K|g|(x).
 \end{align*}
Since $K \, C_0(E) \subset C_0(E)$, the function $K|g|$ belongs to $C_0(E)$,  so $\sup_{s \in [0,t]}Q_s^Y K|g| \in C_0(E)$ by Lemma~\ref{csqdeH3}. This entails in particular that $x \mapsto \E_x [ g(X_{T_1})\1_{T_1\leq t}]$ vanishes at infinity. Let now $k\geq 1$ and assume that $x \mapsto  \E_x[g(X_{T_k})\1_{T_k \leq t}])$ vanishes at infinity.  We compute similarly
\begin{align*}
 \E_x[g(X_{T_{k+1}})\1_{T_{k+1} \leq t}] &= \E_x[\E^Y_{X_{T_k}} [\int_0^{t-T_k} \int_E K(Y_s,dz)g(z) \alpha(Y_s)e^{-\int_0^{s}\alpha(Y_u)\, \dd u}]\1_{T_k\leq t}]
\end{align*}
and
\begin{align*}
 \left|\E_x[g(X_{T_{k+1}})\1_{T_{k+1} \leq t}] \right|&\leq \alpha^*\E_x[\E_{X_{T_k}}[\int_0^t K|g|(Y_s) \, \dd s]\1_{T_k\leq t}]\\
 &= \alpha^* \E_x[\int_0^t Q_s^Y K|g|(X_{T_k}) \, \dd s \1_{T_k\leq t}]\\
 &\leq \alpha^* t \E_x[ \sup_{s \in [0,t]}Q_s^Y K|g| (X_{T_k})\1_{T_k \leq t}].
\end{align*}
Since $\sup_{s \in [0,t]}Q_s^Y K|g| \in C_0(E),$ the result follows from the induction hypothesis.
\end{proof}

In order to prove Theorem \ref{conditionsfeller} (part 2), first remark that $x \mapsto Q_t f(x)$ is continuous for $f \in C_0(E)$ and any $t\geq 0$. 
This follows by the same arguments as in the proof of Theorem \ref{conditionsfeller} (part 1) taking $f \in C_0(E)$. 
Indeed, using the same notation as in the proof of Lemma~\ref{lem24}, we obtain that the function $H(.,u)$ belongs to $C_0(E)$ for any $u \leq t$ by the assumption  $K \, C_0(E) \subset C_0(E)$ and Lemma~\ref{csqdeH3} (item 3.). The conclusion of Lemma~\ref{lem24} then follows by the same proof, using Lemma~\ref{csqdeH3} instead of Lemma~\ref{lem23}. Similarly, the proof of Theorem \ref{conditionsfeller}  (part 1) with the same substitution entails that $Q_t f \in C_b(E)$. 
   
The strong continuity of $Q_t$ follows by Proposition~\ref{prop24} and the first statement of Lemma~\ref{csqdeH3}. 

It remains to prove that $x \mapsto Q_t f(x)$ vanishes at infinity.  By the same decomposition of $Q_t f$ as in the proof of Theorem \ref{conditionsfeller}  (part 1), we obtain using in particular \eqref{Qtfegalpsi}  that for any $j \geq 1$ 
    \begin{equation}\label{decomp Qt}
    |Q_tf(x)| \leq Q_t^Y|f|(x) + \sum_{k=1}^j \E_x[\sup_{s \in [0,t]}Q^Y_{s}|f|(X_{T_k})\1_{T_k \leq t}] + \|f\|_{\infty}\P(N^*_t \geq j)
    \end{equation}
    where $N^*_t \sim \mathcal{P}(\alpha^* t).$ Let $\varepsilon >0.$ First, $Q_t^Y|f| \in C_0(E)$ by assumption, so that $Q_t^Y|f|(x) \leq \varepsilon/3$ for $x$ outside a compact set. Second,     
    since $\lim_{j \rightarrow \infty}\P(N^*_t \geq j)=0,$ there exists $j_0 \geq 1$ such that $  \|f\|_{\infty}\P(N^*_t \geq j) \leq \varepsilon /3$. Third, Lemma \ref{lemC0} entails that for every $ k \leq j_0$ the function 
    \[x \mapsto \E_x[\sup_{s \in [0,t]}Q^Y_{s}|f|(X_{T_k})\1_{T_k \leq t}]\] 
    vanishes at infinity because $\sup_{s \in [0,t]}Q^Y_{s}|f| \in C_0(E)$ by Lemma~\ref{csqdeH3}. It is therefore bounded by $\epsilon/j_0$  for $x$ outside a compact set. Combining these three results in \eqref{decomp Qt} concludes the proof.

\section{Proof of Theorem~\ref{theogenerator} about the infinitesimal generator}\label{proof generator}

Let $f \in L_0^Y$, $x \in E$ and $h >0$. We decompose $\frac{1}{h}(Q_hf(x)-f(x))$ as
     \begin{align*}
        \dfrac{1}{h}  \left ( Q_hf(x) - f(x) \right )
          &=  \dfrac{1}{h} \left ( \E_{x}^Y \left [ f(Y_h) \e^{- \int_0^h \a( Y_u) \, \dd u} \right ] - f(x) + \E_{x} \left [ f(X_h) \ind_{N_h=1}  \right]  + \E_{x} \left [ f(X_h) \ind_{N_h \geq 2} \right ] \right ) \\
           & = \E_x^Y \left [ \dfrac{f(Y_h)-f(x)}{h} \right ] + T(x),
      \end{align*}
    where
    \begin{multline}\label{defT}
    T(x) = -  \dfrac{1}{h} \E_{x}^Y \left [ f(Y_h) \int_0^h  \a( Y_u) \, \dd u \right ]  + \dfrac{1}{h} \E_{x}^Y \left [ f(Y_h) \left ( \e^{- \int_0^h  \a( Y_u) \, \dd u} -1 + \int_0^h \a( Y_u) \, \dd u \right ) \right ] \\
           + \dfrac{1}{h} \E_{x} \left [ f(X_h) \ind_{N_h=1}  \right] + \dfrac{1}{h}  \E_{x} \left [ f(X_h) \ind_{N_h \geq 2}  \right].\end{multline}
   To prove the theorem, we thus need to show that  for any $f \in L_0^Y$    
   \[ \underset{x \in E}{\sup} \left |   T(x)+ \a(x)f(x)-\a(x)Kf(x) \right | \underset{h \searrow 0}{\longrightarrow} 0.\]
  Following \eqref{defT}, we denote $T(x)=T_1(x) + T_2(x) + T_3(x) + T_4(x)$ and we shall prove that 
\begin{equation}\label{T1}
  \underset{x \in E}{\sup} \left | T_1(x) + \a(x)f(x) \right|\underset{h \searrow 0}{\longrightarrow } 0,
\end{equation}
\begin{equation}\label{T2}
 \underset{x \in E}{\sup} \left | T_2(x)\right| \underset{h \searrow 0}{\longrightarrow } 0,
 \end{equation}
\begin{equation}\label{T3}
 \underset{x \in E}{\sup} \left | T_3(x) - \a(x)Kf(x) \right | \underset{h \searrow 0}{\longrightarrow} 0,
 \end{equation}
\begin{equation}\label{T4}
 \underset{x \in E}{\sup} \left | T_4(x)\right| \underset{h \searrow 0}{\longrightarrow } 0.
 \end{equation}

For \eqref{T1},  we compute for $h>0$ and $x \in E$,
    \begin{align*}
      T_1(x) + \a(x)f(x) & =  \a(x)f(x) -  \dfrac{1}{h} \E_{x}^Y \left [ f(Y_h) \int_0^h  \a( Y_u) \, \dd u \right ]  \\
       &= \dfrac{1}{h} \int_0^h \E_{x}^Y \left [  \a(x)f(x)  -   f(Y_h)  \a( Y_u) \right ]  \, \dd u \\
        & = \dfrac{1}{h} \int_0^h \E_x^Y \left [ \E_x^Y \left [ \a(x) f(x)  -   f(Y_h) \a( Y_u) \left | Y_u \right. \right ] \right ]  \, \dd u \\
        & =  \dfrac{1}{h} \int_0^h \E_x^Y \left [   \a(x) f(x)  -   Q^Y_{h-u} f(Y_u)  \a( Y_u) \right ]  \, \dd u \\
        & =  \dfrac{1}{h} \int_0^h \E_x^Y \left [   \a(x) f(x)  -    f(Y_u) \a( Y_u) \right ]  \, \dd u  + \dfrac{1}{h} \int_0^h \E_x^Y \left [    f(Y_u) \a( Y_u)-    Q^Y_{h-u} f(Y_u)  \a( Y_u) \right ]  \, \dd u \\
        & =  \dfrac{1}{h} \int_0^h  \left ( f \times \a  -   Q_u^Y (f \times \a) \right  )(x)  \, \dd u + \dfrac{1}{h} \int_0^h \E_x^Y \left [   \a( Y_u) \left ( f -  Q^Y_{h-u} f \right )(Y_u) \right ]  \, \dd u \\
        & =  \int_0^1  \left ( f \times \a  -   Q_{hv}^Y (f \times \a) \right  )(x)  \, \dd v +  \int_0^1 \E_x^Y \left [   \a( Y_{hv}) \left ( f -  Q^Y_{h(1-v)} f \right )(Y_{hv}) \right ]  \, \dd v.
    \end{align*}
    So,
    \begin{align*}
        \left |  T_1(x) + \a(x)f(x)\right | 
         \leq \int_0^1 \|Q_{hv}^Y (f \times \a) -  f \times \a\|_{\infty} \, \dd v + \a^* \int_0^1 \|  Q^Y_{h(1-v)} f - f \|_{\infty}  \, \dd v ,
    \end{align*}
    that does not depend on $x \in E$ and converges to zero when $h \searrow 0$ by the dominated convergence theorem, the fact that $f\in L_0^Y$ and the assumption $\a \times f \in L_0^Y$. This proves \eqref{T1}.

 Now for $f \in L_0^Y$ and $x \in E$ 
    \begin{align*}
        \left | T_2(x) \right | & \leq \dfrac{\|f\|_{\infty}}{2h}  \E_{x}^Y \left [ \left ( \int_0^h  \a( Y_u) \, \dd u \right )^2 \right ] \\
        & \leq  \dfrac{\|f\|_{\infty}  \, ( \a^* h)^2}{2h} = \dfrac{  \|f\|_{\infty} ( \a^*)^2}{2} \, h,
    \end{align*}
    that does not depend on $x \in E$ and converges to zero when $h \searrow 0$, proving \eqref{T2}.
    
    For \eqref{T3}, we have for any $f\in L_0^Y$,
    \begin{align}\label{T3control}
     T_3(x) & =  \dfrac{1}{h} \E_x \left [ f(X_h) \ind_{  \t_1 \leq h} \ind_{ \t_2 > h-  \t_1}  \right ] \nonumber\\
        &  =  \dfrac{1}{h} \E_x \left [  \E_x \left [ f(X_h) \ind_{  \t_1 \leq h} \ind_{ \t_2 > h-  \t_1}  \left |  \FF_{ \t_1},  Y^{(1)} \right. \right ]   \right ] \nonumber\\
         &  =  \dfrac{1}{h} \E_x \left [ f(Y_{h-\t_1}^{(1)}) \ind_{ \t_1 \leq h}  \P_{x} \left ( \t_2 > h- \t_1 \left | \FF_{ \t_1},  Y^{(1)} \right. \right )   \right ] \nonumber\\
         &  =  \dfrac{1}{h} \E_x \left [  f(Y^{(1)}_{h- \t_1})   \ind_{ \t_1 \leq h}  \e^{ - \int_0^{h- \t_1}  \a \left (  Y_u^{(1)} \right ) \, \dd u } \right ] \nonumber\\
         &  =  \dfrac{1}{h} \E_x \left [ \ind_{ \t_1 \leq h}  \E_x \left [      f(Y^{(1)}_{h- \t_1})  \e^{ - \int_0^{h- \t_1}  \a \left (  Y_u^{(1)} \right ) \, \dd u } \left | \FF_{ \t_1} \right. \right ] \right ] \nonumber\\
         &  =  \dfrac{1}{h}  \E_x \left [ \ind_{ \t_1 \leq h}   \E^{ Y}_{ X_{ \t_1}} \left [     f(Y_{h- \t_1})  \e^{ - \int_0^{h- \t_1}  \a \left (  Y_u \right ) \, \dd u } \right ] \right ] \nonumber\\
         &  = \dfrac{1}{h}  \E_x \left [ \ind_{ \t_1 \leq h}   \E^{ Y}_{ X_{ \t_1}} \left [     f(Y_{h- \t_1})  \right ] \right ] + \dfrac{1}{h}  \E_x \left [ \ind_{ \t_1 \leq h}   \E^{ Y}_{ X_{ \t_1}} \left [     f(Y_{h- \t_1}) \left ( \e^{ - \int_0^{h- \t_1}  \a \left (  Y_u \right ) \, \dd u } -1 \right ) \right ] \right ] .
    \end{align}
  The second term above converges uniformly to $0$ when $h \searrow 0$ because   
    \begin{align*}
         \left |\dfrac{1}{h}  \E_x \left [ \ind_{ \t_1 \leq h}   \E^{ Y}_{ X_{ \t_1}} \left [     f(Y_{h- \t_1}) \left ( \e^{ - \int_0^{h- \t_1}  \a \left (  Y_u \right ) \, \dd u } -1 \right ) \right ] \right ] \right | & \leq \E_x  \left [ \dfrac{\ind_{ \t_1 \leq h} \, \| f \|_{\infty}}{h}  \E^{ Y}_{ X_{ \t_1}} \left | \int_0^{h- \t_1}  \a( Y_u) \, \dd u \right | \right ] \\
         & \leq \frac{ h \a^* \| f \|_{\infty}}{h} \P_{x}(\t_1 \leq h) \\
         & =   \frac{ h \a^* \| f \|_{\infty}}{h} \E_x^Y \left ( 1-\e^{ - \int_0^{h}  \a \left (  Y_u \right ) \, \dd u }  \right ) \\
         & \leq   \a^* \| f \|_{\infty} \, \E_x^Y \left (\int_0^{h}  \a \left (  Y_u \right ) \, \dd u  \right ) \\
         & \leq (\a^*)^2 \| f \|_{\infty} h.
    \end{align*}
    Let us now consider the first term in   \eqref{T3control} and prove that it converges uniformly to $ \a(x) Kf(x)$.
    
    \begin{align*}
     \dfrac{1}{h}  \E_x    \left [ \ind_{ \t_1 \leq h}   \E^{ Y}_{ X_{ \t_1}} \left [     f(Y_{h- \t_1})  \right ] \right ] 
        & = \dfrac{1}{h} \E_x  \left [  \ind_{  \t_1 \leq h}  \E_x  \left [ \E^{ Y}_{ X_{ \t_1}} \left [  f(Y_{h- \t_1}) \right ] \left |  Y^{(0)},  \t_1 \right. \right ] \right ] \\
        & = \dfrac{1}{h} \E_x  \left [ \ind_{  \t_1 \leq h}  \int_{ E}  \E^{ Y}_{z} \left [  f(Y_{h- \t_1}) \right ] K \left ( Y_{ \t_1}^{(0)} , \dd  z  \right ) \right ] \\
        & = \dfrac{1}{h} \E_x \left [ \ind_{  \t_1 \leq h}  \int_{ E}  Q^Y_{h - \t_1}f(z)  K \left ( Y_{ \t_1}^{(0)} , \dd  z  \right ) \right ]   \\
        & =  \dfrac{1}{h} \E_x \left [ \E_x \left [ \ind_{  \t_1 \leq h}  \int_{ E}  Q^Y_{h - \t_1}f(z)  K \left ( Y_{ \t_1}^{(0)} , \dd  z  \right ) \left | Y^{(0)} \right. \right ] \right ]   \\
        & = \dfrac{1}{h} \E_x \left [ \int_0^h   \int_{ E}  Q^Y_{h - s}f(z)      K \left ( Y_{ s}^{(0)} , \dd  z  \right ) \a(Y_s^{(0)}) \e^{-\int_0^s \a(Y_u^{(0)}) \, \dd u} \, \dd s \right ]   \\
        & = \E_x^Y \left [ \int_0^1   \int_{ E}  Q^Y_{h(1 - v)}f(z)  K \left ( Y_{ hv} , \dd  z  \right ) \a(Y_{hv}) \e^{-\int_0^{hv} \a(Y_u) \, \dd u} \, \dd v \right ]  \\
        & =  \E_x^Y \left [ \int_0^1   \int_{ E}  Q^Y_{h(1 - v)}f(z)  K \left ( Y_{ hv} , \dd  z  \right ) \a(Y_{hv})  \, \dd v \right ]  \\
        & \hspace{1cm} + \E_x^Y \left [ \int_0^1   \int_{ E}  Q^Y_{h(1 - v)}f(z)  K \left ( Y_{ hv} , \dd  z  \right ) \a(Y_{hv}) \left ( \e^{-\int_0^{hv} \a(Y_u) \, \dd u} - 1 \right ) \, \dd v  \right ] .
    \end{align*}
    
    On one hand, 
    \begin{align*}
        \left |  \E_x^Y \left [ \int_0^1   \int_{ E}  Q^Y_{h(1 - v)}f(z)  K \left ( Y_{ hv} , \dd  z  \right ) \a(Y_{hv}) \left ( \e^{-\int_0^{hv} \a(Y_u) \, \dd u} - 1 \right ) \, \dd s  \right ] \right | & \leq \a^* \| f \|_{\infty} \E_x^Y \left [ \int_0^1 \int_0^{hv} \a(Y_u) \, \dd u \, \dd v \right ] \\
        & \leq  (\a^*)^2 \| f \|_{\infty} \, h,
    \end{align*}
    which tends uniformly to 0 when $h \searrow 0$. 
    And on the other hand, 
    \begin{align*}
        & \left | \E_x^Y \left [ \int_0^1   \int_{ E}  Q^Y_{h(1 - v)}f(z)  K \left ( Y_{ hv} , \dd  z  \right ) \a(Y_{hv})  \, \dd v \right ]  -  \a(x) Kf(x) \right | \\
         & \leq  \int_0^1 \left |  \E_x^Y \left [ \a ( Y_{hv}  ) \, K  Q^Y_{h(1 - v)}f(Y_{hv}) - \a(Y_{hv}) \, Kf(Y_{hv}) \right ] \right | \, \dd v 
          +  \int_0^1 \left | \E_x^Y \left [ \a(Y_{hv}) \, Kf(Y_{hv}) - \a(x) Kf(x) \right ] \right |  \, \dd v  \\
         & \leq \a^* \int_0^1 \| KQ_{h(1-v)}^Yf - Kf \|_{\infty} \, \dd v + \int_0^1 \left | Q_{hv}^Y(\a \times Kf) (x) - (\a \times Kf)(x) \right | \, \dd v \\
         & \leq \a^* \int_0^1 \| Q_{h(1-v)}^Yf - f \|_{\infty} \, \dd v + \int_0^1 \| Q_{hv}^Y(\a \times Kf)  - (\a \times Kf) \|_{\infty} \, \dd v,
    \end{align*}
    converges to $0$ when $h \searrow 0$ by the dominated convergence theorem and the fact that $f \in L_0^Y$ and $\alpha\times Kf \in L_0^Y$. The latter is implied by the fact that by assumption $g:=Kf \in L_0^Y$,  implying $\alpha\times g\in L_0^Y$. This proves \eqref{T3}.

  To complete the proof, it remains to remark that \eqref{T4} follows from the following, using  \eqref{dominationpoisson}, 
    \begin{align*}
        \left | T_4(x)\right | & \leq \dfrac{\| f \|_{\infty}}{h} \P_x(N_h \geq 2) \\
        & \leq \dfrac{\| f \|_{\infty}}{h} \left ( 1 - \e^{-\a^*h}- \a^*h \e^{-\a^*h} \right ) \\
        & =  \dfrac{\| f \|_{\infty}}{h} \left ( \dfrac{(\a^* h)^2}{2} +  \underset{h \searrow 0}{o}(h^2) \right ) \\
        & = \dfrac{\| f \|_{\infty} \, (\a^*)^2}{2} \, h + \underset{h \searrow 0}{o}(h).
    \end{align*}

\section{Topological results for systems of interacting particles in $\R^d$ } \label{sec6.2}

We detail the topological properties of the state space $E$ for systems of interacting particles in $W\subset \R^d$, introduced in Section~\ref{sec2.3}. Remember that in this setting $E=\cup_{n=0}^\infty E_n$
where $E_n=\pi_n(W^n)$ with $\pi_n((x_1,\dots,x_n))=\{x_1,\dots,x_n\}$, and we have equipped the space $E$  with the distance $d_1$ defined for $x=\{x_1,\dots,x_{n(x)} \}$ and $y=\{y_1,\dots,y_{n(y)} \}$ in $E$ such that $n(x) \leq n(y)$ by 
\begin{equation*} 
     d_1(x,y) = \frac{1}{n(y)} \left ( \min_{\s \in \SS_{n(y)}} \sum_{i=1}^{n(x)} (\|x_i-y_{\s(i)}\| \wedge 1) + (n(y)-n(x)) \right ),
    \end{equation*}
with $d_1(x,\textnormal{\O})=1$ and where $\SS_n$ denotes the set of permutations of $\{1,\dots,n\}$. 

We verify in this section that if $W$ is a closed subset of $\R^d$ (possibly $W=\R^d$), then $(E,d_1)$ is a locally compact and complete set, strengthening results already obtained in \cite{schuhmacher2008}. We also show that $n(.)$ and $\pi_n(.)$ are continuous under this topology, as claimed in Section~\ref{sec2.3}. We continue with the proof of Proposition~\ref{suitecvgentedeE}, which clarifies the meaning of converging sequences in $(E,d_1)$, and of Proposition~\ref{compactsetofE} that describes the compact sets of $E_n$ and $E$, along with some useful corollaries. We finally show that the Hausdorff distance is not appropriate in our setting, not the least because it does not make $n(.)$ continuous.

In the following, we will often use in a equal way the spaces $\left ( \R^{nd}, \|.\| \right )$ and $\left ( (\R^d)^n , \|.\|_n \right )$ where 
\begin{equation*}
    \|x\|_{n}= \dfrac{1}{n} \sum_{i=1}^n \|x_i\|. \label{normen}
\end{equation*}
 Indeed, introducing the natural bijection $\psi_n : z \in \R^{nd} \mapsto (z_1, \dots,z_n) \in (\R^d)^n$ we observe that for any $z \in  \R^{nd}$, $\| z \|/n \leq \| \psi_n(z) \|_n \leq \| z \|/\sqrt{n}$
by the Cauchy-Schwarz inequality. The norms being equivalent, we henceforth abusively confuse $z$ and $\psi_n(z)$.  Similarly, any function from $\R^{nd}$ to $\R^d$  can be seen as a function from $ (\R^d)^n$ to $\R^d$ and we will confuse the two points of view. 

We start in the following lemmas with the continuity of $n(.)$ and $\pi_n(.)$. We will use the following straightforward property,  for all $x,y\in E$,
\begin{equation}\label{minoredkappa}
 d_1(x,y) \geq \frac{|n(y)-n(x)|}{n(x) \vee n(y)}.
\end{equation}

\begin{lem}\label{continuity n}
The function $\displaystyle n(.) : (E,d_1) \rightarrow (\N,|.|)$ is continuous.
\end{lem}
\begin{proof}
Take $x \in E$ and a sequence $(x^{(p)})_{p \geq 0}$ such that $d_1(x^{(p)},x)\to 0$ as $p\to\infty$.  Assume that the sequence $(n(x^{(p)}))_{p \geq 0}$ is not bounded. We then may define a subsequence $(n(x^{(p')}))_{p' \geq 0}$ such that $n(x^{(p')}) \to \infty,$ and by \eqref{minoredkappa} we obtain
        \[ d_1(x,x^{(p')}) \geq \dfrac{|n(x)-n(x^{(p')})|}{n(x) \vee n(x^{(p')})} \underset{p' \rightarrow \infty}{\longrightarrow} 1,\]
        which is a contradiction. The sequence $(n(x^{(p)}))_{p \geq 0}$ is therefore bounded by some $M > 0$, which gives again by \eqref{minoredkappa} 
        \[ |n(x^{(p)}) - n(x) | \leq (M \vee n(x) ) \,  d_1(x^{(p)},x) \underset{p \rightarrow \infty}{\longrightarrow} 0, \]
       that is
        \[ n(x^{(p)})  \underset{p \rightarrow \infty}{\longrightarrow} n(x).\]
\end{proof}

\begin{lem}\label{propContinuousProj}
    The projection $\pi_n : (W^n, \|.\|_n) \rightarrow (E_n,d_1)$ is continuous.
\end{lem}

\begin{proof}
    Let $x,y \in W^n$. Then
    \begin{align*}
        d_1(\pi_n(x),\pi_n(y)) & = \frac{1}{n} \left ( \min_{\s \in \SS_{n}} \sum_{i=1}^{n} (\|x_i-y_{\s(i)}\| \wedge 1) \right ) \\
        & \leq \frac{1}{n} \sum_{i=1}^{n} (\|x_i-y_{i}\| \wedge 1) \\
        & \leq \|x-y\|_n.
    \end{align*}
\end{proof}

From Lemma~\ref{propContinuousProj} we deduce that $(E,d_1)$ is a locally compact space.
\begin{cor} \label{Elocalcompact}
    Let $W$ a closed subset of $\R^d.$ Then $(E,d_1)$ is a locally compact space.
\end{cor}
\begin{proof}
    First recall that $d_1(x,\textnormal{\O})=1$ so $\{ \textnormal{\O} \}$ is a compact neighborhood of \O.
    Now take $x = \{ x_1,\dots,x_n \} \in E_n$ with $n \geq 1$. The space $W^n$ is locally compact so there exists $K \subset W^n$ a compact neighborhood of $(x_1,\dots,x_n).$ Now set $\tilde K = \pi_n(K).$ Then, $x \in \tilde K$ and  $\tilde K$ is a compact set by Lemma~\ref{propContinuousProj}. We show that there is an open set containing $x$ which is included in $\tilde K$. By definition there exists $\ep \in (0,\frac{1}{2})$ such that $B_{\|.\|_n} ( (x_1,\dots,x_n),\ep) \cap W^n \subset K$, where $B_{\|.\|_n} ( (x_1,\dots,x_n),\ep)$ is the open ball  centred at $(x_1,\dots,x_n)$ with radius $\ep$ for the norm $\|.\|_n$. If $z \in B_{d_1}(x,\ep) \cap E_n$ there exists $\s \in \SS_n$ such that 
    \[ 
    \dfrac{1}{n} \sum_{i=1}^n \| x_i-z_{\s(i)} \| < \ep,
    \]
    so $z = \pi_n((z_{\s(1)},\dots, z_{\s(n)}))$ and $(z_{\s(1)},\dots, z_{\s(n)}) \in B_{\|.\|_n} ( (x_1,\dots,x_n),\ep) \cap W^n.$ To sum up,
    \[
    B_{d_1}(x,\ep) \cap E_n \subset \pi_n \left ( B_{\|.\|_n} ( (x_1,\dots,x_n),\ep) \cap W^n \right ) \subset \tilde K,
    \]
    so $ \tilde K$ is a compact neighborhood of $x$ in $E_n$ and so in $E$. 
\end{proof}

A further consequence of Lemma~\ref{propContinuousProj} is the following result, that will turn to be useful when considering the Feller continuous property of a process on $E$. 

\begin{cor} \label{fctCbdeWn}
    If $f \in C_b(E)$ then for any $n \geq 1$,  $f \circ \pi_n \in C_b(W^n)$.
\end{cor}
\begin{proof}
    For any $n \geq 1$ and $f \in C_b(E),$ the function $f \circ \pi_n$ is well-defined on $W^n,$ continuous as the composition of two continuous functions and bounded by $\|f\|_{\infty}.$
\end{proof}

Let us now prove that  $(E,d_1)$ is a complete space.

\begin{prop}\label{propWclosed}
    Suppose that $W$ is closed. Then $(E,d_1)$ is a complete space and  for any $n\geq 1$, $(E_n,d_1)$ is also complete.
\end{prop}
\begin{proof}  
Let $(x^{(p)})_{p \geq 0}$ be a Cauchy sequence in $(E,d_1)$.  First, we show that the sequence $ ( n  ( x^{(p)}  )  )_{p \geq 0}$ is constant for $p$ large enough. Fix $\ep \in (0,1)$. There exists $q \geq 0$ such that for any $p \geq q$, $d_1(x^{(p)},x^{(q)}) < \ep$, so by \eqref{minoredkappa} 
    \[
    \left | n  ( x^{(p)}  ) - n ( x^{(q)}  ) \right | \leq (n ( x^{(p)}  ) \vee n  ( x^{(q)}  ) ) \, \ep \leq ( n  ( x^{(p)}  ) + n  ( x^{(q)}  ) ) \, \ep ,
    \]
implying that 
    $ \displaystyle \left ( 1 - \ep \right )  n  ( x^{(p)}  ) \leq \left ( 1 + \ep \right ) n  ( x^{(q)}  ) $ and   
    $\displaystyle n ( x^{(p)}  ) \leq n  ( x^{(q)}  )  (1 + \ep)/(1 - \ep) $. This entails that  the sequence $ ( n  ( x^{(p)}  )  )_{p \geq 0}$ is bounded by some $N_0 >0$. 
    Now take $\ep \in (0,1)$ and $p_1 \geq 0$ such that for any $p \geq p_1$, $ \displaystyle d_1(x^{(p)},x^{(p_1)})<\ep/N_0.$ Write  $n=n(x^{(p_1)})$ for short. Then by \eqref{minoredkappa} one has for any $p \geq p_1$
    \[
    \left | n  ( x^{(p)}  ) - n \right | \leq (n ( x^{(p)}  ) \vee n  ) \, d_1(x^{(p)},x^{(p_1)}) \leq N_0 \, d_1(x^{(p)},x^{(p_1)}) \leq \ep < 1,
    \]
    which implies that  $n(x^{(p)})=n$ for all $p\geq p_1$. 
    
     Second, we may fix $p_2 \geq 0$ such that $d_1(x^{(p)},x^{(q)}) \leq \ep$ for any $p,q \geq p_2.$  Finally let $p_0 = \max (p_1,p_2),$ so that for all $p,q \geq p_0,$ 
    \[d_1 (x^{(p)},x^{(q)} ) = \frac{1}{n}  \min_{\s \in \SS_{n}} \sum_{i=1}^{n} \|x^{(p)}_i-x_{\s(i)}^{(q)}\| \leq \ep.\]
    In particular for $q=p_0,$ this leads to $  \min_{\s \in \SS_{n}} \sum_{i=1}^{n} \|x^{(p_0)}_i-x_{\s(i)}^{(p)}\|/n \leq \ep$ for any $p \geq p_0.$ 
    The minimum over $\sigma$ is reached for some $\sigma_{p_0,p} \in \SS_{n},$ so that we may define the sequence  $(\hat{x}^{(p)})_{p\geq p_0}$  in $W^n$ by $\hat{x}^{(p)}=(x^{(p)}_{\s_{p_0,p}(1)}, \ldots,x^{(p)}_{\s_{p_0,p}(n)})$  satisfying $ \| \hat{x}^{(p)} - x^{(p_0)} \|_n\leq \ep$ for all $p \geq p_0.$ 
    Then for $p,q \geq p_0$, $\| \hat{x}^{(p)} -\hat{x}^{(q)} \|_n\leq 2\ep$. This proves that  the sequence $( \hat{x}^{(p)})_{p \geq p_0}$ is a Cauchy sequence in the finite dimensional vector space $((\R^d)^n,\|.\|_n),$ implying its convergence to some $\hat x \in W^n$ because $W$ is a closed set. Finally for $p \geq p_0$
    \begin{align*}
        d_1(x^{(p)},\pi_n(\hat{x})) & = d_1(\pi_n( \hat{x}^{(p)}),\pi_n( \hat{x})) \leq \|\hat{x}^{(p)}-\hat{x}\|_n \leq 2 \ep,
    \end{align*}
    which proves that $(x^{(p)})_{p \geq 0}$ converges to $\pi_n(\hat{x})$ in $E$, and so $(E,d_1)$ is complete.
    
   Finally for any $n\geq 1$, $(E_n,d_1)$ is also complete as a closed subset of $(E,d_1)$ by continuity of $n(.)$. 
\end{proof}

\subsection{Proof of Proposition~\ref{suitecvgentedeE}}
Let $x\in E$ and set $n=n(x)$. By Lemma~\ref{continuity n}, if $x^{(p)}$ converges to $x$ as $p\to\infty$, i.e. $d_1(x^{(p)},x)\to 0$, then $n(x^{(p)})$ tends to $n$, which means that there exists $p_0\geq 1$ such that $ n(x^{(p)})=n$ for all $p \geq p_0$. 
From the definition of $d_1,$ for any  $p \geq p_0$ there exists a permutation $\sigma_p\in \SS_n$  satisfying  
 \[d_{1}(x^{(p)},x ) = \frac{1}{n} \sum_{i=1}^n ( \|x_i - x_{\sigma_p(i)}^{(p)}\|\wedge 1) .\]
 Assume that there exists $i\in \{1,\ldots,n\}$ such that $\limsup_{p \rightarrow \infty} \|x_i - x_{\sigma_p(i)}^{(p)}\|>0.$ 
 We then may fix $\eta >0$ and a subsequence $(\varphi(p))_{p \geq p_0}$, both depending on $i$, such that for every $p \geq p_0,$ $ \|x_i - x_{\sigma_{\varphi(p)}(i)}^{(\varphi(p))}\| \geq \eta.$ This implies $d_{1}(x^{(\varphi(p))},x ) \geq (\eta \wedge 1)/n$ and $\limsup_{p \rightarrow \infty} d_{1}(x^{(p)},x ) >0$ which is a contradiction. Finally, for every $i = 1,\ldots,n,$ $\limsup_{p \rightarrow \infty} \|x_i - x_{\sigma_p(i)}^{(p)}\| = 0$, proving the result.

\subsection{Proof of Proposition~\ref{compactsetofE} and corollaries}

In order to prove this proposition, we first recall the following definitions  and results (see e.g. \cite{bourbaki1966general}):
\begin{itemize}
\item A finite subset $L$ of a metric space $(X,d)$ is called an $\ep-$net, for $\ep>0$, if the following property is satisfied :
        \[ \forall \, x \in X, \; \exists \, l \in L, \; s.t. \; d(x,l) \leq \ep.  \]
\item A metric space $(X,d)$ is said to be totally bounded if it contains an $\ep-$net for any $\ep >0.$
\item  Let $(X,d)$ a metric space. Then $(X,d)$ is compact if and only if $(X,d)$ is totally bounded and complete.
    \end{itemize}

To prove the first statement of the proposition, let $A$ be a closed subset of $(E_n,d_1)$. We start by assuming that we may fix $\ep \in \left ( 0,1/n \right )$ and $ w \in W$ such that
    \begin{equation}\label{contraposee}
     \forall \, R >0, \; \exists \, x = \{x_1,...,x_n\} \in A,  \; \underset{1 \leq k \leq n}{\max} \{ \|x_k-w\|  \} > R + n \ep,
    \end{equation}
    and we show that $A$ is not a compact set because it does not contain any $\ep-$net. Take $L=\{ l^{(1)}, \dots , l^{(N)} \}$ a finite subset of $A$ and let us define
    \[  R_0 = \underset{1 \leq i \leq N}{\max} \; \underset{1 \leq k \leq n}{\max} \{ \|l_k^{(i)}-w\| \}. \]
    By \eqref{contraposee} we may define $x \in A$ and $1 \leq j \leq n$ such that
    \[ \|x_j-w\| = \underset{1 \leq k \leq n}{\max} \{ \|x_k-w\|  \} > R_0 +  n \ep. \]
    This leads  for all $\s \in \SS_n$ and $1 \leq i \leq N$ to
    \[ \|x_j-l_{\s(j)}^{(i)}\| \geq \left| \|x_j-w\| - \|l_{\s(j)}^{(i)}-w\| \right| = \|x_j-w\| - \|l_{\s(j)}^{(i)}-w\| >  n \ep   \]
    and for any $1 \leq i \leq N$
    \begin{align*}
        d_1(x,l^{(i)}) & = \frac{1}{n} \; \underset{\s \in \SS_n}{\min} \sum_{k=1}^n \left ( \|x_k-l_{\s(k)}^{(i)}\| \wedge 1  \right ) \\
        & \geq  \frac{1}{n} \; \underset{\s \in \SS_n}{\min}  \left ( \|x_j-l_{\s(j)}^{(i)}\|  \wedge 1 \right ) \\
        & > \frac{ n \ep}{n} = \ep.
    \end{align*}
    Therefore $L$ cannot be an $\ep-$net and A cannot be a  compact set.
    
    Let us now prove the converse. Fix $w \in W$ and assume that there exists a positive $R$ such that for all $x \in A,$ \[  \underset{1 \leq k \leq n}{\max} \{ \|x_k-w\|  \} \leq R.\]
    Under this assumption, $A$ is a subset of 
    \[  C := \{ x \in E_n, \dfrac{1}{n} \sum_{k=1}^n \|x_k-w\| \leq R \}. \]
    Let us show that $C$ is a compact set.
    To this end we define \textbf{w} $=  ( w, \ldots, w) \in W^n$ and write  $\bar{B}_{\|.\|_n} ( \textnormal{\textbf{w}},R) $ for the closed ball of radius $R$ and center \textbf{w} for the norm $\|.\|_{n}$ on the finite dimensional vector space $(\R^d)^n.$ The closed set $ \bar{B}_{\|.\|_n} ( \textnormal{\textbf{w}},R) \cap W^n $ is then a compact set of $W^n$ and by continuity of the projection $\pi_n,$ we get that $\pi_n \left (   \bar{B}_{\|.\|_n} ( \textnormal{\textbf{w}},R) \cap W^n \right )$ is a compact set of $E_n$.
    Let us prove that $ \pi_n \left (   \bar{B}_{\|.\|_n} ( \textnormal{\textbf{w}},R) \cap W^n \right )=C$ to conclude the proof.
   First, if $x=\{x_1,\dots,x_n \} \in C$ then $\hat{x}=(x_1,\dots,x_n) \in  \bar{B}_{\|.\|_n} ( \textnormal{\textbf{w}},R) \cap W^n $ and $\pi_n(\hat{x})=x.$
    Second, if $x = \pi_n(\hat{x})$ with $\hat{x} = (x_1,\dots,x_n) \in \bar{B}_{\|.\|_n} ( \textnormal{\textbf{w}},R) \cap W^n, $ then
    \[\dfrac{1}{n} \sum_{k=1}^n \|x_k-w\|= \| \hat{x} - \textnormal{\textbf{w}} \|_n \leq R\]
    which proves the claim. The set $C$ is then compact and so is $A$ because it is a closed set.

Let us finally prove the second statement of Proposition~\ref{compactsetofE} by contradiction. Let $A$ be a compact subset of $E$ and suppose that $\mathcal P= \{ p \geq 0, \, A \cap E_p \neq \emptyset \}$ is infinite. Then we can construct a sequence $(y_p)_{p \in \mathcal P}$ with $y_p \in A \cap E_p$. But $A$ is a compact set so there exists a subsequence $(y_{p'})_{p' \in \mathcal P}$ which  converges to some $y \in E$ when $p' \to \infty$. But by Lemma~\ref{continuity n}, $n(y_{p'})=p' \to n(y) $ as $p'\to\infty$ which is absurd, concluding the proof.

We end this section with two corollaries of Proposition~\ref{compactsetofE}.

\begin{cor}\label{Cor9}
    If $W$ is a compact set, then $(E_n, d_1)$ is a compact set for any $n \geq 1$.
\end{cor}

\begin{proof}
    $W$ is a compact set of $\R^d$ so it is bounded, i.e. we may fix a non-negative $R$ such that $\|w\|\leq R$ for any $w \in W.$  Let $w \in W$ and $x \in E_n$. Then
    \[ \underset{1 \leq k \leq n}{\max} \{ \|x_k-w\| \}  \leq \underset{1 \leq k \leq n}{\max} \|x_k\| +\|w\| \leq 2R.   \]
    $E_n$ is therefore a compact set by the first statement of Proposition~\ref{compactsetofE}.
\end{proof}

\begin{cor} \label{fctC0deWn}
    If $f \in C_0(E)$ then for any $n \geq 1$, $f \circ \pi_n \in C_0(W^n).$
\end{cor}

\begin{proof}
    Take $f \in C_0(E)$ and $\ep >0$. There exists a compact set $B \subset E$ such that if $ x \notin B$ then $ |f(x)|< \ep.$ In this case $B_n := B \cap E_n$ is a compact set because $E_n$ is closed so by Proposition~\ref{compactsetofE}  there exists $w \in W$ and $R \geq 0$ such that for any $x=\{x_1,\dots,x_n\} \in B_n$, $\max_{1 \leq k \leq n} \| x_k-w \| \leq R.$ Then for any $z \notin \bar{B}_{\|.\|_n}(w, R/n)$ we get $|f \circ \pi_n (z)|<\ep.$
\end{proof}

\subsection{The Hausdorff distance is not appropriate } \label{sec6.5}

For systems of particles in $\R^d$, we have equipped $E$ with the distance $d_1$ defined in \eqref{defd1}. 
A common alternative distance between random sets is the Hausdorff distance defined for $x =\{ x_1, \dots, x_{n(x)} \}$ and $y = \{ y_1 , \dots, y_{n(y)}\}$ in $E$ by
    \[  d_H(x,y)= \max \left \{ \max_{1 \leq i \leq n(x)} \min_{1 \leq j \leq n(y)} \|x_i-y_j\|, \max_{1 \leq j \leq n(y)} \min_{1 \leq i \leq n(x)} \|x_i-y_j\| \right \}.\]
Yet we show in this section that this distance does not make the function $n(.)$ continuous, which has serious consequences on the structure of $C_b(E)$ with this topology. In particular, we show that a simple uniform death kernel is not even Feller continuous in this setting. 

As a preliminary, for the Hausdorff distance to be a  proper distance, we must focus on simple point configurations only. 
We therefore consider for any $n \geq 1$
\[ \tilde W_n = \left \{ (x_1,\dots,x_n ) \in \R^n, \quad i \neq j \implies x_i \neq x_j \right \},\]
and the state space is
\[ \tilde E = \bigcup_{n \geq 0} \tilde E_n,\]
where $\tilde E_n={\tilde \pi_n}(\tilde W_n)$ and $\tilde \pi_n$ is the same projection function as in Section~\ref{sec2.3} but defined on $\tilde W_n$. Then we have
\begin{lem}
The Hausdorff distance $d_H$ is a proper distance function on $\tilde E$.
\end{lem}
\begin{proof}
Symmetry is obvious and triangle inequality is well known for $d_H$. 
We only prove the identity of indiscernibles. Let $x =\{ x_1, \dots, x_{n(x)} \}$ and $y = \{ y_1 , \dots, y_{n(y)} \}$ in $\tilde E$ satisfying $ d_H(x,y)= 0.$ This implies 
    \[ \min_{1 \leq j \leq n(y)} \|x_i-y_j\| = 0 \]
    for any $i \in \{1,...,n(x)\},$ leading for any $i \in \{1,...,n(x)\}$ to the existence of $j \in \{1,...,n(y)\}$ such that $x_i=y_j.$ Since $x$ and $y$ are simple, we deduce that $n(y) \geq n(x).$ We obtain similarly $n(x) \geq n(y)$ and then $n(x)=n(y).$  We then may define a permutation $\s \in \SS_n$ such that for all $i \in \{1,...,n(x)\},$
    $x_i=y_{\s(i)}$ which means that $x=y$ in  $\tilde E$.
\end{proof}

We now verify that $n(.)$ is not continuous for this topology. 

\begin{lem}\label{PropdistanceHausdorff}
Assume that $\mathring W\neq\emptyset$. Then the function $n(.)$ is not continuous on $(\tilde E, d_H).$ 
\end{lem}

\begin{proof}
Assume without loss of generality that $0\in \mathring W$. Let $k \geq 1$ and $y\in \R^d$ such that $\|y\|=1/k$. Take $k$ large enough so that $y\in W$. Then $\left | n \left ( \left \{ 0, y \right \} \right ) - n(\{0\}) \right | = 1$ and $ d_H \left ( \left \{ 0, y \right \}  , \{0\} \right ) = 1/k \to 0$ as $k\to\infty$, proving the result.  
\end{proof}

This result reveals a singularity caused by the distance $d_H$. As a consequence, a simple uniform death kernel is not even Feller continuous, as proved in the following lemma. 
\begin{lem}
Assume that $\mathring W\neq\emptyset$ and consider for $f\in M^b( \tilde E)$ the kernel 
\[ Kf(x)=\frac{1}{n(x)} \sum_{i=1}^{n(x)} f(x \, \backslash \, x_i). \]
Then $K  C_b( \tilde E)$ is not included in $C_b( \tilde E)$, i.e. $K$ is not Feller continuous. 
\end{lem}

\begin{proof}
Consider the function $f(x)=\max_{1\leq i\leq n(x)} x_{i,1}\wedge 1$ where $x_{i,1}$ is the first coordinate of $x_i\in W$. This function is bounded and satisfies for any $x,y \in \tilde E,$ 
\begin{equation}\label{boundfex}
| f(x)-f(y) | \leq \left | \max_{1 \leq i \leq n(x)} x_{i,1} - \max_{1 \leq j \leq n(y)} y_{j,1} \right |,
\end{equation}
for any $x,y \in \tilde E$.  Let us show  that the latter bound is lower than $d_H(x,y)$. Let $\mathcal I_0=\textrm{argmax}_{1 \leq i \leq n(x)} x_{i,1}$ and $\mathcal J_0=\textrm{argmax}_{1 \leq j \leq n(y)} y_{j,1}$. 
This follows from the fact that for any $i_0\in\mathcal I_0$ and $j_0\in\mathcal J_0$, 
 \begin{align*}
        d_H(x,y) & \geq \max_{1 \leq i \leq n(x)} \min_{1 \leq j \leq n(y)} \|x_i-y_j\| 
         \geq \min_{1 \leq j \leq n(y)} \|x_{i_0}-y_j\| 
         \geq  \min_{1 \leq j \leq n(y)} |x_{i_{0},1}-y_{j,1}| 
        = |x_{i_0}-y_{j_0}|.
    \end{align*}
So by \eqref{boundfex} $| f(x)-f(y) |\leq d_H(x,y)$, proving that $f\in C_b( \tilde E)$. 

Assume without loss of generality that $0\in\mathring W$. Let $a\in W$, $a\neq 0$, and $a_k=(1/k,0,\dots,0)\in \R^d$ with $k$ large enough to ensure $a_k\in W$. Consider the sequence  
$ x^{(k)} = \left \{ 0, a , a_k\right \}$ and let $x= \{ 0, a \}$ so that $ d_H (x^{(k)},x) = 1/k$ tends to $0$ as $k\to\infty$.
On the one hand,
\begin{align*}
    Kf(x^{(k)}) = \frac{1}{3} \left [ f \left ( \left \{ 0, a_k \right \} \right ) + f \left ( \left \{ a, a_k \right \} \right )  + f \left ( \left \{ 0, a \right \} \right ) \right ] = \frac{(1/k)+(1/k)\vee a_1 + a_1}{3} \underset{k \rightarrow \infty}{\longrightarrow} \frac{2 a_1}{3},
\end{align*}
and on the other hand,
\[ Kf(x) = \frac{1}{2} \left ( f(\{0\}) + f(\{a\}) \right ) = \frac{a_1}{2} \]
whereby $Kf\notin  C_b( \tilde E)$.  
\end{proof}

\section{Proof of Proposition~\ref{prop33} }

First we show that if $( Z^{\, |n}_t)_{t \geq 0}$ is a Feller continuous process on $W^n$ for every $n \geq 1$ then $(Y_t)_{t \geq 0}$ is a Feller continuous process on $E$. Indeed, let $x \in E$ and a sequence $(x^{(p)})_{p \geq 0}$ converging to $x.$ By Proposition \ref{suitecvgentedeE} we may fix $p_0 \geq 1$ such that $n(x^{(p)})=n(x):=n$ for any $p \geq p_0$ and a sequence of permutations $\s_p$ of $\{1,\dots,n\}$ such that for any $1 \leq i \leq n$,  $x_{\s_p(i)}^{(p)}  \to x_{i}$ as $p\to\infty$. 
We then obtain for any $f \in C_b(E)$ and $p \geq p_0,$ using the permutation equivariance property of $( Z^{ \, |n}_t)_{t \geq 0}$ (that allows us to arbitrarily choose the ordering of its initial value), the continuity of  its transition kernel, and  Corollary~\ref{fctCbdeWn}, that
\begin{align*}
    \E \left ( f( Y_t) \, | \, Y_0 = x^{(p)} \right ) & = \E \left ( f( Y^{|n}_t) \, | \, Y_0 = x^{(p)} \right ) \\
    & = \E \left ( f \circ \pi_n (Z^{ \, |n}_t) \, | Z^{ \, |n}_0 = (x_{\s_p(1)}^{(p)},\dots,x_{\s_p(n)}^{(p)} ) \right )\\
    & \underset{p \rightarrow \infty}{\longrightarrow} \E \left ( f \circ \pi_n (Z^{ \, |n}_t) \, | Z^{ \, |n}_0 = (x_1,\dots,x_n ) \right )\\
    & = \E \left ( f(Y_t) \, | \, Y_0 = x \right ).
\end{align*}

Second, let us prove that if $( Z^{ \, |n}_t)_{t \geq 0}$ is a Feller process on $W^n$ for every $n \geq 1$ then $(Y_t)_{t \geq 0}$ is a Feller process on $E$. Let $f \in C_0(E)$. We start by the strong continuity.  Take $\ep >0$. By the second statement of Proposition~\ref{compactsetofE} there exists $n_0 \geq 0$ such that $n(x) > n_0 \Rightarrow |f(x)|< \frac{\ep}{4}.$ So for any $x\in E$, 
\begin{align*}
    \left | Q^Y_tf(x)-f(x) \right | & \leq \left |Q^Y_tf(x)-f(x)  \right | \ind_{n(x) \leq n_0} + \E_x [|f(Y_t)|] \ind_{n(x) > n_0} + f(x) \ind_{n(x) > n_0} \\
    & \leq \sum_{n=0}^{n_0} \left | Q_t^{Y^{ \, |n}}f(x)-f(x)  \right | \ind_{x \in E_n} + \frac{\ep}{4} + \frac{\ep}{4} \\
    &  \leq \sum_{n=1}^{n_0} \left |\E \left ( f(\pi_n(Z^{ \, |n}_t))  \left | \right. Z^{ \, |n}_0=(x_{1},\dots,x_{n}) \right ) -  f(\pi_n((x_{1},\dots,x_{n}))) \right | \ind_{x \in E_n} + \frac{\ep}{2} \\
    & \leq \sum_{n=1}^{n_0} \|  Q_t^{Z^{ \, |n}}(f \circ \pi_n)-f \circ \pi_n\|_{\infty} +  \frac{\ep}{2}.
\end{align*}
By Corollary \ref{fctC0deWn}, for any $n=1,\dots,n_0$,  there exists $t_n > 0$  such that \[ t \in (0,t_n) \implies \|  Q_t^{Z^{ \, |n}}(f \circ \pi_n)-f \circ \pi_n\|_{\infty} < \frac{\ep}{2n_0}. \]
So for any $t \in (0,t(\ep))$ where $t(\ep) = \underset{1 \leq n \leq n_0}{\min} t_n$, we get $\| Q_t^Yf - f\|_{\infty} < \ep$, which proves the strong continuity of  $Q_t^Y$ at 0. 

It remains to show that $Q_t^Y C_0(E) \subset C_0(E).$ Continuity follows from above. Take now $f \in C_0(E)$ and fix $\ep >0$ and $B \subset E$ a compact set such that $x \notin B \Rightarrow |f(x)|< \frac{\ep}{2}.$ By Proposition~\ref{compactsetofE}  there exists $n_0 \geq 0$ such that $x \in B \Rightarrow n(x) \leq n_0$. Also by Corollary \ref{fctC0deWn} we can fix for any $n=1,\dots,n_0$ a compact set $A_n$ of $W^n$ such that $z \notin A_n \Rightarrow  \left | Q_t^{Z^{ \, |n}}(f \circ \pi_n)(z) \right | < \ep/(2n_0)$. Then, $A=\bigcup_{n=1}^{n_0} \pi_n(A_n)$ is a compact set of $E$ and for any $ x \notin \{ \textnormal{\O} \} \cup A \cup B$ 
\begin{align*}
    \| Q_t^Yf(x) \| \leq \sum_{n=1}^{n_0}  \| Q_t^{Z^{ \, |n}}(f \circ \pi_n)((x_1,\dots,x_n) \| + \frac{\ep}{2} \leq \ep.
\end{align*}

\section{Proofs of lemmas relating to the coupling of Appendix~\ref{sec6.7}}\label{proofs coupling}
%
%

\subsection{Proof of Lemma \ref{QqisC0}}

    First note that for any $x \in E$, $n \geq 0$ and $h > 0$ one has 
    \begin{align*}
        \P_{(x,n)}(\ck T_1 \leq h) & = \E_{(x,n)}\left (1-\e^{-\int_0^h \ck \a( \ck Y_u) \, \dd u} \right ) \leq \E_{(x,n)}\left ( \int_0^h \ck \a( \ck Y_u) \, \dd u  \right ) \leq 2 \a^* h.
    \end{align*}
   Next take $h >0$. Then 
    \begin{align}\label{accroiss psi}
         \psi_q(t+h) - \psi_q(t)  & =  \E_{(x,n)} \left [ \ind_E(X'_{t+h}) \ind_{\{ q \}}(\eta'_{t+h}) - \ind_E(X'_{t}) \ind_{\{ q \}}(\eta'_{t}) \right ] \nonumber \\
        & = \E_{(x,n)} \left [  \E_{(x,n)} \left ( \ind_E(X'_{t+h}) \ind_{\{ q \}}(\eta'_{t+h}) | \ck \FF_t \right )  - \ind_E(X'_{t}) \ind_{\{ q \}}(\eta'_{t}) \right ]  \nonumber \\
        & =  \E_{(x,n)} \left [ \ck Q_h((X'_t,\eta'_t); E \times \{ q\}) - \ind_E(X'_{t}) \ind_{\{q\}}(\eta'_{t})  \right ]  \nonumber \\
        & = \sum_{k \geq 0} \E_{(x,n)} \left [ ( \ck Q_h((X'_t,\eta'_t); E \times \{ q\}) - \ind_E(X'_{t}) \ind_{\{q\}}(\eta'_{t}) ) \ind_{ \eta'_t = k}  \right ]. \nonumber\\
         & = \sum_{k \geq 0} \E_{(x,n)} \left[ \ck Q_h((X'_t,k);E \times \{ q \}) - \ind_{\{q\}}(k) |\eta'_t = k \right] \, \P_{(x,n)}(\eta'_t = k).
    \end{align}
     For any $k \geq 0$ and $y \in E$ 
            \begin{align*}
        \left | \ck Q_h((y,k),E \times \{ q \}) -  \ind_{\{q\}}(k)  \right | & = \left | \E_{(y,k)} \left ( \ind_E(X'_h) \ind_{\{q\}}(\eta'_h) \right )  - \ind_{\{q\}}(k) \right | \\
        & = \left | \E_{(y,k)} \left (  \ind_{\{ q \}}(\eta'_h) \ind_{ \ck T_1 > h} \right ) + \E_{(y,k)} \left ( \ind_{\{ q \}}(\eta'_h) \ind_{ \ck T_1 \leq h} \right )  - \ind_{\{q\}}(k) \right | \\
        & \leq \E_{(y,k)} \left |( \ind_{\{ q \}}(\eta'_h)- \ind_{\{q\}}(k)) \ind_{ \ck T_1 \leq h} \right | + \E_{(y,k)} \left | ( \ind_{\{ q \}}(\eta'_h)- \ind_{\{q\}}(k)) \ind_{\ck T_1 >h} \right | \\
        & \leq 2 \P_{(x,n)}(\ck T_1 \leq h) + \E_{(y,k)} \left | ( \ind_{\{ q \}}(k)- \ind_{\{q\}}(k)) \ind_{\ck T_1 >h} \right | \\
        & \leq 4 \a^* h,
    \end{align*}
    whereby
    \begin{align*}
        \left |   \psi_q(t+h) - \psi_q(t) \right | \leq \sum_{k \geq 0} 4 \a^* h \, \P_{(x,n)}(\eta'_t = k) = 4 \a^* h \underset{h \searrow 0}{\longrightarrow} 0.
    \end{align*}
On the other hand, we obtain with the same calculations for $h \in [0,t]$  
    \begin{align*}
        \psi_q(t) - \psi_q(t-h) & = \E_{(x,n)} \left [  \E_{(x,n)} \left ( \ind_E(X'_{t}) \ind_{\{ q \}}(\eta'_{t}) | \ck \FF_{t-h} \right )  - \ind_E(X'_{t-h}) \ind_{\{ q \}}(\eta'_{t-h}) \right ] \\
        & = \E_{(x,n)} \left [ \ck Q_h((X'_{t-h},\eta'_{t-h}); E \times \{ q\}) - \ind_E(X'_{t-h}) \ind_{\{q\}}(\eta'_{t-h})  \right ] \\
        & = \sum_{k \geq 0} \E_{(x,n)} \left[ \ck Q_h((X'_{t-h},k);E \times \{ q \}) - \ind_{\{q\}}(k) |\eta'_{t-h} = k   \right ] \, \ck Q_{t-h}((x,n);E \times \{ k \} ) \\
        & \leq 4 \a^* h \underset{h \searrow 0}{\longrightarrow} 0.
    \end{align*}
Therefore the function $t \in \R_+ \mapsto  \psi_q(t)$ is continuous.

\subsection{Proof of Lemma \ref{lem41}}
Take $h >0.$ Recall  from \eqref{accroiss psi} that 
    \begin{align}\label{derivpsi}
          \dfrac{1}{h} \left ( \psi_q(t+h) - \psi_q(t) \right ) 
        & = \dfrac{1}{h} \sum_{k \geq 0} \E_{(x,n)} \left[ \ck Q_h((X'_t,k);E \times \{ q \}) - \ind_{\{q\}}(k) |\eta'_{t}=k \right ] \, \ck Q_t((x,n);E \times \{ k \} ).
    \end{align}
For any $y \in E$ and $ k \geq 0$ 
    \begin{equation}\label{A123}
     \ck Q_h((y,k),E \times \{ q \})  - \ind_{\{q\}}(k)   =\E_{(y,k)} \left (  \ind_{\{q\}}(\eta'_h) - \ind_{\{q\}}(k)  \right ) =A_1(h) + A_2(h) + A_3(h),
         \end{equation}
where $A_1(h)=\E_{(y,k)} \left ( ( \ind_{\{ q \}}(\eta'_h) - \ind_{\{q\}}(k) ) \ind_{ \ck T_1 > h} \right )$, 
$A_2(h)=\E_{(y,k)} \left ( ( \ind_{\{ q \}}(\eta'_h)  - \ind_{\{q\}}(k) ) \ind_{ \ck N_h = 1} \right )$
and $A_3(h)= \E_{(y,k)} \left ( ( \ind_{\{ q \}}(\eta'_h)  - \ind_{\{q\}}(k) ) \ind_{ \ck T_2 < h} \right ).$
Let us treat each term separately.

   First,
    \begin{align*}
         A_1(h)= \E_{(y,k)} \left ( ( \ind_{\{ q \}}(k) - \ind_{\{q\}}(k) ) \ind_{ \ck T_1 > h} \right ) = 0.
    \end{align*}
    
    Second, $A_2(h)$ reads
    \begin{align*}
       &\E_{(y,k)} \left (( \ind_{\{ q \}}(\eta'_h)  - \ind_{\{q\}}(k)) \ind_{ \ck \t_1 \leq h} \ind_{\ck \t_2 > h- \ck \t_1}  \right ) \\
       &=  \E_{(y,k)} \left [  \E_{(y,k)} \left [ (\ind_{\{ q \}}(\eta'_{\ck \t_1})   - \ind_{\{q\}}(k))) \ind_{ \ck \t_1 \leq h} \ind_{\ck \t_2 > h- \ck \t_1} \left | \ck \FF_{\ck \t_1}, \ck Y^{(1)} \right. \right ]   \right ] \\
         & = \E_{(y,k)} \left [ ( \ind_{\{ q \}}(\eta'_{\ck \t_1})   - \ind_{\{q\}}(k)) \ind_{ \ck \t_1 \leq h}  \P_{(y,k)} \left ( \ck \t_2 > h- \ck \t_1 \left | \ck \FF_{\ck \t_1}, \ck Y^{(1)} \right. \right )   \right ] \\
         & =  \E_{(y,k)} \left [  ( \ind_{\{ q \}}(\eta'_{\ck \t_1})   - \ind_{\{q\}}(k))   \ind_{ \ck \t_1 \leq h}  \e^{- \int_0^{h-\ck \t_1} \ck \a \left ( \ck Y_u^{(1)} \right ) \, \dd u }  \right ] \\
         & = \E_{(y,k)} \left [ \ind_{ \ck \t_1 \leq h} ( \ind_{\{ q \}}(\eta'_{\ck \t_1})   - \ind_{\{q\}}(k)) \E_{(y,k)} \left [   \e^{- \int_0^{h-\ck \t_1} \ck \a \left ( \ck Y_u^{(1)} \right ) \, \dd u }  \left | \ck \FF_{\ck \t_1} \right. \right ] \right ] \\
         & = \, \E_{(y,k)} \left [ \ind_{ \ck \t_1 \leq h}  (\ind_{\{ q \}}(\eta'_{\ck \t_1})  - \ind_{\{q\}}(k)) \E_{\ck C_{\ck \t_1}}^{\ck Y} \left [\e^{ - \int_0^{h-\ck \t_1} \ck \a \left ( \ck Y_u \right ) \, \dd u }  \right ] \right ] \\
         & =  \E_{(y,k)} \left [ \ind_{ \ck \t_1 \leq h}  (\ind_{\{ q \}}(\eta'_{\ck \t_1})   - \ind_{\{q\}}(k)) \E_{\ck C_{\ck \t_1}}^{\ck Y} \left [\e^{ - \int_0^{h-\ck \t_1} \ck \a \left ( \ck Y_u \right ) \, \dd u } -1 \right ] \right ]  +  \E_{(y,k)} \left [ \ind_{ \ck \t_1 \leq h}  (\ind_{\{ q \}}(\eta'_{\ck \t_1})  - \ind_{\{q\}}(k))  \right ]. 
         \end{align*}
  For the first term above, 
         \begin{align*}
             \dfrac{1}{h} \left | \E_{(y,k)} \left [ \ind_{ \ck \t_1 \leq h}  (\ind_{\{ q \}}(\eta'_{\ck \t_1}) - \ind_{\{q\}}(k))) \E_{\ck C_{\ck \t_1}}^{\ck Y} \left [\e^{ - \int_0^{h-\ck \t_1} \ck \a \left ( \ck Y_u \right ) \, \dd u } -1 \right ] \right ]  \right | & \leq 4 \a^* \E_{(y,k)}(\ind_{ \ck \t_1 \leq h} ) \leq 8 (\a^*)^2 h.
         \end{align*}
   For the second term, we have 
         \begin{align}\label{A2trans}
             \E_{(y,k)}  \left [ \ind_{ \ck \t_1 \leq h}  (\ind_{\{ q \}}(\eta'_{\ck \t_1})  - \ind_{\{q\}}(k))  \right ] 
             & =  \E_{(y,k)} \left [ \ind_{ \ck \t_1 \leq h}  \E_{(y,k)} \left [ \,  \ind_{\{ q \}} (\eta'_{\ck \t_1})  - \ind_{\{q\}}(k) \left | \ck Y^{(0)}, \ck \t_1 \right. \right ] \right ] \nonumber \\
              & =  \E_{(y,k)} \left [ \ind_{ \ck \t_1 \leq h} \left ( \int_{z_1 \in E} \ck K((Y_{\ck \t_1}^{'(0)},k); \dd z_1 \times \{ q \})   - \ind_{\{q\}}(k)  \right ) \right ] \nonumber \\
              & =  \E_{(y,k)} \left [ \ind_{ \ck \t_1 \leq h}  \left ( \ck K((Y_{\ck \t_1}^{'(0)},k); E \times \{ q \})  - \ind_{\{q\}}(k)  \right ) \right ]. 
              \end{align}
Following the definition of $\ck K$ in Section~\ref{construction coupling}, this formula takes two forms depending on whether $y \notin E_k$ or $y\in E_k$.  If $y \notin E_k$, then 
          \begin{multline*}
             \E_{(y,k)}  \left [ \ind_{ \ck \t_1 \leq h}  (\ind_{\{ q \}}(\eta'_{\ck \t_1})  - \ind_{\{q\}}(k))  \right ]   = \E_{(y,k)} \left [ \ind_{ \ck \t_1 \leq h} \left ( \dfrac{\a \left ( Y_{\ck \t_1}^{'(0)} \right )}{\ck \a \left ( Y_{\ck \t_1}^{'(0)},k \right ) }  -1 \right ) \right ] \ind_{\{ q \}}(k) \\ + \E_{(y,k)} \left [ \ind_{ \ck \t_1 \leq h}  \dfrac{\b_k}{\ck \a \left ( Y_{\ck \t_1}^{'(0)},k \right )}  \right ]  \ind_{\{ q-1 \}}(k)  + \E_{(y,k)} \left [ \ind_{ \ck \t_1 \leq h} \dfrac{\d_k}{\ck \a \left ( Y_{\ck \t_1}^{'(0)},k \right )}  \right ]  \ind_{\{ q+1 \}}(k), 
             \end{multline*}
that is             
              \begin{multline}\label{A2trans2}
             \E_{(y,k)}  \left [ \ind_{ \ck \t_1 \leq h}  (\ind_{\{ q \}}(\eta'_{\ck \t_1})  - \ind_{\{q\}}(k))  \right ]  \\ = \E_{(y,k)} \left [ \ind_{ \ck \t_1 \leq h}  \dfrac{1}{\ck \a \left ( Y_{\ck \t_1}^{'(0)},k \right ) }   \right ] \left ( - \a_k \ind_{ \{ q \}}(k)  + \b_k \ind_{ \{ q-1 \}}(k) + \d_k \ind_{ \{ q +1 \}}(k) \right ).
      \end{multline}
 Since
    \begin{align*}
        \left | \dfrac{1}{h} \E_{(y,k)} \left [ \ind_{ \ck \t_1 \leq h} \dfrac{1}{\ck \a \left ( Y_{\ck \t_1}^{'(0)},k \right )}  \right ] -1 \right | & = \left | \dfrac{1}{h} \E_{(y,k)}^{\ck Y} \left [ \int_0^h  \left ( \e^{- \int_0^s \ck \a(Y'_u,k) \, \dd u}-1  \right ) \dd s \right ] \right |   \leq 2 \a^* h,
    \end{align*}
    we then conclude that 
    \begin{equation}\label{A2sup}
        \underset{k \geq 0, y \notin E_k}{\sup} \left | \frac{A_2(h)}{h} + \a_k  \ind_{\{q\}}(k) - \b_k \ind_{\{q-1\}}(k) - \d_k \ind_{\{ q+1 \}}(k) \right | \underset{h \searrow 0}{\longrightarrow} 0 .
    \end{equation}
    If now $y \in E_k$, we  obtain from \eqref{A2trans}
     \begin{multline*}
             \E_{(y,k)}  \left [ \ind_{ \ck \t_1 \leq h}  (\ind_{\{ q \}}(\eta'_{\ck \t_1})  - \ind_{\{q\}}(k))  \right ]   \\=
              \E_{(y,k)} \left [ \ind_{ \ck \t_1 \leq h} \dfrac{\b(Y_{\ck \t_1}^{'(0)})}{\ck \a(Y_{\ck \t_1}^{'(0)},k)}  \right ] \ind_{\{ q-1 \}}(k)   +  \E_{(y,k)} \left [ \ind_{ \ck \t_1 \leq h}   \dfrac{\b_k-\b(Y_{\ck \t_1}^{'(0)})}{\ck \a(Y_{\ck \t_1}^{'(0)},k)}  \right ] \ind_{\{ q-1 \}}(k) \\
         + \E_{(y,k)} \left [ \ind_{ \ck \t_1 \leq h} \dfrac{\d_k}{\ck \a(Y_{\ck \t_1}^{'(0)},k)}    \right ] \ind_{\{ q+1 \}}(k)  +  \E_{(y,k)} \left [ \ind_{ \ck \t_1 \leq h} \left ( \dfrac{\d(Y_{\ck \t_1}^{'(0)})-\d_k}{\ck \a(Y_{\ck \t_1}^{'(0)},k)} - 1 \right ) \right ] \ind_{\{ q \}}(k),
          \end{multline*}
   that is the same expression as \eqref{A2trans2}. The convergence \eqref{A2sup} then remains true when the supremum is taken over $y\in E_k$, and so over $y\in E$, i.e. 
       \begin{align*}
        \underset{k \geq 0, y \in E}{\sup} \left | \frac{A_2(h)}{h}  + \a_k  \ind_{\{q\}}(k) - \b_k \ind_{\{q-1\}}(k) - \d_k \ind_{\{ q+1 \}}(k) \right | \underset{h \searrow 0}{\longrightarrow} 0 .
    \end{align*}
    
   Third, for $A_3(h)$ in \eqref{A123}, we have, using \eqref{dominationpoisson} and denoting $\ck N_h^* \sim \PP(2\a^*h)$,
    \begin{align*}
        \frac{1}{h} \left | A_3(h) \right | & \leq \frac{2}{h} \P_{(y,k)} \left ( \ck N_h \geq 2 \right ) \\
        & \leq \frac{2}{h} \P \left ( \ck N_h^* \geq 2 \right ) \\
        & \leq \dfrac{2}{h} \left ( 1 - \e^{-2\a^*h}- 2\a^*h \e^{-2\a^*h} \right ) \\
        & = 4 \, (\a^*)^2 \, h + \underset{h \searrow 0}{o}(h).
    \end{align*}
    
    Combining the results for $A_1(h)$, $A_2(h)$ and $A_3(h)$ in \eqref{A123}, we get
    \begin{align*}
        \underset{(y,k) \in \ck E}{\sup} \left |\dfrac{1}{h} \left ( \ck Q_h((y,k);E \times \{ q \}) - \ind_{\{q\}}(k) \right )  + \ind_{ \{ q \} }(k) \a_k  - \b_k \ind_{\{q-1\}}(k) - \d_k \ind_{\{ q+1 \}}(k)  \right | \underset{h \searrow 0}{\longrightarrow} 0.
    \end{align*}

  Finally, coming back to \eqref{derivpsi},  we obtain by uniform convergence, for any $x \in  E$,
    \begin{align*}
  \dfrac{1}{h} \left ( \psi_q(t+h) - \psi_q(t) \right ) 
         &  \underset{h \searrow 0}{\longrightarrow} \sum_{k \geq 0} \left \{  - \a_k \ind_{q}(k) + \b_k \ind_{\{q-1\}}(k) + \d_k \ind_{\{ q+1 \}}(k) \right \} \, \ck Q_t((x,n);E \times \{ k \} ) ,
        \end{align*}
 where the limit reads $- \a_q \, \psi_q(t) + \b_{q-1} \, \psi_{q-1}(t) + \d_{q+1} \,  \psi_{q+1}(t)$ using the convention $\b_{-1}=0$.

\subsection{Proof of Corollary \ref{Marg1diff}}
 For $t\geq 0$, define  
    \begin{align*}
G(t) = &  \psi_q(t) - \ind_{q}(n)  - \int_0^t \left (  - \a_q \, \psi_q(s) + \b_{q-1} \, \psi_{q-1}(s) + \d_{q+1} \,  \psi_{q+1}(s) \right ) \, \dd s .
    \end{align*}
   Then $G$ is continuous, right-differentiable on $\R_+$ by Lemma \ref{lem41}, and $\partial_+G(t)/\partial t =0$. So $G$ is constant. But $G(0)=0$ because $s\mapsto \psi_q(s) $ is bounded  on $\R_+$ for any $q \geq 0$. As a consequence we obtain 
        \begin{align*}
             \psi_q(t) = \ind_{q}(n) + \int_0^t \left (  - \a_q \, \psi_q(s) + \b_{q-1} \, \psi_{q-1}(s) + \d_{q+1} \,  \psi_{q+1}(s) \right ) \, \dd s .
        \end{align*}
    In particular the integrand is continuous by Lemma \ref{QqisC0} so  $\psi_q$  is differentiable.

\subsection{Proof of Lemma~\ref{wsdiff}}
Let us develop $q_s \ind_{\{q\}}$ as
    \begin{equation}\label{qindq}
     q_s \ind_{\{q\}} = \sum_{k \geq 0} q_s(k,\{ q \}) \ind_{\{ k \}} =  \sum_{k \geq 0} \P_k(\eta_s=q) \ind_{\{ k \}}. 
     \end{equation}
    Take $p > q$. Then using  \eqref{dominationpoisson} by denoting $n^*_t \sim \PP(\a^* t)$
    \begin{align*}
        \sum_{k=p}^{\infty} \P_k(\eta_s=q) & \leq \sum_{k=p}^{\infty} \P_k (n_s > k-q) \\
        & \leq \sum_{k=p}^{\infty} \P_k (n_t > k-q) \\
        & \leq \sum_{k=p}^{\infty} \P (n^*_t > k-q)  \\
        & = \sum_{j = p-q}^{\infty} \P(n^*_t > j) \underset{p \rightarrow \infty}{\longrightarrow} 0,
    \end{align*}
    because $\E(n^*_t) < \infty.$
    Coming back to \eqref{qindq}, we thus have that for any $\ep >0$, there exists $p \geq 0$ such that any $d \geq p$ satisfies 
    \begin{equation}\label{supq}
     \sup_{s \in [0,t]} \left \| q_s \ind_{\{q\}} - \sum_{k =0}^d q_s(k,\{ q \}) \ind_{\{ k \}} \right \|_{\infty} < \ep. 
     \end{equation}
    Since $\ck Q_t$ is a continuous linear operator on $\MM_b(E \times \N)$
    \begin{align}\label{wsseries}
        w_s & = \ck Q_{t-s} ( \ind_E \times q_s \ind_{\{q\}}) \nonumber \\
        & = \ck Q_{t-s} \left ( \ind_E \times \lim \limits_{p \rightarrow \infty} \sum_{k =0}^p q_s(k,\{ q \}) \ind_{\{ k \}}  \right ) \nonumber \\
        & = \lim \limits_{p \rightarrow \infty} \ck Q_{t-s} \left ( \ind_E \times \sum_{k =0}^p q_s(k,\{ q \}) \ind_{\{ k \}} \right ) \nonumber \\
        & = \lim \limits_{p \rightarrow \infty} \sum_{k =0}^p q_s(k,\{ q \}) \ck Q_{t-s} \left (  \ind_E \times  \ind_{\{ k \}}  \right ) \nonumber \\
        & = \sum_{k=0}^{\infty}  q_s(k,\{ q \}) \ck Q_{t-s} \left (  \ind_E \times  \ind_{\{ k \}}   \right ).
    \end{align}
    Let $\phi_k(s) = q_s(k,\{ q \}) \ck Q_{t-s} \left (  \ind_E \times  \ind_{\{ k \}}   \right ). $  From the Kolmogorov backward equation \eqref{kbebdm}, we deduce that $\partial q_s(k,\{ q \})/\partial s =- \a_k q_s(k,\{ q \}) + \b_{k} q_s(k+1,\{ q \}) + \d_k q_s(k-1,\{ q \})$. Using in addition Corollary~\ref{Marg1diff}, we deduce that 
 $\phi_k$ is differentiable  and 
     \begin{align*}
        \phi'_k(s) = & \left [ - \a_k q_s(k,\{ q \}) + \b_{k} q_s(k+1,\{ q \}) + \d_k q_s(k-1,\{ q \})  \right ] \ck Q_{t-s}\left (  \ind_E  \times  \ind_{\{ k \}} \right ) \\
        & + q_s(k,\{ q \}) \left [  \a_k \ck Q_{t-s}\left (  \ind_E  \times \ind_{\{ k \}} \right ) - \b_{k-1} \ck Q_{t-s}\left (  \ind_E  \times \ind_{\{ k-1 \}} \right ) - \d_{k+1} \ck Q_{t-s}\left (  \ind_E  \times \ind_{\{ k +1 \}} \right )  \right ].
    \end{align*}
    Since $$  \sup_{s \in [0,t]} \left \| \ck Q_{t-s}\left (  \ind_E  \times  \ind_{\{ k \}} \right ) \right \|_{\infty} \leq 1  $$
    and
    $$  \sup_{s \in [0,t]} \left \|  \a_k \ck Q_{t-s}\left ( \ind_E  \times  \ind_{\{ k \}} \right ) - \b_{k-1} \ck Q_{t-s}\left (   \ind_E  \times \ind_{\{ k-1 \}} \right ) - \d_{k+1} \ck Q_{t-s}\left (   \ind_E  \times \ind_{\{ k +1 \}} \right ) \right \|_{\infty} \leq 3 \a^* ,$$
    we can show similarly as for \eqref{supq} that
    $$ \sup_{s \in [0,t]} \left \| \sum_{k \geq p} \phi'_k(s) \right \|_{\infty} \underset{p \rightarrow \infty}{\longrightarrow} 0.   $$
   Since by \eqref{wsseries} $w_s=\sum_{k\geq 0}  \phi_k(s)$, we deduce that $w_s$ is differentiable on $[0,t]$ and 
    \begin{align*}
        \dfrac{\partial }{\partial s} w_s  = &\sum_{k =0}^{\infty}   \phi'_k(s) \\
       = &\sum_{k=0}^{\infty}  \left [ \b_k q_s(k+1,\{ q \}) \ck Q_{t-s}\left (   \ind_E  \times\ind_{\{ k \}} \right ) - \b_{k-1} q_s(k,\{ q \}) \ck Q_{t-s}\left (  \ind_E  \times \ind_{\{ k-1 \}} \right ) \right ] \\
        &\hspace{1cm} + \sum_{k=0}^{\infty}  \left [ \d_k q_s(k-1,\{ q \}) \ck Q_{t-s}\left (  \ind_E  \times \ind_{\{ k \}} \right ) - \d_{k+1} q_s(k,\{ q \}) \ck Q_{t-s}\left (   \ind_E  \times\ind_{\{ k +1 \}} \right )  \right ],
    \end{align*}
    where $\b_{-1}=\d_0=0$.  The first of these two telescoping series vanishes because $\b_{-1}=0$ and
    $$  \left \|  \b_k q_s(k+1,\{ q \}) \ck Q_{t-s}\left (   \ind_E  \times \ind_{\{ k \}} \right ) \right \|_{\infty}  \leq \a^* q_s(k+1,\{ q \}) \leq \a^* \, \P(n_t^* > k+1-q) \to 0.$$
   The second series vanishes by similar arguments  and we have $\partial w_s/\partial s \equiv0$.

\subsection{Proof of Lemma \ref{marg2C0}}   
       Let $h > 0$, then 
    \begin{align}\label{accrois psi}
    \psi_f(t+h)-\psi_f(t)  & = \E_{(x,n)}\left ( f(X'_{t+h}) \ind_\N(\eta'_{t+h}) \right ) - \E_{(x,n)}\left ( f(X'_{t}) \ind_\N (\eta'_t) \right ) \nonumber \\
        & = \E_{(x,n)} \left [ \E_{(x,n)} \left ( f(X'_{t+h}) \ind_\N(\eta'_{t+h}) \left | \ck \FF_t \right. \right ) \right] - \E_{(x,n)}\left ( f(X'_{t}) \ind_\N (\eta'_t) \right ) \nonumber \\
        & = \E_{(x,n)} \left [ \ck Q_h(f \times \ind_\N)(X'_t,\eta'_t) - f(X'_t)\ind_\N(\eta'_t) \right ].
    \end{align}
    For any $y \in E$ and $k \in \N$ one has 
    \begin{align*}
        \left | \ck Q_h(f \times \ind_\N)(y,k) - f(y)\ind_\N(k) \right | & = \left | \E_{(y,k)} \left ( f(X'_h) \right ) - f(y) \right | \\
        & = \left | \E_{(y,k)} \left ( f(X'_h) (\ind_{\ck T_1 \leq h} + \ind_{\ck T_1 > h} ) \right ) - f(y) \E_{(y,k)} \left (\ind_{\ck T_1 \leq h} + \ind_{\ck T_1 > h}  \right ) \right | \\
        &  \leq 2 \|f\|_{\infty} \P_{(y,k)}(\ck T_1 \leq h) + \left | \E_{(y,k)} \left (  (f(Y^{'(0)}_h)-f(y))  (1-\ind_{\ck T_1 \leq h}) \right ) \right | \\
        & \leq 8 \a^* \|f\|_{\infty} h + \left | \E_{y}^Y \left [ f(Y_h)-f(y)\right ] \right | \\
        & \leq 8 \a^* \|f\|_{\infty} h  +   \left \| Q_h^Yf-f \right \|_{\infty},
    \end{align*}
  where we have used \eqref{QtYcheck} in the second last step. 
     Coming back to \eqref{accrois psi}, we deduce that
    \begin{align*}
        \left | \psi_f(t+h)-\psi_f(t) \right |  \leq 8 \a^* \|f\|_{\infty} h  +   \left \| Q_h^Yf-f \right \|_{\infty} .
    \end{align*}
    As a Feller process, $(Y_t)_{t \geq 0}$ is strongly continuous at $0$, whereby $\psi_f(t+h)\to \psi_f(t)$ as $h \searrow 0$. 
    
  On the other hand, for $h \in [0,t]$,  
    \begin{align*}
        \left | \psi_f(t)-\psi_f(t-h) \right | & = \left | \E_{(x,n)}\left ( f(X'_{t}) \ind_\N(\eta'_{t}) \right ) - \E_{(x,n)}\left ( f(X'_{t-h}) \ind_\N (\eta'_{t-h}) \right ) \right | \\
        & = \left | \E_{(x,n)} \left [ \E_{(x,n)} \left ( f(X'_{t}) \ind_\N(\eta'_{t}) \left | \ck \FF_{t-h} \right. \right ) \right] - \E_{(x,n)}\left ( f(X'_{t-h}) \ind_\N (\eta'_{t-h}) \right ) \right | \\
        & = \left | \E_{(x,n)} \left [ \ck Q_h(f \times \ind_\N)(X'_{t-h},\eta'_{t-h}) - f(X'_{t-h})\ind_\N(\eta'_{t-h}) \right ] \right | \\
        & \leq  8 \a^* \|f\|_{\infty} h  +   \left \| Q_h^Yf-f \right \|_{\infty},
    \end{align*}
    and we conclude similarly that $\psi_f(t-h)\to \psi_f(t)$ as $h \searrow 0$.

\subsection{Proof of Lemma \ref{suptendsvers0}}

Let $h>0$. For any $t \geq 0$
    \begin{align*}
     \left | \dfrac{ \psi_f(t+h)-\psi_f(t) }{h} -   \psi_{\AA f}(t) \right |  & = \left | \ck Q_t \left ( \dfrac{\ck Q_h(f \times \ind_\N)(x,n) - f(x) \ind_\N(n)}{h} - \AA f(x) \times \ind_\N(n)  \right ) \right |  \\
         & \leq  \underset{(y,k)\in \ck E}{\sup} \left | \dfrac{\ck Q_h(f \times \ind_\N)(y,k) - f(y) \ind_\N(k)}{h} - \AA f(y) \times \ind_\N(k) \right |.
    \end{align*}
    The proof thus consists in showing that
        \[
    \underset{(y,k)\in \ck E}{\sup} \left | \dfrac{\ck Q_h(f \times \ind_\N)(y,k) - f(y) \ind_\N(k)}{h} - \AA f(y) \times \ind_\N(k) \right | \underset{h \searrow 0}{\longrightarrow} 0.
    \]
For any $h> 0$, $y \in E$ and $k \geq 0$ 
    \begin{align*}
         \dfrac{1}{h} \Big ( \ck Q_h(f \times \ind_\N)(y,k)& - f(y) \ind_\N(k) \Big ) \\
         & = \E_{(y,k)} \left [ \dfrac{f(X'_h) -f(y)}{h} \right ] \\
         & = \dfrac{1}{h} \left ( \E_{(y,k)} \left [ f(X'_h)  \ind_{\ck T_1 >h}  \right ] + \E_{(y,k)} \left [ f(X'_h)   \ind_{\ck N_h=1} \right ] + \E_{(y,k)} \left [ f(X'_h)   \ind_{\ck T_2 <h}   \right ] - f(y) \right ).
    \end{align*}
    But
    \begin{align*}
        \frac{1}{h} & \left (  \E_{(y,k)} \left [ f(X'_h)  \ind_{\ck T_1 >h} \right ] -f(y) \right ) \\
        &=  \frac{1}{h} \left ( \E_{(y,k)} \left [ f(Y_h^{'(0)})  \ind_{\ck T_1 >h} \right ] -f(y) \right )\\
        & = \dfrac{1}{h} \left ( \E_{(y,k)} \left [ f(Y^{'(0)}_h)  \P_{(y,k)} \left ( \ck T_1 > h \left | \ck Y^{(0)}\right. \right ) \right ] - f(y) \right ) \\
        & = \dfrac{1}{h}  \left ( \E_{(y,k)} \left [ f(Y^{'(0)}_h)  \e^{- \int_0^h \ck \a( Y_u^{'(0)},k) \, \dd u} \right ] - f(y) \right ) \\
        & =  \E_{y}^Y \left [ \dfrac{f(Y_h) -f(y)}{h} \right ] + \dfrac{1}{h} \E_{(y,k)}^{\ck Y} \left [ f(Y^{'}_h) \left ( \e^{- \int_0^h \ck \a( Y_u^{'},k) \, \dd u} -1 + \int_0^h \ck \a( Y_u^{'},k) \, \dd u \right ) \right ]  \\
        &\hspace{0.3cm} - \dfrac{1}{h} \E_{(y,k)}^{\ck Y} \left [ \int_0^h  f(Y^{'}_h) \ck \a( Y_u^{'},k) \, \dd u \right ].
    \end{align*}
Using this result and the expression of    $\AA f(y)$ from Theorem~\ref{theogenerator}, we can write 
      \begin{align*}
         \dfrac{1}{h}  \Big ( \ck Q_h(f \times \ind_\N)(y,k)& - f(y) \ind_\N(k) \Big)  - \AA f(y) \\
         &= \dfrac{1}{h}  \left ( \ck Q_h(f \times \ind_\N)(y,k) - f(y) \ind_\N(k) \right ) - \AA^Y f(y) + \a(y) f(y) - \a(y) Kf(y) \\
         &= A_1(h) + A_2(h) + A_3(h) + A_4(h) + A_5(h),
    \end{align*}
where
 \begin{align*}
 A_1(h) &=\E^Y_y \left [ \dfrac{f(Y_h) -f(y)}{h} \right ] - \AA^Y f(y), \\
  A_2(h) &= \ck \a(y,k) f(y) -  \dfrac{1}{h} \E_{(y,k)}^{\ck Y} \left [ \int_0^h  f(Y^{'}_h) \ck \a( Y_u^{'},k) \, \dd u \right ], \\
    A_3(h) &=\dfrac{1}{h} \E_{(y,k)}^{\ck Y} \left [ f(Y^{'}_h) \left ( \e^{- \int_0^h \ck \a( Y_u^{'},k) \, \dd u} -1 + \int_0^h \ck \a( Y_u^{'},k) \, \dd u \right ) \right ], \\
     A_4(h) &=\dfrac{1}{h} \E_{(y,k)} \left [ f(X'_h) \ind_{\ck N_h=1}  \right] - \a(y) Kf(y) + (\a(y)-\ck \a(y,k))f(y), \\
     A_5(h) &=\dfrac{1}{h} \E_{(y,k)} \left [ f(X'_h) \ind_{\ck T_2<h}  \right].
  \end{align*}

The end of the proof consists in proving that each of these five terms tends  uniformly to 0 as $h \searrow 0$.

 For the first one, note that
    \begin{align*}
      \E^Y_y \left [ \dfrac{f(Y_h) -f(y)}{h} \right ] = \dfrac{Q_h^Y f(y)-f(y)}{h} ,
    \end{align*}
    and since $f \in \DD_\AA^Y$, by the definition of $\AA^Y$, $\sup_{(y,k) \in \ck E} \left | A_1(h)  \right |$ tends to 0 as $h\to 0$.
    
To show that ${\sup}_{(y,k) \in \ck E} | A_2(h)| \underset{h \searrow 0}{\longrightarrow} 0
$, we consider two cases whether $y \notin E_k$ or $y \in E_k$. 
     First suppose that $y \notin E_k$. Then  $\ck \a (y,k) = \a (y) + \a_k$ and
    \begin{align*}
     A_2(h) &= \dfrac{1}{h} \int_0^h \E_{(y,k)}^{\ck Y} \left [  \ck \a(y,k) f(y)  -   f(Y^{'}_h) \ck \a( Y_u^{'},k) \right ]  \, \dd u \\
        & = \a_k \E_{y}^Y \left [   f(y)  -   f(Y_h)  \right ]  +  \dfrac{1}{h} \int_0^h \E_{y}^Y \left [   \a(y) f(y)  -   f(Y_h)  \a( Y_u) \right ]  \, \dd u 
        \end{align*}
    where the switch from $  \E_{(y,k)}^{\ck Y}$ to $\E_{y}^Y$ is a consequence of \eqref{QtYcheck}, specifically the bivariate generalization  of it. Therefore,

        \begin{align*}       
        & A_2(h)  = \a_k \left ( f(y) - Q_h^Yf(y) \right ) + \dfrac{1}{h} \int_0^h \E_y^Y \left [ \E_y^Y \left [ \a(y) f(y)  -   f(Y_h) \a( Y_u) \left | Y_u \right. \right ] \right ]  \, \dd u \\
        & = \a_k \left ( f(y) - Q_h^Yf(y) \right ) + \dfrac{1}{h} \int_0^h \E_y^Y \left [   \a(y) f(y)  -   Q^Y_{h-u} f(Y_u)  \a( Y_u) \right ]  \, \dd u \\
        & = \a_k \left ( f(y) - Q_h^Yf(y) \right ) + \dfrac{1}{h} \int_0^h \E_y^Y \left [   \a(y) f(y)  -    f(Y_u) \a( Y_u) \right ]  \, \dd u  \\ 
        &\hspace{5cm}+ \dfrac{1}{h} \int_0^h \E_y^Y \left [    f(Y_u) \a( Y_u)-    Q^Y_{h-u} f(Y_u)  \a( Y_u) \right ]  \, \dd u \\
        & =  \a_k \left ( f(y) - Q_h^Yf(y) \right ) + \dfrac{1}{h} \int_0^h  \left ( f \times \a  -   Q_u^Y (f \times \a) \right  )(y)  \, \dd u + \dfrac{1}{h} \int_0^h \E_y^Y \left [   \a( Y_u) \left ( f -  Q^Y_{h-u} f \right )(Y_u) \right ]  \, \dd u \\
        & =  \a_k \left ( f(y) - Q_h^Yf(y) \right ) +  \int_0^1  \left ( f \times \a  -   Q_{hv}^Y (f \times \a) \right  )(y)  \, \dd v +  \int_0^1 \E_y^Y \left [   \a( Y_{hv}) \left ( f -  Q^Y_{h(1-v)} f \right )(Y_{hv}) \right ]  \, \dd v.
    \end{align*}
    So when $y \notin E_k$,
    \begin{align}\label{A2final}
        \left | A_2(h) \right | 
         \leq \a^* \|Q_h^Y f-f\|_{\infty} + \int_0^1 \|Q_{hv}^Y (f \times \a) -  f \times \a\|_{\infty} \, \dd v + \a^* \int_0^1 \|  Q^Y_{h(1-v)} f - f \|_{\infty}  \, \dd v ,
    \end{align}
    that does not depend on $(y,k) \in \ck E$ and converges to zero when $h$ goes to $0$ by the dominated convergence theorem because $f \in \DD_{\AA}^Y \subset C_0(E)$. 
    When $y \in E_k$, $\ck \a(y,k) = \b_k + \d(y)$ and we obtain with the same computations the same inequality \eqref{A2final}, leading to the same convergence. So ${\sup}_{(y,k) \in \ck E} | A_2(h)| \underset{h \searrow 0}{\longrightarrow} 0$.

   Regarding $A_3(h)$, its uniform convergence towards 0 is easily obtained from 
       \begin{align*}
        \left | A_3(h) \right | & \leq \dfrac{\|f\|_{\infty}}{2h} \E_{(y,k)}^{\ck Y}  \left [ \left ( \int_0^h \ck \a( Y_u^{'},k) \, \dd u \right )^2 \right ] 
         \leq  \dfrac{\|f\|_{\infty}  \, (2 \a^* h)^2}{2h} = 2 h \|f\|_{\infty} ( \a^*)^2.
    \end{align*}
    
   Let us now prove that ${\sup}_{(y,k) \in \ck E} | A_4(h)| \underset{h \searrow 0}{\longrightarrow} 0$.
    We compute 
    \begin{align*}
         \dfrac{1}{h} & \E_{(y,k)} \left [ f(X'_h) \ind_{ \ck N_t=1}   \right ] =  \dfrac{1}{h} \E_{(y,k)} \left [ f(X'_h) \ind_{ \ck \t_1 \leq h} \ind_{\ck \t_2 > h- \ck \t_1}  \right ] \\
        & =  \dfrac{1}{h} \E_{(y,k)} \left [  \E_{(y,k)} \left [ f(X'_h) \ind_{ \ck \t_1 \leq h} \ind_{\ck \t_2 > h- \ck \t_1}  \left | \ck \FF_{\ck \t_1}, \ck Y^{(1)} \right. \right ]   \right ] \\
         & =  \dfrac{1}{h} \E_{(y,k)} \left [ f(Y^{'(1)}_{h-\ck \t_1}) \ind_{ \ck \t_1 \leq h}  \P_{(y,k)} \left ( \ck \t_2 > h- \ck \t_1 \left | \ck \FF_{\ck \t_1}, \ck Y^{(1)} \right. \right )   \right ] \\
         & =  \dfrac{1}{h} \E_{(y,k)} \left [  f(Y^{'(1)}_{h-\ck \t_1})   \ind_{ \ck \t_1 \leq h}  \e^{ - \int_0^{h-\ck \t_1} \ck \a \left ( \ck Y_u^{(1)} \right ) \, \dd u } \right ] \\
         & =  \dfrac{1}{h} \E_{(y,k)} \left [ \ind_{ \ck \t_1 \leq h}  \E_{(y,k)} \left [     f(Y^{'(1)}_{h-\ck \t_1})  \e^{ - \int_0^{h-\ck \t_1} \ck \a \left ( \ck Y_u^{(1)} \right ) \, \dd u } \left | \ck \FF_{\ck \t_1} \right. \right ] \right ] \\
         & =  \dfrac{1}{h}  \E_{(y,k)} \left [ \ind_{ \ck \t_1 \leq h}   \E^{\ck Y}_{\ck C_{\ck \t_1}} \left [ f(Y'_{h-\ck \t_1})     \e^{ - \int_0^{h-\ck \t_1} \ck \a \left ( \ck Y_u \right ) \, \dd u }  \right ] \right ] \\
         & = \dfrac{1}{h}  \E_{(y,k)} \left [ \ind_{ \ck \t_1 \leq h}   \E^{\ck Y}_{\ck C_{\ck \t_1}} \left [ f(Y'_{h-\ck \t_1})     \right ] \right ] + \dfrac{1}{h}  \E_{(y,k)} \left [ \ind_{ \ck \t_1 \leq h}   \E^{\ck Y}_{\ck C_{\ck \t_1}} \left [ f(Y'_{h-\ck \t_1})  \left (   \e^{ - \int_0^{h-\ck \t_1} \ck \a \left ( \ck Y_u \right ) \, \dd u } -1 \right ) \right ] \right ] .
    \end{align*}
    The second term converges uniformly to $0$ because its norm is bounded by    
    $$ \E_{(y,k)}  \left [ \dfrac{\ind_{\ck \t_1 \leq h} \, \| f \|_{\infty}}{h}  \E^{\ck Y}_{\ck C_{\ck \t_1}} \left | \int_0^{h-\ck \t_1} \ck \a(\ck Y_u) \, \dd u \right | \right ] 
          \leq \frac{2 h \a^* \| f \|_{\infty}}{h} \P_{(y,k)}(\ck \t_1 \leq h) \leq 4 h  (\a^*)^2 \| f \|_{\infty}.$$
  Let us prove that the first term converges uniformly to $\a(y)Kf(y) -  (\a(y)-\ck \a(y,k))f(y)$, proving that ${\sup}_{(y,k) \in \ck E} | A_4(h)| \underset{h \searrow 0}{\longrightarrow} 0$. We have
    \begin{align*}
        \dfrac{1}{h}  \E_{(y,k)} \left [ \ind_{ \ck \t_1 \leq h}   \E^{\ck Y}_{\ck C_{\ck \t_1}} \left [ f(Y'_{h-\ck \t_1})     \right ] \right ] & = \dfrac{1}{h} \E_{(y,k)}  \left [  \ind_{ \ck \t_1 \leq h}  \E_{(y,k)}  \left [ \E^{\ck Y}_{\ck C_{\ck \t_1}} \left [  f(Y'_{h-\ck \t_1}) \right ] \left | \ck Y^{(0)}, \ck \t_1 \right. \right ] \right ] \\
         & = \dfrac{1}{h} \E_{(y,k)} \left [ \ind_{ \ck \t_1 \leq h}  \int_{z_1 \in E} \sum_{q \geq 0}  \E^{ \ck Y}_{(z_1,q)} \left [  f(Y'_{h-\ck \t_1}) \right ]      \ck K \left ( (Y_{\ck \t_1}^{'(0)},k) ; \dd  z_1 \times \{ q \} \right ) \right ] .
    \end{align*}
We separate as before the cases  $y \notin E_k$ and $y\in E_k$. If $y \notin E_k$, we obtain 
    \begin{align}\label{A4temp}
         \dfrac{1}{h}  \E_{(y,k)} \left [ \ind_{ \ck \t_1 \leq h}   \E^{\ck Y}_{\ck C_{\ck \t_1}} \left [ f(Y'_{h-\ck \t_1})     \right ] \right ] & =  \dfrac{1}{h} \E_{(y,k)} \left [ \ind_{ \ck \t_1 \leq h} \dfrac{\a \left ( Y_{\ck \t_1}^{'(0)} \right )}{\ck \a \left ( Y_{\ck \t_1}^{'(0)},k \right )} \int_E \E^{ \ck Y}_{(z_1,k)} \left [  f(Y'_{h-\ck \t_1}) \right ] K(Y_{\ck \t_1}^{'(0)}, \dd z_1)  \right ]  \nonumber\\
         & \hspace{0.5cm}+  \dfrac{1}{h} \E_{(y,k)} \left [ \ind_{ \ck \t_1 \leq h}  \dfrac{\b_k}{\ck \a \left ( Y_{\ck \t_1}^{'(0)},k \right )} \E^{\ck Y}_{ (Y_{\ck \t_1}^{'(0)},k+1)} \left [  f(Y'_{h-\ck \t_1}) \right ]  \right ]   \nonumber \\
         & \hspace{0.5cm} +  \dfrac{1}{h} \E_{(y,k)} \left [ \ind_{ \ck \t_1 \leq h} \dfrac{\d_k}{\ck \a \left ( Y_{\ck \t_1}^{'(0)},k \right ) } \E^{\ck Y}_{ (Y_{\ck \t_1}^{'(0)},k-1)} \left [  f(Y'_{h-\ck \t_1}) \right ] \right ] .
    \end{align}
   Let us show that the first term in \eqref{A4temp} converges to $\a(y) Kf(y)$:
    \begin{align*}
          \dfrac{1}{h} \E_{(y,k)} &\left [ \ind_{ \ck \t_1 \leq h} \dfrac{\a \left ( Y_{\ck \t_1}^{'(0)} \right )}{\ck \a \left ( Y_{\ck \t_1}^{'(0)},k \right )} \int_E \E^{\ck Y}_{(z_1,k)} \left [ f(Y'_{h-\ck \t_1})     \right ] K(Y_{\ck \t_1}^{'(0)}, \dd z_1)  \right ] - \a(y) Kf(y) \\
         & = \dfrac{1}{h} \E_{(y,k)} \left [ \ind_{ \ck \t_1 \leq h} \dfrac{\a \left ( Y_{\ck \t_1}^{'(0)} \right )}{\ck \a \left ( Y_{\ck \t_1}^{'(0)},k \right )} \int_E Q^{\ck Y}_{h - \ck \t_1}(f\times \1_\N)(z_1,k) K(Y_{\ck \t_1}^{'(0)}, \dd z_1)  \right ]  - \a(y) Kf(y) \\ 
         & = \dfrac{1}{h} \E_{(y,k)} \left [ \E_{(y,k)} \left [ \ind_{ \ck \t_1 \leq h} \dfrac{\a \left ( Y_{\ck \t_1}^{'(0)} \right )}{\ck \a \left ( Y_{\ck \t_1}^{'(0)},k \right )} \int_E Q^{\ck Y}_{h - \ck \t_1}(f\times \1_\N)(z_1,k) K(Y_{\ck \t_1}^{'(0)}, \dd z_1) \left | \ck Y^{(0)} \right. \right ]  \right ]  - \a(y) Kf(y) \\
        & = \dfrac{1}{h}\E^{\ck Y}_{(y,k)} \left [ \int_0^h \a \left ( Y_{s}^{'} \right ) \int_E  Q^Y_{h - s}f(z_1) K(Y_{s}^{'}, \dd z_1) \e^{- \int_0^s \ck \a(Y_u^{'},k) \, \dd u} \dd s \right ]  - \a(y) Kf(y) \\
        & = \E^{\ck Y}_{(y,k)} \left [ \int_0^1 \a \left ( Y'_{hv} \right ) \int_E  Q^Y_{h(1 - v)}f(z_1) K(Y'_{hv}, \dd z_1) \e^{- \int_0^{hv} \ck \a(Y'_u,k) \, \dd u} \dd v \right ]  - \a(y) Kf(y) \\
        & = \E^{\ck Y}_{(y,k)} \left [ \int_0^1 \a \left ( Y'_{hv} \right ) \int_E  Q^Y_{h(1 - v)}f(z_1) K(Y'_{hv}, \dd z_1) \dd v \right ]  - \a(y) Kf(y) \\
        & \hspace{0.5cm} + \E^{\ck Y}_{(y,k)}  \left [ \int_0^1 \a \left ( Y'_{hv} \right ) \int_E  Q^Y_{h(1 - v)}f(z_1) K(Y'_{hv}, \dd z_1) \left ( \e^{- \int_0^{hv} \ck \a(Y'_u,k) \, \dd u} -1 \right ) \dd v \right ] 
    \end{align*}
     But for the last term,
    \begin{align*}
        \Big |  \E^{\ck Y}_{(y,k)}  \Big [ \int_0^1 \a \left ( Y'_{hv} \right ) \int_E  Q^Y_{h(1 - v)}f(z_1) &K(Y'_{hv}, \dd z_1) \left ( \e^{- \int_0^{hv} \ck \a(Y'_u,k) \, \dd u} -1 \right ) \dd v \Big ]  \Big | \\
        & \leq \a^* \| f \|_{\infty}  \E^{\ck Y}_{(y,k)}  \left [ \int_0^1 \int_0^{hv} \ck \a(Y'_u,k) \, \dd u \, \dd v \right ] 
         \leq 2 h ( \a^*)^2  \| f \|_{\infty}
    \end{align*}
    and from (the bivariate version of) \eqref{QtYcheck}, we can replace $ \E_{(y,k)}^{\ck Y}$ by $\E_{y}^Y$  in the other  term to get 
    \begin{align*}
        & \left | \E_y^Y  \left [ \int_0^1 \a( Y_{hv} ) \int_E  Q^Y_{h(1 - v)}f(z_1) K(Y_{hv}, \dd z_1) \dd v \right ]  -  \a(y) Kf(y) \right | \\
         & \leq  \int_0^1 \left |  \E_y^Y \left [ \a ( Y_{hv}  ) K  Q^Y_{h(1 - v)}f(Y_{hv}) - \a(Y_{hv}) Kf(Y_{hv}) \right ] \right | \, \dd v 
          +  \int_0^1 \left | \E_y^Y \left [ \a(Y_{hv}) Kf(Y_{hv}) - \a(y) Kf(y) \right ] \right |  \, \dd v  \\
         & \leq \a^* \int_0^1 \| KQ_{h(1-v)}^Yf - Kf \|_{\infty} \, \dd v + \int_0^1 \left | Q_{hv}^Y(\a \times Kf) (y) - (\a \times Kf)(y) \right | \, \dd v \\
         & \leq \a^* \int_0^1 \| Q_{h(1-v)}^Yf - f \|_{\infty} \, \dd v + \int_0^1 \| Q_{hv}^Y(\a \times Kf)  - (\a \times Kf) \|_{\infty} \, \dd v,
    \end{align*}
    converges to $0$ when $h$ goes to $0$ by the dominated convergence theorem, using the fact that $f \in \DD_{\AA}^Y \subset C_0(E)$ and $KC_0(E) \subset C_0(E)$. This proves the convergence to $\a(y) Kf(y)$ of the first term in \eqref{A4temp}. 
    Concerning the second  and third terms  in \eqref{A4temp}, their sum converges to  $(\beta_k+\delta_k)f(y)=(\ck \a(y,k) -  \a(y))f(y)$. Indeed for any $q$, 
        \begin{align*}
        & \dfrac{1}{h} \left | \E_{(y,k)} \left [ \ind_{ \ck \t_1 \leq h} \dfrac{1}{\ck \a \left ( Y_{\ck \t_1}^{'(0)},k \right ) } \E^{ \ck Y}_{ (Y_{\ck \t_1}^{'(0)},q)} \left [  f(Y'_{h-\ck \t_1}) \right ] \right ] - f(y) \right | \\
        & = \dfrac{1}{h} \left | \int_0^h \E^{\ck Y}_{(y,k)} \left [\E^{ Y}_{ Y_{s}^{'}} \left [  f(Y_{h-s}) \right ] \e^{- \int_0^s \ck \a(Y_u^{'},k) \, \dd u} \right ] \, \dd s - f(y) \right | \\
        & =  \left | \int_0^1  \E^{\ck Y}_{(y,k)} \left [\E^{ Y}_{ Y_{hv}^{'}} \left [  f(Y_{h(1-v)}) \right ] \e^{- \int_0^{hv} \ck \a(Y_u^{'},k) \, \dd u} \right ] \, \dd v - f(y) \right | \\
        & \leq \left | \int_0^1  \E^{\ck Y}_{(y,k)}\left [\E^{ Y}_{ Y_{hv}^{'}} \left [  f(Y_{h(1-v)}) \right ] \left ( \e^{- \int_0^{hv} \ck \a(Y_u^{'},k) \, \dd u}-1 \right ) \right ] \, \dd v \right | + \left | \int_0^1 \E_y^Y \left [\E^{ Y}_{ Y_{hv}} \left [  f(Y_{h(1-v)}) \right ]  \right ] \, \dd v - f(y) \right | \\
        & \leq \int_0^1 \|f\|_{\infty}   \E^{\ck Y}_{(y,k)} \left [ \left | \int_0^{hv} \ck \a(Y_u^{'},k) \, \dd u \right | \right] \, \dd v + \int_0^1 \left |  \E_y^Y \left [ Q_{h(1-v)}^Y f(Y_{hv}) \right ]  - f(y) \right | \, \dd v\\
        & \leq  \a^* \|f\|_{\infty}  h + \int_0^1 \left | Q_h^Yf(y)-f(y) \right | \, \dd v \\
        & \leq \a^* \|f\|_{\infty}  h + \|Q_h^Yf-f\|_{\infty},
    \end{align*}
that converges to $0$ when $h$ goes to $0$ by  the dominated convergence theorem because $f \in \DD_{\AA}^Y \subset C_0(E)$. This completes the proof of the claimed convergence of $A_4(h)$ when $y\notin E_k$. 

Suppose now that $y \in E_k.$ The development made in \eqref{A4temp} becomes in this case
    \begin{align}\label{A4temp2}
         \dfrac{1}{h}  \E_{(y,k)} \Big [ \ind_{ \ck \t_1 \leq h}  & \E^{\ck Y}_{\ck C_{\ck \t_1}} \left [ f(Y_{h-\ck \t_1})     \right ] \Big ]  \nonumber\\
         =&  \dfrac{1}{h}  \E_{(y,k)} \left [ \ind_{ \ck \t_1 \leq h} \dfrac{\b(Y_{\ck \t_1}^{'(0)})}{\ck \a(Y_{\ck \t_1}^{'(0)},k)}   \int_{E }   \E^{ \ck Y}_{(z_1,k+1)} \left [ f(Y'_{h-\ck \t_1})     \right ]  K_\b \left (  Y_{\ck \t_1}^{'(0)} , \dd z_1 \right ) \right ]   \nonumber \\
       &  +  \dfrac{1}{h} \E_{(y,k)} \left [ \ind_{ \ck \t_1 \leq h}   \dfrac{\b_k-\b(Y_{\ck \t_1}^{'(0)})}{\ck \a(Y_{\ck \t_1}^{'(0)},k)} \E^{\ck Y}_{ (Y_{\ck \t_1}^{'(0)},k+1)} \left [ f(Y'_{h-\ck \t_1})     \right ]  \right ]   \nonumber \\
       &  +  \dfrac{1}{h} \E_{(y,k)} \left [ \ind_{ \ck \t_1 \leq h} \dfrac{\d_k}{\ck \a(Y_{\ck \t_1}^{'(0)},k)}   \int_{E }    \E^{\ck Y}_{ (Y_{\ck \t_1}^{'(0)},k-1)}  \left [ f(Y'_{h-\ck \t_1})     \right ]  K_\d \left (  Y_{\ck \t_1}^{'(0)} , \dd z_1 \right ) \right ]   \nonumber\\
        &+   \dfrac{1}{h} \E_{(y,k)} \left [ \ind_{ \ck \t_1 \leq h} \dfrac{\d(Y_{\ck \t_1}^{'(0)})-\d_k}{\ck \a(Y_{\ck \t_1}^{'(0)},k)}    \int_{E }   \E^{\ck Y}_{ (Y_{\ck \t_1}^{'(0)},k)} \left [ f(Y'_{h-\ck \t_1})     \right ]  K_\d \left (  Y_{\ck \t_1}^{'(0)} , \dd z_1 \right ) \right ].
    \end{align}
    The first, third and fourth terms above can be treated exactly as the first term in \eqref{A4temp} to prove that they converge uniformly towards $\beta(y)K_\beta f(y)$, $\delta_k K_\delta f(y)$ and $(\delta(y)-\delta_k)K_\delta f(y)$, respectively, the sum of which is $\alpha(y)K f(y)$. 
  For the second term, we compute
    \begin{align*}
         \Bigg | \dfrac{1}{h} \E_{(y,k)} \Bigg [ \ind_{ \ck \t_1 \leq h}&\dfrac{\b(Y_{\ck \t_1}^{'(0)})}{\ck \a(Y_{\ck \t_1}^{'(0)},k)}  \E^{\ck Y}_{ (Y_{\ck \t_1}^{'(0)},k+1)} \left [ f(Y'_{h-\ck \t_1})     \Bigg ]  \Bigg ] - \b(y)f(y) \right | \\
        & =\left | \dfrac{1}{h} \E_{(y,k)} \left [ \E_{(y,k)} \left [ \ind_{ \ck \t_1 \leq h}   \dfrac{\b(Y_{\ck \t_1}^{'(0)})}{\ck \a(Y_{\ck \t_1}^{'(0)},k)} \E^{\ck Y}_{ (Y_{\ck \t_1}^{'(0)},k+1)} \left [ f(Y'_{h-\ck \t_1})     \right ] \left | \ck Y^{(0)} \right. \right ] \right ] - \b(y)f(y) \right | \\
        & = \left | \dfrac{1}{h} \int_0^h  \E^{\ck Y}_{(y,k)} \left [ \b(Y_s^{'})  \E^{ Y}_{Y_{s}^{'}} \left [ f(Y_{h-s}) \right ] \e^{- \int_0^s \ck \a(Y^{'}_u,k+1) \, \dd u} \right ] \, \dd s - \b(y)f(y)  \right | \\
        & =  \left |  \int_0^1  \E^{\ck Y}_{(y,k)}  \left [ \b(Y_{hv}^{'})  \E^{ Y}_{Y_{hv}^{'}} \left [ f(Y_{h(1-v)}) \right ] \e^{- \int_0^{hv} \ck \a(Y^{'}_u,k+1) \, \dd u} \right ] \, \dd v - \b(y)f(y)  \right | \\
        & \leq \left |  \int_0^1  \E^{\ck Y}_{(y,k)}  \left [ \b(Y_{hv}^{'})  \E^{ Y}_{Y_{hv}^{'}} \left [ f(Y_{h(1-v)}) \right ] \left ( \e^{- \int_0^{hv} \ck \a(Y^{'}_u,k+1) \, \dd u} -1 \right ) \right ] \, \dd v  \right | \\
        & \hspace{1cm}+  \left |  \int_0^1  \E_y^Y \left [ \b(Y_{hv})  \E^{ Y}_{Y_{hv}} \left [ f(Y_{h(1-v)}) \right ] \right ] \, \dd v - \b(y)f(y)  \right |  \\
        & \leq \a^* \| f \|_{\infty}  \int_0^1 \int_0^{hv}   \E^{Y}_{y} \left[\ck \a(Y_u,k+1)\right] \, \dd u \, \dd v +    \left |  \int_0^1  \E_y^Y \left [ \b(Y_{hv})  Q_{h(1-v)}^Y f(Y_{hv})   \right ] \, \dd v - \b(y)f(y)  \right |  \\
        &  \leq  (\a^*)^2 \| f \|_{\infty} h + \left |  \int_0^1  \E_y^Y \left [ \b(Y_{hv})  Q_{h(1-v)}^Y f(Y_{hv}) - \b(Y_{hv})f(Y_{hv})   \right ] \, \dd v  \right | \\
        & \hspace{1cm}+  \left |  \int_0^1  \E_y^Y \left [ \b(Y_{hv})   f(Y_{hv})  \right ] \, \dd v - \b(y)f(y)  \right | \\
        & \leq   (\a^*)^2 \| f \|_{\infty} h + \a^* \int_0^1 \| Q^Y_{h(1-v)} f - f \|_{\infty} \, \dd v + \int_0^1 \| Q^Y_{hv}(\b \times f) - (\beta \times f) \|_{\infty} \, \dd v,
    \end{align*}
    that converges to $0$ when $h$ goes to $0$ by the dominated convergence theorem because $f \in \DD_{\AA}^Y \subset C_0(E)$. Given this result and the convergence already proven for the second term in \eqref{A4temp}, we deduce that the second term  in \eqref{A4temp2} converges uniformly in $(y,k)$ to $(\b_k - \b(y)) f(y) = (\ck \a(y,k)-\a(y)) f(y).$ 
    
    The study for $y \notin E_k$ and $y \in E_k$ yields to  the same convergence results, so in conclusion
        $$ \sup_{(y,k) \in E \times \N} \left |  \E_{(y,k)} \left ( f(X'_h) \ind_{ \ck N_t=1}   \right ) +(\a(y)-\ck \a(y,k))f(y) -\a(y)Kf(y) \right | \underset{h \searrow 0}{\longrightarrow} 0, $$
     that is ${\sup}_{(y,k) \in \ck E} | A_4(h)| \underset{h \searrow 0}{\longrightarrow} 0$.
    
    To finish the proof, it remains to handle $A_5(h)$ using \eqref{dominationpoisson} where $\ck N_h^* \sim \PP(2\a^*h)$ 
    \begin{align*}
        \left |A_5(h) \right |& \leq \frac{\|f\|_{\infty}}{h} \P_{(y,k)} \left ( \ck N_h \geq 2 \right ) \\
        & \leq \frac{\|f\|_{\infty}}{h} \P \left ( \ck N_h^* \geq 2 \right ) \\
        & \leq \dfrac{\|f\|_{\infty}}{h} \left ( 1 - \e^{-2\a^*h}- 2\a^*h \e^{-2\a^*h} \right ) \\
        & =  \dfrac{\|f\|_{\infty}}{h} \left ( \dfrac{(2\a^* h)^2}{2} +  \underset{h \searrow 0}{o}(h^2) \right ) \\
        & = 2 \|f\|_{\infty} \, (\a^*)^2 \, h + \underset{h \searrow 0}{o}(h),
    \end{align*}
that converges uniformly to $0$ when $h$ goes to $0$.

\subsection{Proof of Corollary \ref{coro46}}
 Let 
    \begin{align*}
    G(t)=& \psi_f(t)  -  f(x) - \int_0^t    \psi_{\AA f}(s)  \, \dd s.
    \end{align*}
    This function is continuous, right-differentiable on $\R_+$ from Lemmas~\ref{marg2C0} and \ref{suptendsvers0}, and $\partial_+ G(t)/\partial t=0$. So $G$ is constant. But $G(0)=0$ because $s \geq 0 \mapsto  \psi_{\AA f}(s)$ is bounded. As a consequence, we obtain \eqref{diffpsi}. 
  Moreover $\AA f \in C_0(E)$ so by Lemma \ref{marg2C0} the function $s \geq 0 \mapsto \psi_{\AA f}(s)$ is continuous, and by  \eqref{diffpsi} we deduce that $\psi_f$ is differentiable.

\section{Proof of Theorem \ref{existunicitémesinvariante}}\label{proof 4.1}

We recall and complete notations introduced  in Appendix~\ref{sec6.7} regarding the coupling between  $X$ and  $\eta$. The coupled process is $\ck C=(X',\eta')$, where from Theorem~\ref{Marginals}, $X'$ and $\eta'$ have the same distributions as $X$ and $\eta$. 
  We denote by $T_j$ and $t_j$ the jump times of $X$ and $\eta$. Similarly we denote by $T_j'$ and $t_j'$ the jump times of $X'$ and $\eta'$. To prove Theorem~\ref{existunicitémesinvariante}, we start with the following lemma where $s_0:=\inf \{ t \geq t_1, \; \eta_t=0 \}$ is the time of the first return of $\eta$ in the state $0$ and $S_ \textnormal{\O}:=\inf \{ t \geq T_1, \; X_t= \textnormal{\O} \}$ is the time of the first return of $(X_t)_{t \geq 0}$ in the state \O.

\begin{lem} \label{prop54}
 Suppose that $0$ is an ergodic state for the simple process $\eta$, that is $\E_0(s_0)< \infty$. Then $\lim \limits_{t \rightarrow \infty} Q_t( \textnormal{\O}, A)$ exists for all $A \in \EE.$ 
     Suppose moreover that for all $n \geq 0$, $\E_n(s_0)< \infty$. Then, $\lim \limits_{t \rightarrow \infty} Q_t( x, A)$ exists for all $x \in E$, $A \in \EE$, and is independent of $x$.
\end{lem}

\begin{proof} 
Let $\ck s_0:=\inf \{ t \geq t_1', \; \ck C_t \in E\times \{0\} \}$. Using the first statement of Theorem~\ref{Marginals}, we can prove that $\P_{(\textnormal{\O},0)}(\ck s_0>t)=\P_0(s_0>t)$. Similarly, by the second statement of this theorem, 
$\P_{(\textnormal{\O},0)}(\ck S_ \textnormal{\O}>t)=\P_{\textnormal{\O}}(S_{\textnormal{\O}}>t)$ where 
$\ck S_ \textnormal{\O}:=\inf \{ t \geq T_1', \; \ck C_t \in  \{\textnormal{\O}\}\times \N \}$. We thus have
\[\P_{\textnormal{\O}}(S_{\textnormal{\O}}>t) = \P_{(\textnormal{\O},0)}(\ck S_ \textnormal{\O}>t) \leq \P_{(\textnormal{\O},0)}(\ck s_0>t)=\P_0(s_0>t),\]
where the inequality comes from the third statement of Theorem~\ref{Marginals}.

By the assumptions of Lemma~\ref{prop54}, this implies that $S_{\textnormal{\O}}< \infty \; \P_{\textnormal{\O}}-a.s.$ and that
\begin{align*}
    \E_{\textnormal{\O}}(S_{\textnormal{\O}})  = \int_0^{\infty} \P_{\textnormal{\O}}(S_{\textnormal{\O}}>t) \, \dd t 
     \leq \int_0^{\infty} \P_0(s_0>t) \, \dd t < \infty,
\end{align*}
proving that  \O \ is an ergodic state for the process $(X_t)_{t \geq 0}$. Note moreover that  $S_{\textnormal{\O}}$ has a density with respect to the Lebesgue measure, that we denote by $\mu_{\textnormal{\O}}.$  This comes from the fact that $\tau_j$ has a density for any $j$, so does $T_j$, whereby given a Lebesgue null set  $I \in \BB (\R)$, $\P_\textnormal{\O} ( S_\textnormal{\O} \in I) \leq \sum_{j=1}^{\infty} \P_\textnormal{\O} (T_j \in I) = 0$. 

We have the following equation 
\begin{align*}
    Q_t(\textnormal{\O},A) & = \P_{\textnormal{\O}}(X_t \in A, S_\textnormal{\O} > t) + \int_0^t \P_{\textnormal{\O}}(X_t \in A, S_\textnormal{\O} \in \, \dd s) \\
    & = \P_{\textnormal{\O}}(X_t \in A, S_\textnormal{\O} > t) + \int_0^t \P_{\textnormal{\O}}(X_t \in A | S_\textnormal{\O} = s) \mu_\textnormal{\O}(s) \, \, \dd s \\
    & = \P_{\textnormal{\O}}(X_t \in A, S_\textnormal{\O} > t) + \int_0^t Q_{t-s}(\textnormal{\O},A) \mu_\textnormal{\O}(s) \, \, \dd s.
    \end{align*}
This is a renewal equation and we may apply the renewal theorem given in \cite[Chapter XI]{feller1971}. To this end, denote by $\ZZ(t)=Q_t(\textnormal{\O},A)$, $\xi(t)=\P_{\textnormal{\O}}(X_t \in A, S_\textnormal{\O} > t)$ and $F \{ I \}=\P_{\textnormal{\O}}(S_\textnormal{\O} \in I).$  Remark that $\ZZ$ is bounded, $\xi$ is non-negative, bounded by $1$ and directly Riemann integrable on $\R_+$  because it is dominated by the monotone integrable function $t \mapsto  \P_{\textnormal{\O}}( S_\textnormal{\O} > t) $. 
Moreover, $0 <\E_{\textnormal{\O}}(S_\textnormal{\O} ) < \infty $ and 
since $S_\textnormal{\O}$ has a density, $F$ is not arithmetic. Then, by the renewal theorem, we obtain:
\begin{equation}\label{conv Qt}
Q_t(\textnormal{\O},A) = \ZZ(t) \underset{t \rightarrow \infty}{\longrightarrow}  \dfrac{1}{\E_{\textnormal{\O}}(S_\textnormal{\O})}  \int_0^{\infty} \xi(u) \, \, \dd u = \dfrac{1}{\E_{\textnormal{\O}}(S_\textnormal{\O})} \int_0^{\infty} \P_{\textnormal{\O}}(X_u \in A, S_\textnormal{\O} > u) \, \, \dd u 
\end{equation}
which proves the first statement of Lemma~\ref{prop54}.

Let now turn to the second part of Lemma~\ref{prop54}. 
Let $x \in E_n$. By the arguments as in the beginning of the proof, we get that $S_\textnormal{\O} < \infty, \; \; \P_{x}-a.s.$ and that $\E_{x}(S_\textnormal{\O}) \leq \E_{n}(s_0) < \infty. $ We have 
\begin{align*}
    Q_t(x,A) & = \P_x(X_t \in A) \\
    & = \P_{x}(X_t \in A, S_\textnormal{\O} > t) + \int_0^t \P_{x}(X_t \in A | S_\textnormal{\O} = s) \mu_\textnormal{\O}(s) \, \, \dd s \\
    & = \P_{x}(X_t \in A, S_\textnormal{\O} > t) + \int_0^t \P_{\textnormal{\O}}(X_{t-s} \in A) \mu_\textnormal{\O}(s) \, \, \dd s \\ 
    & = \P_{x}(X_t \in A, S_\textnormal{\O} > t) + \int_0^t Q_{t-s}(\textnormal{\O},A) \mu_\textnormal{\O}(s) \, \, \dd s.  
\end{align*}
The first term tends to $0$ as $t\to\infty$ because it is dominated by $\P_{x}( S_\textnormal{\O} > t)$ and we know that   $\P_x(S_\textnormal{\O} < \infty)=1.$
For the second term, for all $s \geq 0$, we have by \eqref{conv Qt}
 $$Q_{t-s}(\textnormal{\O},A) \ind_{[0,t]}(s) \underset{t \rightarrow \infty}{\longrightarrow} \dfrac{1}{\E_{\textnormal{\O}}(S_\textnormal{\O})} \int_0^{\infty} \P_{\textnormal{\O}}(X_u \in A, S_\textnormal{\O} > u) \, \, \dd u . $$
Moreover  $| Q_{t-s}(\textnormal{\O},A) \ind_{[0,t]}(s) \mu_\textnormal{\O}(s) | \leq  \mu_\textnormal{\O}(s)$ which is integrable. So by the dominated convergence theorem,
$$ Q_t(x,A) \underset{t \rightarrow \infty}{\longrightarrow} \dfrac{1}{\E_{\textnormal{\O}}(S_\textnormal{\O})} \int_0^{\infty} \P_{\textnormal{\O}}(X_u \in A, S_\textnormal{\O} > u) \, \, \dd u $$
which is independent of $x$.
 \end{proof}

We are now in position to prove Theorem~\ref{existunicitémesinvariante}. The conditions \eqref{eq30} or \eqref{eq31} of \cite{karlin1957} imply the assumptions made in Lemma~\ref{prop54}. We  then deduce that $\mu(A):=\lim_{t \rightarrow \infty} Q_t(x,A)$ exists for all $x \in E$ and $A \in \EE$, and is independent of $x$. It is a probability measure because for any $t \geq 0$ and $x \in E$, $Q_t(x,.)$ is a probability measure.

Let us prove that $\mu$ is an invariant measure. The previous convergence reads 
\begin{equation}\label{cv temp renewal}
\int_E f(y) Q_s(x,dy) \underset{s \rightarrow \infty}{\longrightarrow} \int_E f(y) \mu(dy).
\end{equation}
where $f=\ind_A$ with $A \in \EE$. It is not difficult to extend it to any step function and by limiting arguments to any 
$f \in \mathcal{M}_b^+(E)$. By the Markov property,  for all $t,s \geq $, $x \in E$ and $A \in \EE$,
    \begin{equation*}
        Q_{t+s}(x,A)=\int_E Q_t(y,A) \, Q_s(x,dy).
    \end{equation*}
    Letting $s$ tend to $\infty$, we obtain that the   left hand side converges to $\mu(A)$, while for the right hand side, we may apply \eqref{cv temp renewal}  to $f=Q_t(.,A) \in \mathcal{M}_b^+(E)$ to finally obtain
    $$\mu(A) =  \int_E Q_t(y,A) \mu(dy).$$
    
Finally, if $\nu$ is a probability measure on $E$, such that for any $A\in \EE$
    \begin{equation*}
        \nu(A)=\int_E Q_t(y,A) \nu(dy),
    \end{equation*}
    then as $Q_t(x,A) \leq 1$, taking $t \rightarrow  \infty$, we get by the dominated convergence theorem
    \begin{equation*}
        \nu(A)=\int_E \mu(A) \nu(dy)=\mu(A).
    \end{equation*}
    Hence $\mu$ is the unique invariant probability measure.

\bibliographystyle{chicago}
\bibliography{bibref}

\begin{thebibliography}{}

\bibitem[\protect\citeauthoryear{Athreya and Ney}{Athreya and
  Ney}{2012}]{athreya2012}
Athreya, K.~B. and P.~E. Ney (2012).
\newblock {\em Branching processes}, Volume 196.
\newblock Springer Science \& Business Media.

\bibitem[\protect\citeauthoryear{Bansaye and M{\'e}l{\'e}ard}{Bansaye and
  M{\'e}l{\'e}ard}{2015}]{bansaye2015}
Bansaye, V. and S.~M{\'e}l{\'e}ard (2015).
\newblock {\em Stochastic models for structured populations}, Volume~16.
\newblock Springer.

\bibitem[\protect\citeauthoryear{Bass}{Bass}{2011}]{Bass}
Bass, R.~F. (2011).
\newblock {\em Stochastic processes}, Volume~33.
\newblock Cambridge University Press.

\bibitem[\protect\citeauthoryear{Bourbaki}{Bourbaki}{1966}]{bourbaki1966general}
Bourbaki, N. (1966).
\newblock {\em General topology: elements of mathematics}.
\newblock Addison-Wesley.

\bibitem[\protect\citeauthoryear{{\c{C}}inlar and Kao}{{\c{C}}inlar and
  Kao}{1991}]{cinlar1991}
{\c{C}}inlar, E. and J.~S. Kao (1991).
\newblock Particle systems on flows.
\newblock {\em Applied Stochastic Models and Data Analysis\/}~{\em 7\/}(1),
  3--15.

\bibitem[\protect\citeauthoryear{Comas}{Comas}{2009}]{comas2009}
Comas, C. (2009).
\newblock Modelling forest regeneration strategies through the development of a
  spatio-temporal growth interaction model.
\newblock {\em Stochastic Environmental Research and Risk Assessment\/}~{\em
  23\/}(8), 1089--1102.

\bibitem[\protect\citeauthoryear{Davis}{Davis}{1984}]{davis1984}
Davis, M.~H. (1984).
\newblock Piecewise-deterministic markov processes: A general class of
  non-diffusion stochastic models.
\newblock {\em Journal of the Royal Statistical Society: Series B
  (Methodological)\/}~{\em 46\/}(3), 353--376.

\bibitem[\protect\citeauthoryear{Dynkin}{Dynkin}{1965}]{dynkin1965markov}
Dynkin, E.~B. (1965).
\newblock {\em Markov processes}.
\newblock Springer.

\bibitem[\protect\citeauthoryear{Fattler and Grothaus}{Fattler and
  Grothaus}{2007}]{fattler2007}
Fattler, T. and M.~Grothaus (2007).
\newblock Strong feller properties for distorted brownian motion with
  reflecting boundary condition and an application to continuous n-particle
  systems with singular interactions.
\newblock {\em Journal of Functional Analysis\/}~{\em 246\/}(2), 217--241.

\bibitem[\protect\citeauthoryear{Feller}{Feller}{1971}]{feller1971}
Feller, W. (1971).
\newblock {\em An introduction to probability theory and its applications, Vol
  2}.
\newblock John Wiley \& Sons, New York.

\bibitem[\protect\citeauthoryear{H{\"a}bel, Myllym{\"a}ki, and
  Pommerening}{H{\"a}bel et~al.}{2019}]{habel2019}
H{\"a}bel, H., M.~Myllym{\"a}ki, and A.~Pommerening (2019).
\newblock New insights on the behaviour of alternative types of
  individual-based tree models for natural forests.
\newblock {\em Ecological modelling\/}~{\em 406}, 23--32.

\bibitem[\protect\citeauthoryear{Ikeda, Nagasawa, and Watanabe}{Ikeda
  et~al.}{1968}]{ikeda1968}
Ikeda, N., M.~Nagasawa, and S.~Watanabe (1968).
\newblock Branching markov processes ii.
\newblock {\em Journal of Mathematics of Kyoto University\/}~{\em 8\/}(3),
  365--410.

\bibitem[\protect\citeauthoryear{Kallenberg}{Kallenberg}{2017}]{kallenberg}
Kallenberg, O. (2017).
\newblock {\em Random measures, theory and applications}.
\newblock Springer, Cham.

\bibitem[\protect\citeauthoryear{Karlin and McGregor}{Karlin and
  McGregor}{1957}]{karlin1957}
Karlin, S. and J.~McGregor (1957).
\newblock The classification of birth and death processes.
\newblock {\em Transactions of the American Mathematical Society\/}~{\em
  86\/}(2), 366--400.

\bibitem[\protect\citeauthoryear{Lavancier and Le~Guével}{Lavancier and
  Le~Guével}{2021}]{Lavancier_LeGuevel}
Lavancier, F. and R.~Le~Guével (2021).
\newblock Spatial birth–death–move processes: basic properties and
  estimation of their intensity functions.
\newblock {\em Journal of the Royal Statistical Society: Series B (Statistical
  Methodology)\/}~{\em 83\/}(4), 798--825.

\bibitem[\protect\citeauthoryear{L{\"o}cherbach}{L{\"o}cherbach}{2002}]{locherbach2002likelihood}
L{\"o}cherbach, E. (2002).
\newblock Likelihood ratio processes for markovian particle systems with
  killing and jumps.
\newblock {\em Statistical inference for stochastic processes\/}~{\em 5\/}(2),
  153--177.

\bibitem[\protect\citeauthoryear{Markley}{Markley}{2004}]{markley}
Markley, N.~G. (2004).
\newblock {\em Principles of differential equations}.
\newblock John Wiley \& Sons.

\bibitem[\protect\citeauthoryear{Masuda and Holme}{Masuda and
  Holme}{2017}]{masuda2017}
Masuda, N. and P.~Holme (2017).
\newblock {\em Temporal network epidemiology}.
\newblock Springer.

\bibitem[\protect\citeauthoryear{M{\o}ller}{M{\o}ller}{1989}]{moller1989}
M{\o}ller, J. (1989).
\newblock On the rate of convergence of spatial birth-and-death processes.
\newblock {\em Annals of the Institute of Statistical Mathematics\/}~{\em
  41\/}(3), 565--581.

\bibitem[\protect\citeauthoryear{M{\o}ller and Waagepetersen}{M{\o}ller and
  Waagepetersen}{2004}]{moeller:waagepetersen:03}
M{\o}ller, J. and R.~P. Waagepetersen (2004).
\newblock {\em Statistical inference and simulation for spatial point
  processes}.
\newblock Chapman and Hall/CRC, Boca Raton.

\bibitem[\protect\citeauthoryear{{\O}ksendal}{{\O}ksendal}{2013}]{oksendal}
{\O}ksendal, B. (2013).
\newblock {\em Stochastic differential equations: an introduction with
  applications}.
\newblock Springer Science \& Business Media.

\bibitem[\protect\citeauthoryear{Pommerening and Grabarnik}{Pommerening and
  Grabarnik}{2019}]{pommerening2019}
Pommerening, A. and P.~Grabarnik (2019).
\newblock {\em Individual-based methods in forest ecology and management}.
\newblock Springer.

\bibitem[\protect\citeauthoryear{Preston}{Preston}{1975}]{preston}
Preston, C. (1975).
\newblock Spatial birth and death processes.
\newblock {\em Advances in applied probability\/}~{\em 7\/}(3), 371--391.

\bibitem[\protect\citeauthoryear{Renshaw, Comas, and Mateu}{Renshaw
  et~al.}{2009}]{renshaw2009}
Renshaw, E., C.~Comas, and J.~Mateu (2009).
\newblock Analysis of forest thinning strategies through the development of
  space--time growth--interaction simulation models.
\newblock {\em Stochastic Environmental Research and Risk Assessment\/}~{\em
  23\/}(3), 275--288.

\bibitem[\protect\citeauthoryear{Renshaw and S{\"a}rkk{\"a}}{Renshaw and
  S{\"a}rkk{\"a}}{2001}]{renshaw2001}
Renshaw, E. and A.~S{\"a}rkk{\"a} (2001).
\newblock Gibbs point processes for studying the development of
  spatial-temporal stochastic processes.
\newblock {\em Computational statistics \& data analysis\/}~{\em 36\/}(1),
  85--105.

\bibitem[\protect\citeauthoryear{Revuz and Yor}{Revuz and
  Yor}{1991}]{RevuzYor1991}
Revuz, D. and M.~Yor (1991).
\newblock {\em Continuous martingales and Brownian motion}.
\newblock Springer-Verlag, Berlin.

\bibitem[\protect\citeauthoryear{Schilling and Partzsch}{Schilling and
  Partzsch}{2012}]{schilling2012}
Schilling, R.~L. and L.~Partzsch (2012).
\newblock {\em Brownian motion : an introduction to stochastic processes}.
\newblock De Gruyter.

\bibitem[\protect\citeauthoryear{Schuhmacher and Xia}{Schuhmacher and
  Xia}{2008}]{schuhmacher2008}
Schuhmacher, D. and A.~Xia (2008).
\newblock A new metric between distributions of point processes.
\newblock {\em Advances in applied probability\/}~{\em 40\/}(3), 651--672.

\bibitem[\protect\citeauthoryear{Skorokhod}{Skorokhod}{1964}]{skorokhod1964}
Skorokhod, A.~V. (1964).
\newblock Branching diffusion processes.
\newblock {\em Theory of Probability \& Its Applications\/}~{\em 9\/}(3),
  445--449.

\end{thebibliography}

\end{document}